\documentclass[12pt]{amsart}
\usepackage{tikz-cd}
\usepackage{enumitem}
\setlist[description]{leftmargin=\parindent,labelindent=\parindent}
\usepackage{amssymb}
\usepackage{amsmath,amscd}
\usepackage{mathrsfs}
\usepackage{url}
\usepackage{mathtools}
\usepackage{cite}
\usepackage{fullpage}
\usepackage[hidelinks]{hyperref}
\usepackage{verbatim}
\usepackage{comment}
\usepackage{fancyvrb}
\usepackage{fvextra}
\usepackage{caption}
\usepackage{tikz}
\usepackage{tkz-graph}
\usepackage{diagbox}
\usetikzlibrary{matrix}

\usepackage[enableskew]{youngtab}
\usepackage{ytableau}
 \usetikzlibrary{shapes}
\usetikzlibrary{arrows,decorations.markings, calc}
\usepackage{tikz-cd}
\usepackage{stmaryrd}
\usepackage{mathtools}
\usepackage{tikz}

\tikzset{every picture/.style={baseline=-.65ex}}
\tikzset{ext/.style={circle, draw,inner sep=1pt},int/.style={circle,draw,fill,inner sep=1pt},nil/.style={inner sep=1pt}}
\tikzset{exte/.style={circle, draw,inner sep=3pt},inte/.style={circle,draw,fill,inner sep=3pt}}
\tikzset{diagram/.style={matrix of math nodes, row sep=3em, column sep=2.5em, text height=1.5ex, text depth=0.25ex}}
\tikzset{diagram2/.style={matrix of math nodes, row sep=0.5em, column sep=0.5em, text height=1.5ex, text depth=0.25ex}}
\tikzset{every loop/.style={draw}}

\pgfarrowsdeclare{bad to}{bad to}
{
  \pgfarrowsleftextend{-2\pgflinewidth}
  \pgfarrowsrightextend{\pgflinewidth}
}
{
  \pgfsetlinewidth{0.8\pgflinewidth}
  \pgfsetdash{}{0pt}
  \pgfsetroundcap
  \pgfsetroundjoin
  \pgfpathmoveto{\pgfpoint{-3\pgflinewidth}{4\pgflinewidth}}
  \pgfpathcurveto
  {\pgfpoint{-2.75\pgflinewidth}{2.5\pgflinewidth}}
  {\pgfpoint{0pt}{0.25\pgflinewidth}}
  {\pgfpoint{0.75\pgflinewidth}{0pt}}
  \pgfpathcurveto
  {\pgfpoint{0pt}{-0.25\pgflinewidth}}
  {\pgfpoint{-2.75\pgflinewidth}{-2.5\pgflinewidth}}
  {\pgfpoint{-3\pgflinewidth}{-4\pgflinewidth}}
  \pgfusepathqstroke
}

\newtheorem{thm}{Theorem}[section]
\newtheorem{prop}[thm]{Proposition}
\newtheorem{lem}[thm]{Lemma}
\newtheorem{cor}[thm]{Corollary}
\newtheorem{conj}[thm]{Conjecture}

\theoremstyle{definition}

\newtheorem{example}[thm]{Example}
\newtheorem{rem}[thm]{Remark}

\newtheorem*{thm*}{Theorem}

\numberwithin{equation}{section}

\newcommand{\s}{\mathsf{S}}

\newcommand{\X}{\mathcal{X}}

\newcommand{\Q}{\mathcal{Q}}
\newcommand{\gr}{\mathrm{gr}}
\newcommand{\T}{\mathcal{T}}

\newcommand{\zz}{\mathbb{Z}}
\newcommand{\cc}{\mathbb{C}}
\newcommand{\qq}{\mathbb{Q}}

\newcommand{\C}{\mathcal{C}}

\renewcommand{\ss}{\mathbb{S}}

\newcommand{\M}{\mathcal{M}}

\newcommand{\A}{\mathcal{A}}

\newcommand{\F}{\mathsf{F}}
\newcommand{\vv}{\mathbb{V}}
\renewcommand{\P}{\mathcal{P}}

\renewcommand{\tilde}{\widetilde}

\DeclareMathOperator{\Aut}{Aut}

\DeclareMathOperator{\SL}{SL}

\DeclareMathOperator{\im}{Im}
\DeclareMathOperator{\coker}{coker}

\DeclareMathOperator{\Sym}{Sym}
\DeclareMathOperator{\id}{id}

\DeclareMathOperator{\Sp}{Sp}

\newcommand{\FA}{\mathbf{FA}}

\newcommand{\FS}{\mathbf{FS}}
\newcommand{\FB}{\mathbf{FB}}
\newcommand{\Vect}{\mathcal{V}\mathit{ect}}

\newcommand{\hboxtimes}{\widehat\boxtimes}

\usepackage{wrapfig}

\newcommand{\Mb}{\overline{\M}}

\newcommand{\Ind}{\mathrm{Ind}}
\newcommand{\GK}{\mathsf{GK}}

\usepackage{color, colortbl}
\usepackage{array}
\usepackage{varwidth} %

\definecolor{Gray}{gray}{0.9}

\newcolumntype{g}{>{\columncolor{Gray}}c}
\newcolumntype{M}{V{5cm}} %

\newcommand{\Com}{\mathsf{Com}}

\author{Samir Canning}
\author{Hannah Larson}
\author{Sam Payne}
\author{Thomas Willwacher}
\thanks{
S.C. was supported by a Hermann-Weyl-Instructorship from the Forschungsinstitut für Mathematik at ETH Z\"urich and the SNSF Ambizione grant PZ00P2\_223473.
This research was partially conducted during the period H.L served as a Clay Research Fellow. 
S.P. was supported in part by NSF grant DMS--2542134 and conducted parts of this research during a visit to the Institute for Advanced Study supported by the Charles Simonyi Endowment. T.W. has been supported by the NCCR SwissMAP, funded by the Swiss National Science Foundation.}

\title{$\FA$-modules of holomorphic forms on $\Mb_{g,n}$}

\begin{document}

\begin{abstract}
For fixed genus $g$ and varying finite marking set $A$, the gluing and forgetful maps give the spaces of holomorphic forms on $\Mb_{g,A}$ the structure of an $\FA$-module, i.e., a functor from the category of finite sets to vector spaces.
We describe the spaces of holomorphic forms $H^{k,0}(\Mb_{g,A})$ for $k \leq 18$ and all $g$ as $\FA$-modules, and prove that they are simple whenever they are nonzero. Conditional upon the conjectured vanishing of $H^{19,0}(\Mb_{3,15})$ and $H^{20,0}(\Mb_{3,16})$, this description extends to $k = 19$ and $k = 20$, respectively.
\end{abstract}

\maketitle

\setcounter{tocdepth}{1}
\tableofcontents

\section{Introduction}

In this paper, we study the holomorphic $k$-forms on  moduli spaces of stable curves $\Mb_{g,n}$ for $k \leq 20$ and all $g$ and $n$. Our main new results are a complete description of $H^{17,0}(\Mb_{g,n})$ for all $g$ and $n$ and, conditional upon a widely believed vanishing conjecture in genus $3$, a similar description of $H^{19,0}(\Mb_{g,n})$ for all $g$ and $n$.
One of the key insights is that these new results, as well as earlier ones, fit naturally into the framework of \emph{$\FA$-modules}, functors from the category $\FA$ of finite sets with all maps to the category of vector spaces. 

For fixed $g$ and $k$, the system of vector spaces $H^{k,0}(\Mb_{g,A})$ as the marking set $A$ varies naturally form an $\FA$-module, which we denote $H^{k,0}(\Mb_{g,*})$.  %
On objects, the functor $H^{k,0}(\Mb_{g,*})$ sends the finite set $A$ to the vector space $H^{k,0}(\Mb_{g,A})$.  To define the functor on morphisms, we associate to each map of finite sets $f\colon  A \to B$ a composition of gluing and forgetful maps
\[\varphi_f \colon \Mb_{g, B} \times \prod_{b \in B \text{ with }|f^{-1}(b)| \geq 2} \Mb_{0,b' \cup f^{-1}(b)} \to \Mb_{g,A}\]
which forgets $b$ if $f^{-1}(b) = \varnothing$, sends $b$ to $f^{-1}(b)$ if $|f^{-1}(b)| = 1$, and glues $b$ to $b'$ if $|f^{-1}(b)| \geq 2$ (see Example \ref{ex1}).
Pullback then defines our desired morphism
\[H^{k,0}(\Mb_{g,A}) \xrightarrow{\varphi_f^*} H^{k,0}\bigg( \Mb_{g, B} \times \prod_{b \in B \text{ with }|f^{-1}(b)| \geq 2} \Mb_{0,b' \cup f^{-1}(b)}\bigg) \cong H^{k,0}(\Mb_{g,B}).\] 
For $g = 0$ and $1$, we adopt the convention that $\Mb_{1, \varnothing}$ is a point, as is $\Mb_{0,A}$ for $|A| \leq 2$.

Working with $\FA$-modules streamlines both the statements and proofs of our results. The category of $\FA$-modules is not semisimple, but the simple $\FA$-modules are classified. See \cite[Theorem~4.1]{Rains09} and \cite[Theorem~5.5]{WiltshireGordon}. Given a partition $m = \lambda_1 + \ldots + \lambda_\ell$ with $\lambda_1 \geq \cdots \geq \lambda_\ell$, let $V_\lambda$ be the associated irreducible representation of $\mathbb{S}_m$. For each such $\lambda$,
there is a naturally associated $\FA$-module $C_\lambda$, whose definition we review in Section \ref{sec:FAsimple}. For each $n$, its associated $\mathbb{S}_n$-representations are
\[C_\lambda(n) = \begin{cases} 0 & \text{if $n < m$} \\
\Ind_{\mathbb{S}_{m}  \times \mathbb{S}_{n-m}}^{\mathbb{S}_n}(V_\lambda \boxtimes \mathbf{1}) & \text{if $n \geq m$.}
\end{cases}\]
If $\lambda_1 \geq 2$, then $C_\lambda$ is simple. If $\lambda_1 = 1$, then $C_\lambda$ is not simple, but the $\FA$-module
 \[
        \tilde C_{1^m} := \coker\left( C_{1^{m+1}} \to C_{1^{m}}\right)
        \]
is simple. The associated $\mathbb{S}_n$-representations satisfy the following formula
\[\tilde{C}_{1^{m}}(n) = \begin{cases} 0 & \text{if $n < m$} \\
V_{n-m+1,1^{m-1}} & \text{if $n \geq m$.} \end{cases}\]
Here, $V_{n-m+1,1^{m-1}}$ is the irreducible representation associated to the partition whose Young diagram is a hook shape, with $m$ boxes in the first column, as shown.
\begin{figure}[h!]
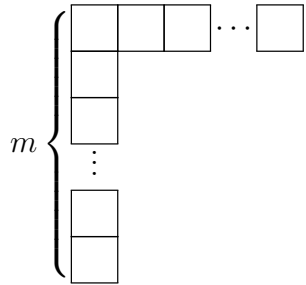

\ytableausetup{mathmode, boxframe=normal, boxsize=6mm}
\[
m \begin{cases} \begin{ytableau}
\phantom{.} &  \phantom{.} & \phantom{.}  & \none[\cdots]  & \\
\\
\\
\none[\vdots] \\
\\
\\
\end{ytableau}
\end{cases}
\]
\caption{The hook shape  partition $n-m+1, 1^{m-1}$.}
\label{fig:hook}
\end{figure}

\begin{thm} \label{thm:upto17}
Suppose $g \geq 1$ and $k \leq 18$.
We have the following equalities of $\FA$-modules
\[
H^{k,0}(\Mb_{g,*}) \cong 
\begin{cases}
\tilde{C}_{1^{k}} & \text{if $k = 11, 15, 17$ and $g = 1$} \\
C_{2^7} & \text{if $k = 17$ and $g = 2$} \\
0 & \text{otherwise.}
\end{cases}\]
\end{thm}

Several cases of the above theorem were previously known. The cases $k \leq 12$ are established in \cite{CanningLarsonPayne}, the cases $k = 13, 14$ in \cite{CLP-STE}, the case $k = 15$ in \cite{CLPW2}, and the cases $k = 16, 18$ in \cite{Fontanari}.
The case $k = 17$ and the unifying description using $\FA$-modules are new.
We will also establish that the pullbacks along gluing maps not appearing in the definition of the $\FA$-module structure (i.e. those besides pullback to a genus $g$ component glued to genus $0$ components) all vanish (Section \ref{sec:pull}). In this way, the $\FA$-module structures given above completely capture the modular cooperad of holomorphic forms in degrees $\leq 18$. 

Conditional upon additional vanishing statements in genus $3$, we extend Theorem~\ref{thm:upto17} to describe the $\FA$-modules of holomorphic $k$-forms for $k \leq 20$. In \cite{BergstromFaber}, Bergstr\"om and Faber use point counting to give a precise conjectural description of $H^\bullet(\Mb_{3,n})$ for $n \leq 14$, and prove that their predictions would follow from certain widely believed conjectures related to the Langlands program.
In private communication, they confirmed that further computations support the following prediction.

\begin{conj} \label{bfc}
The spaces of holomorphic forms $H^{19,0}(\Mb_{3,15})$ and $H^{20,0}(\Mb_{3,16})$ vanish.
\end{conj}

\begin{thm} \label{thm:19}
If $H^{19,0}(\Mb_{3,15}) = 0$ then %
 \[
H^{19,0}(\Mb_{g,*}) \cong
\begin{cases}
\tilde{C}_{1^{19}} & \text{if $g = 1$} \\
C_{2^51^6} & \text{if $g = 2$} \\
0 & \text{otherwise.}
\end{cases}\]
Furthermore, if $H^{20,0}(\Mb_{3,16}) = 0$ then $H^{20,0}(\Mb_{g,*}) = 0$ for all $g$.
 \end{thm} 

\noindent The statements in genus $1$ and $2$ are proven unconditionally; in  Section~\ref{sec:gleq2}, we present a more general description of $H^{k,0}(\Mb_{g,*})$ for $g \leq 2$ and all $k$. The conditional statement for $H^{20,0}(\Mb_{g,*})$ is due to Fontanari \cite{Fontanari}. %

\begin{rem}
New phenomena appear in the $\FA$-modules of holomorphic forms of higher degree. We note, in particular, that $H^{21,0}(\Mb_{3,15}) \neq 0$. This follows from work of Cléry, Faber, and van der Geer \cite[Theorem 13.1]{CleryFabervanderGeer}, by an application of the Leray spectral sequence.
\end{rem}

As an application, we use the $\ss_{14}$-action on the holomorphic $17$-forms on $\Mb_{2,14}$ to construct odd cohomology classes on other moduli spaces of stable curves. 
\begin{thm}\label{thm:odd}
    For $g\geq 16$ and $n\geq 0$, $H^{45+2(g-16)}(\Mb_{g,n})\neq 0$. In particular, the point count $\#\Mb_{g}(\mathbb{F}_q)$ is not polynomial in $q$ if $g \geq 16$. 
\end{thm}

\noindent Pikaart was the first to show that $\Mb_g$ has odd cohomology, and hence $\#\Mb_{g}(\mathbb{F}_q)$ is not polynomial, for $g$ sufficiently large \cite{Pikaart}. His construction gives a large bound that was not made explicit. The bound was subsequently reduced to $g \geq 67$ in \cite[Theorem~1.5]{CanningLarsonPayne}. 

\medskip

\noindent {\bf Overview.} Our proofs of Theorems \ref{thm:upto17} and \ref{thm:19} use a reframing of the inductive method of Arbarello and Cornalba for computing cohomology groups of moduli spaces of curves in the language of modular operads. Roughly speaking, given a collection of subspaces of  $H^{k,0}(\Mb_{g,n})$ for all $g, n$, (more precisely, a modular subcooperad of $H^{k,0}(\Mb_{*,*})$) one can confirm whether the subspace is all of $H^{k,0}(\Mb_{g,n})$ by an induction on $g$ and $n$, starting from a relatively small set of base cases. The inductive step involves the vanishing statements for cohomology groups of the Feynman transforms in the degrees corresponding to graphs with at most 1 edge. Thus, a relatively simple cohomology vanishing argument for a graph complex replaces the more hands-on steps applied to compute $H^2$ in \cite[Section~4]{ArbarelloCornalba}.

In Section~\ref{sec:FAclassification}, we recall the classification of simple $\FA$-modules and discuss the appearance of $\FA$-modules in the cohomology of moduli spaces of curves. In Section~\ref{sec:gleq2}, we compute $H^{k,0}(\Mb_{1,n})$ and $H^{k,0}(\Mb_{2,n})$ in terms of simple $\FA$-modules for all $k$ and $n$ and compute the tautological pullbacks that are required to understand the differentials in the Feynman transform. Known results on the unirationality of $\Mb_{g,n}$ for small $g$ and $n$ support the hypothesis that $H^{k,0}(\Mb_{g,n})$ vanishes for $k \leq 20$ and $g \leq 3$.  Indeed, such results provide all of the base cases needed to run our inductive argument for $k \leq 18$. The expected vanishing of $H^{19,0}(\Mb_{3,15})$ and $H^{20,0}(\Mb_{3,16})$, i.e. the hypotheses in Theorem~\ref{thm:19}, are the remaining base cases needed to start the inductive argument for $k = 19$ and $20$, respectively. However, new ideas are needed to prove they vanish, since $\Mb_{3,15}$ and $\Mb_{3,16}$ are of general type \cite{dePreter25}.

In Section~\ref{sec:AC}, we develop the framework for our inductive argument in terms of cohomology vanishing for Feynman transforms. In Section~\ref{sec:1719}, we apply this inductive argument to prove Theorems~\ref{thm:upto17} and \ref{thm:19}, assuming a cohomology vanishing result (Proposition~\ref{prop:lowdegree}). %

In Section~\ref{sec:FA}, we study graph complexes associated to simple $\FA$-modules and prove results about their cohomology in low degrees and genera. In particular, we prove Proposition~\ref{prop:lowdegree}, which completes the proof of Theorems~\ref{thm:upto17} and \ref{thm:19}.   As further applications, we give formulas for 
 Euler characteristics of graph complexes associated to simple $\FA$-modules and prove results about their cohomology in low genera. 
 In Section \ref{sec:app}, we apply these results, together with
 Theorem \ref{thm:upto17}, to give a generating function for the
 Hodge weight $(17,0)$ compactly supported Euler characteristics of $\M_{g,n}$ (Theorem \ref{thm:ec17}) and compute the Hodge weight $(17, 0)$ cohomology of $\M_g$ for $g \leq 13$ (Corollary \ref{lowg17}). 
 
 We conclude in Section~\ref{sec:odd}, with the proof of Theorem~\ref{thm:odd}. %

\subsection*{Acknowledgments} We are most grateful to Jonas Bergstr\"om and Carel Faber for valuable discussions related to this work.

\section{Simple \texorpdfstring{$\FA$}{FA}-modules} \label{sec:FAclassification}

Let $\FA$ be the category of finite sets with all morphisms. %
An $\FA$-module is a functor 
\[
M \colon \FA \to \Vect 
\]
from $\FA$ into the category $\Vect$ of vector spaces. Throughout, all of the $\FA$-modules that we consider take values in vector spaces over a field of characteristic zero, and we use values in $\cc$-vector spaces except where explicitly stated otherwise.

Here we explain how to construct $\FA$-modules from the cohomology of moduli spaces of stable curves and review the classification of simple $\FA$-modules.  %
For the operadically-minded reader, we remark that an $\FA$-module is the same data as a right operadic comodule over the (counital) commutative cooperad $\Com^c$, with $\Com^c(r)=\cc$ for $r\geq 0$.

\subsection{\texorpdfstring{$\FA$}{FA}-modules from moduli spaces of curves}
Here we explain how to use the gluing and forgetful maps to build $\FA$-modules from the cohomology of $\Mb_{g,A}$ for fixed $g$ as the marking set $A$ varies.
We construct the $\FA$-modules associated to the holomorphic part of cohomology, but similar ideas apply to any piece of the Hodge decomposition.

 Fix some $g,k$. 
 We define the functor $M = H^{k,0}(\Mb_{g,*})$ on objects by
\[
M(A) = H^{k,0}(\Mb_{g,A}).
\]
To define $M$ on morphisms we use pullbacks along the gluing and forgetful maps. %
Namely, given $f\colon A \to B$, we define the composition of gluing and forgetful maps
\[\varphi_f \colon \Mb_{g, B} \times \prod_{b \in B \text{ with }|f^{-1}(b)| \geq 2} \Mb_{0,b' \cup f^{-1}(b)} \to \Mb_{g,A}\]
which forgets $b$ if $f^{-1}(b) = \varnothing$, sends $b$ to $f^{-1}(b)$ if $|f^{-1}(b)| = 1$, and glues $b$ to $b'$ if $|f^{-1}(b)| \geq 2$.
For $g = 0$ and $1$, we adopt the convention that $\Mb_{1, \varnothing}$ is a point, as is $\Mb_{0,A}$ for $|A| \leq 2$.
If $f$ is an injection, then $\varphi_f$ is just a forgetful map. If $f$ is a surjection, then $\varphi_f$ is just a gluing map. For an example which is neither an injection or surjection, we have the following.

\begin{example} \label{ex1}
Define $f\colon A = \{a_1, a_2, a_3, a_4, a_5, a_6, a_7\} \to B = \{b_1, b_2, b_3, b_3, b_5, b_6\}$ by
\begin{align*} 
a_1 &\mapsto b_6, &
a_2 &\mapsto b_1, & a_3 &\mapsto b_1, & a_4 &\mapsto b_2, & a_5 &\mapsto b_2, & a_6 &\mapsto b_2, & a_7 &\mapsto b_5.
\end{align*}
The map of stable graphs induced by $\varphi_f$ is pictured below:
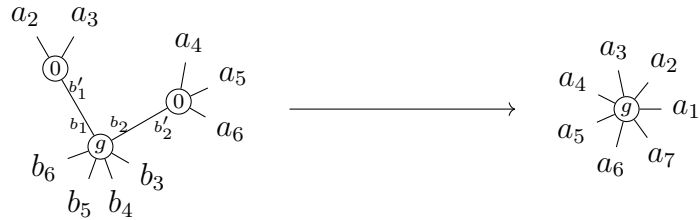
\[
\begin{tikzpicture}[shift={(0pt, -.5 cm)}]
\node[ext] (c) at (0,0) {$\scriptstyle g$};
\node[ext] (l) at (120:1.2) {$\scriptstyle 0$};
\node[ext] (r) at (30:1.2) {$\scriptstyle 0$};
\node at ($(l)+(-40:.4)$ ) {$\scriptscriptstyle b_1'$};
\node at ($(c)+(135:.4)$ ) {$\scriptscriptstyle b_1$};
\node at ($(r)+(-120:.4)$ ) {$\scriptscriptstyle b_2'$};
\node at ($(c)+(50:.4)$ ) {$\scriptscriptstyle b_2$};
\node (a2) at ($(l)+(120:.8)$) {$a_2$};
\node (a3) at ($(l)+(60:.8)$) {$a_3$};
\node (b6) at ($(c)+(200:.8)$) {$b_6$};
\node (b5) at ($(c)+(-110:.8)$) {$b_5$};
\node (b4) at ($(c)+(-70:.8)$) {$b_4$};
\node (b3) at ($(c)+(-30:.8)$) {$b_3$};
\node (a4) at ($(r)+(80:.8)$) {$a_4$};
\node (a5) at ($(r)+(25:.8)$) {$a_5$};
\node (a6) at ($(r)+(-30:.8)$) {$a_6$};
\draw (c) edge (l) edge (r) edge (b3) edge (b4) edge (b5) edge (b6)
(l) edge (a2) edge (a3) (r) edge (a4) edge (a5) edge (a6);
\end{tikzpicture}
\quad
\begin{tikzpicture}
    \draw[->] (0,0) -- (3,0);
\end{tikzpicture}
\quad
\begin{tikzpicture}
    \node[ext] (c) at (0,0) {$\scriptstyle g$};
    \node (a1) at ($(c)+(0:.8)$) {$a_1$};
    \node (a2) at ($(c)+(51:.8)$) {$a_2$};
    \node (a3) at ($(c)+(102:.8)$) {$a_3$};
    \node (a4) at ($(c)+(153:.8)$) {$a_4$};
    \node (a5) at ($(c)+(204:.8)$) {$a_5$};
    \node (a6) at ($(c)+(255:.8)$) {$a_6$};
    \node (a7) at ($(c)+(306:.8)$) {$a_7$};
    \draw (c) edge (a1) edge (a2) edge (a3) edge (a4) edge (a5) edge (a6) edge (a7);
\end{tikzpicture}
\]
The map $\varphi_f$ glues $b_1 \sim b_1'$ and $b_2 \sim b_2'$, forgets $\{b_3, b_4\}$, and sends $b_5 \mapsto a_7$ and $b_6 \mapsto a_1$.
\end{example}

We use pullback along $\varphi_f$ to define $M(f)\colon M(A) \to M(B)$ as 
\[H^{k,0}(\Mb_{g,A}) \xrightarrow{\varphi_f^*} H^{k,0}\bigg( \Mb_{g, B} \times \prod_{b \in B \text{ with }|f^{-1}(b)| \geq 2} \Mb_{0,b' \cup f^{-1}(b)}\bigg) \cong H^{k,0}(\Mb_{g,B}).\]
One readily checks that this rule respects compositions of morphisms. Indeed, suppose that $f\colon A \to B$ and $f'\colon B \to C$ are morphisms of finite sets. 
Then we have the following commutative diagram where in the top row the primes on products indicate that they are taken over elements whose preimage under $f'$ or $f$ has size at least $2$
\begin{center}
\begin{tikzcd}
\Mb_{g, C} \times \prod'_{c \in C} \Mb_{0,c' \cup f'^{-1}(c)} \times \prod'_{b \in B} \Mb_{0, b' \cup f^{-1}(c)} \arrow{r}{\varphi_{f'} \times \mathrm{id}} \arrow{d}
& \Mb_{g,B} \times \prod'_{b \in B} \Mb_{0, b' \cup f^{-1}(b)} \arrow{d}{\varphi_f} \\
\Mb_{g, C} \times \prod_{c \in C} \Mb_{0,c' \cup (f' \circ f)^{-1}(c)} \arrow{r}[swap]{\varphi_{f' \circ f}} & \Mb_{g,A}.
\end{tikzcd}
\end{center} 
Above, the left downward arrow glues $b'$ to $b$ if it exists in the marking set of one of the factors of the source and $f'(b)$ otherwise. This map corresponds to contracting edges between genus $0$ components and so induces the identity on $H^{k,0}$. Commutativity of this diagram implies the statement that $M(f') \circ M(f) = M(f' \circ f)$. 

\begin{rem} \label{rem:operadic}
From the operadic viewpoint the above derivation simplifies conceptually: It is well known that the spaces $\Mb_{g,n}$ assemble into a modular operad $\Mb$, and hence the cohomology $H^\bullet(\Mb)$ forms a modular cooperad. The genus zero part of any modular operad is a cyclic operad, and any fixed genus piece of the modular operad is an operadic right module for this cyclic operad. In particular, the collection of spaces $\Mb_{g,*}$ assembles to form an operadic right $\Mb_{0,*}$-module. Dually, the cohomology $H^\bullet(\Mb_{g,*})$ is naturally an operadic right comodule over the cooperad $H^\bullet(\Mb_{0,*})$ (the hypercommutative cooperad). The degree 0 part of the latter is the commutative cooperad, so that by corestriction we obtain the desired commutative comodule structure on $H^\bullet(\Mb_{g,*})$.

The only technical caveat here is that in order for a right operadic commutative comodule structure to be the same as an $\FA$-module structure we need to use the unital version of the commutative operad, containing an operation in arity zero. Hence in the above argument we need to also allow unstable operations in modular operads, in particular those of type $(g,n)=(0,1)$, encoding the operation of forgetting markings.
This can be done without problem, but we leave the corresponding extension of the original definition of modular operad by Getzler and Kapranov to the reader.
\end{rem}

\subsection{Classification of simple \texorpdfstring{$\FA$}{FA}-modules} \label{sec:FAsimple}
We now define certain $\FA$-modules associated to representations of symmetric groups.
To do so, we also need the categories $\FS$ of finite sets with the surjective maps as morphisms and $\FB$ of finite sets with the bijective maps as morphisms.
We define $\FS$-modules (resp. $\FB$-modules) as covariant functors $\FS\to \Vect$ (resp. $\FB\to \Vect$). In particular, $\FB$-modules are also known as symmetric sequences. An $\FB$-module is the same data as a collection $\{\rho_n\}_{n\geq 0}$ of representations of the symmetric groups $\mathbb{S}_n$. We will call $n$ the arity.

Now fix some $m\geq 1$ and let $\rho$ be a representation of the symmetric group $\mathbb{S}_m$. We may consider $\rho$ as an $\FB$-module by setting $\rho_m=\rho$ and all other $\rho_n=0$ above. We may also extend the $\FB$-module $\rho$ to an $\FS$-module also denoted $\rho$ by letting all non-bijective maps act as the zero morphism. Finally we may define an $\FA$-module
\[
C_\rho := \Ind_{\FS}^{\FA} \rho,
\]
where $\Ind_{\FS}^{\FA}$ is the induction functor from $\FS$-modules to $\FA$-modules, which is defined as the left adjoint to the natural forgetful functor $\Vect^{\FA}\to \Vect^\FS$. In category theoretic terms, $\Ind_{\FS}^{\FA}:=\mathrm{Lan}_\iota$ is the left Kan extension along the inclusion $\iota \colon \FS\to \FA$. By abstract categorical results $\mathrm{Lan}_\iota$ exists and there are explicit formulas since $\Vect$ is cocomplete and copowered over sets, see \cite[Section 4.2]{Kelly}. 
Concretely, the representation of the symmetric groups $\mathbb{S}_n$ associated to $C_\rho$ are
\[
C_\rho(n) \cong 
\begin{cases}
    0 & \text{for $n<m$} \\
    \rho & \text{for $n=m$} \\
    \Ind_{\mathbb{S}_m\times \mathbb{S}_{n-m}}^{\mathbb{S}_n} \rho & \text{for $n>m$.}
\end{cases}
\]
To determine the morphisms, it suffices to describe them on a generating set of morphisms in $\FA$. The morphisms for bijections are already described above, so we describe the maps associated to inclusions and surjections.
One can write 
\[
\Ind_{\mathbb{S}_m \times \mathbb{S}_{n-m}}^{\mathbb{S}_n} \rho \cong \bigoplus_{A \subset \{1, \ldots, n\}} \rho,
\] 
where the sum runs over subsets $A$ of size $m$. 
An inclusion $\{1, \ldots, n\} \to \{1, \ldots, n'\}$ induces a map on subsets and thus a corresponding map $C_\rho(n) \to C_{\rho}(n')$.
Any surjection is a composition of bijections and morphisms of the form $B_1 \sqcup B_2 \to B_1 \cup p$ sending $B_1$ identically to $B_1$ and collapsing $B_2$ to $p$. 
The associated map 
\[C_\rho(B_1 \sqcup B_2) = \bigoplus_{A \subset B_1 \sqcup B_2} \rho \to C_\rho(B_1 \cup p) = \bigoplus_{A' \subset B_1 \cup p} \rho\]
is defined on each of the summands of the source as follows. If $|A \cap A_2| \geq 2$, then the $A$ summand is sent to zero. If $|A \cap A_2| = 1$, then the $A$ summand is sent to the $A'$ summand obtained by replacing the unique element of $A \cap A_2$ with $p$. If $|A \cap A_2| = 0$, then $A$ is sent to the $A' = A$ summand.

If $\lambda$ is a partition of $m$ and $\rho_\lambda$ the corresponding irreducible $\mathbb{S}_m$-representation, we write 
$$
C_{\lambda} := C_{\rho_\lambda}.
$$

The category of $\FA$-modules is abelian but not semisimple. However, the simple $\FA$-modules are classified, as follows. Note that $C_{1^m}(m+1)$ contains a $1$-dimensional subspace on which $\ss_{m+1}$ acts by sign, giving rise to a nonzero morphism $C_{1^{m+1}} \to C_{1^m}$. %
We define
\[
        \tilde C_{1^m} := \coker\left( C_{1^{m+1}} \to C_{1^{m}}\right)
\]
The simple $\FA$-modules are classified as follows; see  \cite[Theorem~4.1]{Rains09} and \cite[Theorem~5.5]{WiltshireGordon}.

\begin{thm*}%
\label{thm:WiltshireGordon}
The $\FA$-module $\tilde C_{1^m}$ is simple for $m \geq 0$, as is $C_\lambda$ for  $\lambda \neq 1^m$. Moreover, any simple $\FA$-module is isomorphic to one of these.
\end{thm*}

\noindent These simple $\FA$-modules are pairwise non-isomorphic, and each is determined by its value (as a symmetric group representation) in the smallest arity where it is nonzero.

The nonzero morphisms $C_{1^{m+1}} \to C_{1^m}$ give rise to an exact sequence of $\FA$-modules 
\begin{equation}\label{equ:C long exact}
\dots  \to C_{1^m} \to C_{1^{m-1}} \to \cdots \to C_{1^2} \to C_1 \to C_{1^0} \to 0.
\end{equation}
Truncating the sequence \eqref{equ:C long exact} gives a finite resolution of $\tilde C_{1^m}$:%
\begin{equation}
\label{equ:tilde C fin res}
0 \to \tilde C_{1^m} \to C_{1^{m-1}} \to \cdots \to C_{1^2} \to C_1 \to C_{1^0} \to 0.
\end{equation}

We conclude this section with a few lemmas on morphisms of $\FA$-modules.

\begin{lem}\label{lem:FA inj}
    Let $M=M_1\oplus\cdots \oplus M_K$ be a sum of simple $\FA$-modules such that for some fixed $r_0$, we have $M_j(r_0)\neq 0$ for all $1\leq j\leq K$. Let $N\subset M$ be an $\FA$-submodule such that $N(r)=0$ for $r\leq r_0$. Then $N\equiv 0$.
\end{lem}
\begin{proof}
    We proceed by induction on $K$.
    For $K=1$ the statement is clear, since $N\subsetneq M_1$ is by assumption a proper submodule of the simple $\FA$-module $M_1$.
    For $K\geq 2$ let $r$ be the smallest $r$ such that $N(r)\neq 0$. We have $r>r_0$ by assumption.
    If $N(r)\cap M_K(r)\neq 0$, then $N\cap M_K$ is a proper $\FA$-submodule of $M_K$, contradicting simplicity of $M_K$.
    Otherwise replace $M$ by $M/M_K=M_1\oplus\cdots\oplus M_{K-1}$ and replace $N$ by the image $N'$ of $N$ in $M/M_K$. We still have $N'(r)\neq 0$ since $N(r)\not\subset M_K(r)$, and hence the lemma follows by induction. 
\end{proof}
\begin{cor}\label{cor:FA inj}
    Let $M_1,\dots, M_K$ be simple $\FA$-modules corresponding to Young diagrams with $n$ boxes, and let 
    \[
    f: M:=M_1\oplus\cdots \oplus M_K \to M'
    \]
    be a morphism of $\FA$-modules into some other $\FA$-module $M'$.
    If the arity-$n$-map $f(n):M(n)\to M'(n)$ is injective, then $f$ is injective in all arities.
\end{cor}
\begin{proof}
    The $\FA$-submodule $N:=\ker f\subset M$ satisfies the assumptions of Lemma \ref{lem:FA inj}.
\end{proof}

\section{Holomorphic forms when \texorpdfstring{$g\leq 2$}{gleq2}}\label{sec:gleq2}
The moduli space $\Mb_{0,n}$ has no nonzero holomorphic forms of positive degree, as all of its cohomology is algebraic. Here, we describe the holomorphic forms on $\Mb_{1,n}$ and $\Mb_{2,n}$.

\subsection{Holomorphic forms on \texorpdfstring{$\Mb_{1,n}$}{Mb1n}}
We recall well-known results about holomorphic forms on $\Mb_{1,n}$. Let $\s_{k+1}$ be the weight $k$ Hodge structure associated to the space of weight $k+1$ cusp forms for $\SL_{2}(\zz)$. Then $\s_{k+1}$ has Hodge weights in $\{(k,0), (0,k)\}$, and is nonzero exactly when $k\geq 11$ is odd and $k\neq 13$. Let $K^m_n := V_{n-m+1,1^{m-1}}$ denote the $\ss_n$-representation associated to the hook shape whose vertical part has exactly $m$ boxes, as in Figure~\ref{fig:hook}. 
\begin{prop}\label{prop:genus1kforms}
    The $\ss_n$-representation $H^{k,0}(\Mb_{1,n})$ is the $(k,0)$ part of $K^{k}_n\otimes \s_{k+1}\otimes \cc$. 
\end{prop}

\begin{proof}
    There is a short exact sequence
    \[
    H^{k-2}(\tilde{\partial \M}_{1,n})\rightarrow H^{k}(\Mb_{1,n})\rightarrow W_{k}H^{k}(\M_{1,n})\rightarrow 0.
    \]
    The first morphism is of Hodge type $(1,1)$, and hence $H^{k,0}(\Mb_{1,n})$ is identified with the $(k,0)$ part of $W_{k}H^{k}(\M_{1,n})$. By \cite[Proposition 2.2]{CLP-STE}, $W_{k}H^{k}(\M_{1,n})\cong K^{k}_n\otimes \mathsf{S}_{k+1}$.  
\end{proof}
Using Proposition \ref{prop:genus1kforms} and Corollary \ref{cor:FA inj}, we describe the $\FA$-module $H^{k,0}(\Mb_{1,*})$ for all $k$. %

\begin{cor}\label{cor:FA genus 1}
Let $N_k=\dim \s_{k+1} /2$ be the dimension of the space of weight $k+1$-cusp forms. In the category of $\FA$-modules, we have an isomorphism
\[
H^{k,0}(\Mb_{1,*}) \cong
\underbrace{\tilde{C}_{1^{k}} \oplus \cdots \oplus \tilde{C}_{1^{k}}}_{N_k}=\tilde{C}_{1^{k}}\otimes \cc^{N_k}.
\]
\end{cor}
\begin{proof}
By Proposition \ref{prop:genus1kforms}, we have a morphism of $\FS$-modules 
\[
\rho_{1^k}\otimes \cc^{N_k} \to H^{k,0}(\Mb_{1,*}),
\]
with the irreducible representation $\rho_{1^k}$ of $\ss_k$ considered as an $\FS$-module concentrated in arity $k$. By adjunction, we have a morphism of $\FA$-modules 
\[
C_{1^k}\otimes \cc^{N_k} := \Ind_{\FS}^{\FA} \rho_{1^k}\otimes \cc^{N_k} \to H^{k,0}(\Mb_{1,*}).
\]
The image of $C_{1^{k+1}}\otimes \cc^{N_k} \to C_{1^k}\otimes \cc^{N_k}$ is in the kernel of this morphism, and hence it descends to
\[
\tilde C_{1^k} \otimes \cc^{N_k} \to H^{k,0}(\Mb_{1,*}).
\]
This morphism is a bijection in arity $k$, and hence in particular injective in arity $k$ by Proposition~\ref{prop:genus1kforms}.
Hence it is an injection in all arities by Corollary \ref{cor:FA inj}, because $\tilde C_{1^k}$ is simple. The dimensions of both sides agree arity-wise by Proposition \ref{prop:genus1kforms}, so it is an isomorphism. 
\end{proof}

Recall that $N_k \neq 0$ exactly when $k \geq 11$ is odd and $k \neq 13$.
If $11 \leq k \leq 21$ odd and $k \neq 13$, then $N_k = 1$, so $H^{k,0}(\Mb_{1,*})$ is simple. Meanwhile, for odd $k \geq 23$, we have $N_k > 1$, so $H^{k,0}(\Mb_{1,*})$ is not simple. Nevertheless, $H^{k,0}(\Mb_{1,*})$ is the holomorphic part of the simple object $\mathsf{S}_{k+1} \otimes \widetilde{C}_{1^k}$ in the category of functors from $\FA$ to rational Hodge structures, see Remark \ref{rem:motivic} for further discussion.

\subsection{The cohomology of local systems on \texorpdfstring{$\A_2$}{A2}}
We now recall the cohomology of local systems on $\A_2$, the moduli space of principally polarized abelian surfaces, following \cite{Petersen}. In the next section, we will discuss how the holomorphic forms on $\Mb_{2,n}$ come from the cohomology of these local systems.

Let $\pi \colon \X_2\rightarrow \A_2$ be the universal abelian surface, and let $\vv := R^1 \pi_* \cc$.
The local system $\vv$ is associated to the contragredient of the standard representation of $\mathrm{GSp}_{4}(\zz)$. Let $\lambda$ be a partition of length at most 2, with parts $a\geq b\geq 0$. Up to Tate twist, every irreducible local system on $\A_2$ is associated to such a partition via the irreducible representation of highest weight in $\Sym^{a-b} \vv \otimes \Sym^b(\wedge^2 \vv)$. We denote the associated local system by $\vv_{\lambda}$ or $\vv_{a,b}$.

The cohomology groups $H^i(\A_2, \vv_{\lambda})$ carry a mixed Hodge structure of weight $\geq a+b+i$ \cite[p. 233]{FaltingsChai}. %
We will need, in particular, the pure Hodge structure $W_{a+b+3}H^3(\A_2,\mathbb{V}_{\lambda})$, which has Hodge weights in $\{(a+b+3,0)$,$(a+2,b+1)$, $(b+1,a+2)$, and $(0,a+b+3)\}$ \cite{FaltingsChai, Getzlertoprecursion}.

The cohomology in the case of trivial coefficients is well-known. We have $H^{i}(\A_2,\vv_{0,0}) = 0$ unless $i=0,2$. Furthermore, $H^0(\A_2,\vv_{0,0}) = \cc$, and $H^2(\A_2,\vv_{0,0}) = \cc(-1)$. %
Petersen has computed the %
cohomology groups for nontrivial coefficients. We only need the pure weight part. %
For any $k\geq 0$ and $j\geq 3$, let $\mathsf{S}_{j,k}$ be the complexification of the $\ell$-adic Galois representation associated to the space of vector valued Siegel eigenforms for $\Sp_4(\zz)$ of type $\Sym^j\otimes \det^k$ for some chosen prime $\ell$. %
We define  $$\overline{\s}_{j,k} := \gr^W_{j+2k-3} \s_{j,k}.$$ %
We have $\overline{\mathsf{S}}_{j,k}=\mathsf{S}_{j,k}$ unless $j=0$ and $k$ is even. When $j = 0$ and $k$ is even, we have
\[
\s_{0,k} \cong \cc(-k+2)^{\oplus \dim \mathsf{S}_{2k-2}}\oplus \overline{\s}_{0,k} \oplus \cc(-k+1)^{\oplus \dim \mathsf{S}_{2k-2}}.
\]
The additional Tate summands are explained by the structure of Saito--Kurokawa lifts associated to cusp forms for $\SL_2(\zz)$, using the $\ell$-adic counterparts of these local systems \cite[p.~42]{Petersenlocalsystems}.

\begin{prop}\label{prop:pureweightabelian}
Let $(a,b)\neq (0,0)$. Then $H^i(\A_2,\vv_{a,b}) = 0$ unless $a+b$ is even and $i=2,3$. When $i=2,3$, the following holds.
    \begin{enumerate}
        \item The space $W_{a+b+2} H^2(\A_2,\vv_{a,b}) = 0$ unless $a=b=2c$ for some $c\geq 1$, in which case $W_{2c+2} H^2(\A_2,\vv_{2c,2c})$ is pure Tate. %
        \item The space $W_{a+b+3} H^3(\A_2,\vv_{a,b})=\overline{\s}_{a-b,b+3}$. 
    \end{enumerate}
\end{prop}
\noindent This is a simplified version of \cite[Theorem 2.1]{Petersenlocalsystems}, which is stated for the weight-graded compactly supported $\ell$-adic cohomology. Proposition~\ref{prop:pureweightabelian} is obtained by taking the pure weight part, tensoring with $\cc$, and applying the comparison theorems and Poincar\'e duality.

\subsection{Holomorphic forms on \texorpdfstring{$\Mb_{2,n}$}{Mb2n}}
We now discuss how holomorphic forms on $\Mb_{2,n}$ arise from the cohomology of local systems on $\A_2$. %
\begin{prop}\label{prop:genus2evenforms}
    If $k\neq 0$ is even, then $H^{k,0}(\Mb_{2,n}) = 0$.
\end{prop}
\begin{proof}
     We have the short exact sequence
    \[
    H^{k-2}(\tilde{\partial \M}_{2,n})\rightarrow H^{k}(\Mb_{2,n})\rightarrow W_{k}H^{k}(\M_{2,n})\rightarrow 0,
    \]
    where the first morphism is of Hodge type $(1,1)$. Therefore, $H^{k,0}(\Mb_{2,n})$ is identified with the Hodge type $(k,0)$ part of $W_{k}H^{k}(\M_{2,n})$.

    Let $\pi\colon \C\rightarrow \M_2$ be the universal curve and $f\colon \C^n\rightarrow \M_2$ its $n$-fold fiber product. The open embedding $\M_{2,n}\subset \C^n$ induces an exact sequence
    \[
    \bigoplus W_{k-2} H^{k-2}(\C^{n-1})\rightarrow W_k H^k (\C^{n})\rightarrow W_k H^k(\M_{2,n})\rightarrow 0.
    \]
    The first morphism is of Hodge type $(1,1)$, so the $(k,0)$ part of $W_k H^k(\M_{2,n})$ is identified with the $(k,0)$ part of  $W_k H^k (\C^{n})$.

    Consider the Leray spectral sequence for the morphism $f\colon \C^n\rightarrow \M_2$. Because $f$ is smooth and proper, we have
    \[
    H^k(\C^n)\cong \bigoplus_{p+q=k} H^p(\M_2,R^qf_*\cc).
    \]
    When $q$ is odd, $H^p(\M_2,R^qf_*\cc)=0$, as there are no invariants under the action of the hyperelliptic involution. Moreover, $\M_2$ is affine of dimension $3$, so $H^p(\M_2,R^qf_*\cc)=0$ for $p>3$. Therefore, because $k$ is even, we have
    \[
    H^k(\C^n)\cong H^0(\M_2,R^{k}f_*\cc)\oplus H^2(\M_2,R^{k-2}f_*\cc).
    \]
    
    We have an open embedding $\M_2\subset \A_2$. The local systems $R^{k}f_*\cc$ and $R^{k-2}f_*\cc$ decompose as sums of symplectic local systems of weight $k$ and $k-2$, respectively, which are pulled back from $\A_2$. For any such symplectic local system $\mathbb{V}_{\lambda}$ with $|\lambda| = k$, we have a surjection
    \[
    W_{k} H^{0} (\A_2,\mathbb{V}_{\lambda})\rightarrow W_k H^{0}(\M_2,\mathbb{V}_{\lambda}),
    \]
    and so $W_k H^{0}(\M_2,\mathbb{V}_{\lambda}) = 0$ by Proposition \ref{prop:pureweightabelian}. %
    Similarly, if $|\lambda|=k-2$, we have a surjection  
    \[
    W_{k-2} H^{2} (\A_2,\mathbb{V}_{\lambda})\rightarrow W_{k-2} H^{2}(\M_2,\mathbb{V}_{\lambda}).
    \]
    After tensoring with $\cc$, we see that $W_{k-2} H^{2}(\M_2,\mathbb{V}_{\lambda})$ is pure Tate, by Proposition \ref{prop:pureweightabelian}, and so the $(k,0)$ part of $W_k H^k (\C^{n})$ vanishes.
\end{proof}
We now consider the odd degree case.
\begin{prop}\label{prop:genus2forms}
Let $k$ be odd and $n\geq k-3$. Then $H^{k,0}(\Mb_{2,n})\cong \bigoplus_{A}H^{k,0}(\Mb_{2,A})$, where the summation runs over subsets $A\subset \{1,\dots,n\}$ such that $|A|=k-3$ and the isomorphism is given by the pullback maps forgetting the markings not in $A$. When $n=k-3$, the $\ss_n$-representation $H^{k,0}(\Mb_{2,n})$ is identified with the $(k,0)$ part of $\bigoplus_{|\lambda|=k-3} W_{k} H^{3}(\A_2, \mathbb{V}_{\lambda})\otimes V_{\lambda^T}$.
\end{prop}
\begin{proof}

 As in the proof of Proposition \ref{prop:genus2evenforms}, we reduce to the study of the $(k,0)$ part of $W_kH^k(\C^n)$. We have
    \[
    H^k(\C^n)\cong H^1(\M_2,R^{k-1}f_*\cc)\oplus H^3(\M_2,R^{k-3}f_*\cc).
    \]
     The local system $R^{k-1}f_*\cc$  decomposes as sums of symplectic local systems of weight $k-1$, which are pulled back from $\A_2$. For any such symplectic local system $\mathbb{V}_{\lambda}$, we have a surjection
    \[
    W_{k} H^{1} (\A_2,\mathbb{V}_{\lambda})\rightarrow W_k H^{1}(\M_2,\mathbb{V}_{\lambda}).
    \]
    By Proposition \ref{prop:pureweightabelian}, $H^{1} (\A_2,\mathbb{V}_\lambda)=0$, where again in the Betti setting we apply a comparison theorem to obtain the vanishing. %
    Hence, we have
    \[
    W_k H^k(\C^n) \cong W_k H^3(\M_2,R^{k-3}f_*\cc).
    \]
    By the K\"unneth formula,
    \[
    W_k H^3(\M_2,R^{k-3}f_*\cc) \cong \bigoplus_{i_1+\dots+i_n=k-3} W_k H^3(\M_2,R^{i_1}\pi_*\cc \otimes \dots \otimes R^{i_n}\pi_*\cc).
    \]

    Given a subset $A \subset \{1, \ldots, n\}$, let $\C^n \to \C^A$ be the projection onto the factors indexed by $A$. We consider the following subspaces of $H^k(\C^n)$:

 \begin{itemize}
     \item $\tilde{\Phi}$, the span of the pullbacks of $H^k(\C^{\{i\}^c})$ along projection $\C^n \to \C^{\{i\}^c}$
     \item $\tilde{\Psi}$, the span of $\psi_i \cdot H^{k-2}(\C^{\{i\}^c})$
\end{itemize}

    As observed in \cite[Section 5.2.2]{PTY}, the subspace $\widetilde{\Phi}$ corresponds to the span of the summands where some $i_s = 0$. 
Meanwhile, modulo $\tilde{\Phi}$, the subspace $\tilde{\Psi}$ corresponds to the span of summands where some $i_s = 2$. This follows from the formulas for the projector $\pi_2$ (which projects onto such summands) in \cite[Section 5.1]{PTY}. Therefore, any summand with some index $i_s=2$ cannot contribute to the $(k,0)$ part of the Hodge structure, as $\psi$ classes are of type $(1,1)$. From now on, we only consider the case when each index $i_s$ is $0$ or $1$. 

For $n>k-3$, consider a summand $W_k H^3(\M_2,R^{i_1}\pi_*\cc\otimes \dots \otimes R^{i_n}\pi_*\cc)$. Let $A\subset \{1,\dots,n\}$ be the set of indices such that $i_s=1$
if and only if $s\in A$. Then $|A|=k-3$, and the summand $H^3(\M_2,R^{i_1}\pi_*\cc\otimes \dots \otimes R^{i_n}\pi_*\cc)$ is pulled back from $H^k(\C^{A})$ along the projection $\C^n\rightarrow \C^{A}$.    

When $n=k-3$, the only possibility is that each index $i_s$ is $1$. We thus identify $H^{k,0}(\Mb_{2,n})$ with the $(k,0)$ part of $W_k H^3(\M_2, (R^1\pi_* \cc)^{\otimes k-3})$.  As in the proof of \cite[Lemma 3.1(a)]{CLP-STE}, we have an isomorphism 
\[
W_{k} H^3(\M_2, (R^1\pi_* \cc)^{\otimes k-3}) \cong \bigoplus_{|\lambda|=k-3} W_{k} H^{3}(\M_2, \mathbb{V}_{\lambda})\otimes V_{\lambda^T}.
\] 
We have an exact sequence
\[
W_{k-2} H^{1}(\A_1\times \A_1,\mathbb{V}_{\lambda})\rightarrow W_k H^3(\A_2,\mathbb{V}_{\lambda})\rightarrow W_k H^3(\M_2,\mathbb{V}_{\lambda}),
\]
where the first morphism is of type $(1,1)$. Hence, the $(k,0)$ part of $W_k H^3(\M_2,\mathbb{V}_{\lambda})$ is identified with the $(k,0)$ part of $W_k H^3(\A_2,\mathbb{V}_{\lambda})$.
\end{proof}

By \cite{CanningLarsonPayne, CLP-STE}, $H^{k,0}(\Mb_{2,n}) = 0$ for $1\leq k\leq 15$. We compute $H^{k,0}(\Mb_{2,n})$ for $k=17$ and $19$.
\begin{prop}\label{prop:genus217and19}\leavevmode There are isomorphisms of $\ss_n$-representations
\[
H^{17,0}(\Mb_{2,n})\cong  %
\mathrm{Ind}_{\ss_{14}\times \ss_{n-14}}^{\ss_{n}} (V_{2^{7}}\boxtimes \mathbf{1}) \quad \mbox{ and } \quad %
    H^{19,0}(\Mb_{2,n}) %
    \cong \mathrm{Ind}_{\ss_{16}\times \ss_{n-16}}^{\ss_{n}} (V_{2^5,1^6}\boxtimes \mathbf{1}).%
    \]
\end{prop}
\begin{proof}
    It suffices to compute the $(k,0)$ part of $\bigoplus_{|\lambda|=k-3} W_{k} H^{3}(\A_2, \mathbb{V}_{\lambda})\otimes V_{\lambda^T}$ for $k=17,19$, by Proposition \ref{prop:genus2forms}. These spaces can be explicitly computed in terms of Siegel modular forms as in Proposition \ref{prop:pureweightabelian}. We use dimension formulas for these spaces of modular forms, recorded at \cite{modular}. When $k = 17$, the only partition $\lambda$ that contributes a nonzero space of forms is $\lambda = (7,7)$, in which case we have $W_{17} H^3(\A_2,\vv_{7,7}) = \overline{\s}_{0,10}$. The $(17,0)$ part of $\overline{\s}_{0,10}$ is one-dimensional with trivial $\ss_{17}$-action. When $k=19$, the only partition $\lambda$ that contributes %
    is $\lambda = (11,5)$. Its contribution is $W_{19} H^3(\A_2,{\vv_{11,5}}) = \s_{6,8}$. The $(19,0)$ part of $\s_{6,8}$ is one-dimensional with trivial $\ss_{19}$-action
\end{proof}
\begin{rem}
    In principle, the calculation in Proposition \ref{prop:genus217and19} can be done in arbitrary cohomological degree $k$. As $k$ grows, more spaces of forms contribute. For example, when $k=21$, there is a nonzero contribution from $\overline{\s}_{0,12}, \s_{4,10}, \s_{8,8}$, and $\s_{12,6}$.
\end{rem}
\begin{cor} \label{cor:FA genus 2}
There are isomorphisms of $\FA$-modules
\[H^{17,0}(\Mb_{2,*}) \cong C_{2^{7}} \qquad \text{and} \qquad H^{19,0}(\Mb_{2,*}) \cong C_{2^5,1^6}.\]
\end{cor}
\begin{proof}
    We have morphisms of $\FS$-modules 
    \begin{align*}
     \rho_{2^7} &\to  H^{17,0}(\Mb_{2,*})
     &
     \rho_{2^5,1^6} &\to H^{19,0}(\Mb_{2,*}).
    \end{align*}
    By adjunction, these extend to (non-zero) morphisms of $\FA$-modules 
    \begin{align*}
     C_{2^7} &\to  H^{17,0}(\Mb_{2,*})
     &
     C_{2^5,1^6} &\to H^{19,0}(\Mb_{2,*}).
    \end{align*}
    Because $C_{2^7}$ and $C_{2^5,1^6}$ are simple, these morphisms must be injective.
    Hence, they are isomorphisms because both source and target have the same dimension in each arity, by Proposition \ref{prop:genus217and19}.
\end{proof}

\begin{rem}
     A longer proof of Corollaries \ref{cor:FA genus 1} and \ref{cor:FA genus 2} can be obtained by directly analyzing the pullbacks of holomorphic forms to boundary divisors, following \cite[Section 2.2]{CanningLarsonPayne}.
\end{rem}

\subsubsection{Pullback formulas} \label{sec:pull}
The $\FA$-module structure on $H^{k,0}(\Mb_{2,*})$ encodes the data of the pullback morphisms to boundary divisors with genus $0$ tails. Here, we explain how holomorphic forms pull back to the other boundary divisors.
For $B\subset \{1,\dots,n\}$, we write $D_B = \Mb_{1,B\cup p}\times \Mb_{1,B^c\cup q}$, and we define $\iota_B$ to be the map gluing $p$ to $q$. 
\begin{lem}\label{lem:genus2togenus1}
    We have $\iota_B^*H^{k,0}(\Mb_{2,n})=0$.
\end{lem}
\begin{proof}
    For $A\subset \{1,\dots,n\}$ with $|A|=k-3$, we have the commutative diagram
    \begin{equation} \label{dgenus2}
\begin{tikzcd}
D_B \arrow{r}{\iota_B} \arrow{d} & \Mb_{2,n} \arrow{d}{f_A} \\
\Mb_{1,(A \cap B) \cup p} \times \Mb_{1,(A \cap B^c) \cup q} \arrow{r} & \Mb_{2,A},
\end{tikzcd}
\end{equation}
where the vertical maps forget the markings not in $A$ and the horizontal maps glue $p$ to $q$. Both $|A\cap B|$ and $|A \cap B^c|$ are at most $k-3$, so $H^{k,0}(\Mb_{1,(A \cap B) \cup p} \times \Mb_{1,(A \cap B^c) \cup q})=0$ by the K\"unneth formula and Proposition \ref{prop:genus1kforms}.
\end{proof}
Next we consider the pull back under the self-gluing map. Let $P$ be a set of size $n$. We have $\xi\colon \Mb_{1,P\cup \{p,q\}}\rightarrow\Mb_{2,P}$, gluing the markings $p$ and $q$. 
\begin{lem}\label{lem:selfglue17}
    We have $\xi^*H^{k,0}(\Mb_{2,n})=0$.
\end{lem}
\begin{proof}
We have a commutative diagram
\begin{equation}
\begin{tikzcd}
\Mb_{1,P\cup \{p,q\}} \arrow{r}{\xi} \arrow{d}[swap]{} & \Mb_{2,P} \arrow{d}{f_A} \\
\Mb_{1,A \cup \{p,q\}} \arrow{r}[swap]{} & \Mb_{2,A},
\end{tikzcd}
\end{equation}
where the vertical maps forget all points in $P\smallsetminus A$, and $A$ is a set of size $k-3$. We have that $A\cup \{p,q\}$ is of size $k-1$, but $H^{k,0}(\Mb_{1,k-1})=0$, by Proposition \ref{prop:genus1kforms}.
\end{proof}

\begin{rem} \label{rem:motivic}
We have stated our results in this section, and throughout this paper, in terms of holomorphic forms. However, our approach is informed by a motivic perspective, meaning that each cohomology group $H^k(\Mb_{g,n})$ should be thought of not only as a complex vector space with a Hodge decomposition, but rather as a motivic structure, i.e. a rational vector space $V$ (the rational singular cohomology), together with a (pure or mixed) Hodge structure on $V \otimes \cc$ and a continuous action of the absolute Galois group on $V \otimes \qq_p$ for each prime $p$, with suitable compatibilities and comparison isomorphisms, as in \cite{Deligne89}.  (In many circumstances, one may also wish to include other realizations, such as de Rham and crystalline cohomology, but we leave this aside.)

The spaces of holomorphic forms that we consider all come from natural motivic structures such as  $\s_{k+1} := W_k H^k (\M_{1,k})$ and $\overline \s_{a-b,b+3} := W_{a + b + 3} H^3(\A_2, \vv_{a,b})$. Here, we consider $\vv_{a,b}$ as an algebraic local system with rational coefficients, so its cohomology carries a (mixed) motivic structure.  For $k = 17$, the image of $H^{17}_c(\M_{1,17}) \to H^{17}(\Mb_{1,17})$ is an $\ss_{17}$-equivariant motivic structure isomorphic to $\s_{18} \otimes V_{1^{17}}$. By pulling back under tautological morphisms and working operadically, as in Remark~\ref{rem:operadic}, we obtain a functor from $\FA$ to motivic structures isomorphic to $\s_{18} \otimes \widetilde{C}_{1^m}$. Here, we consider $\widetilde C_{1^m}$ as a simple $\FA$-module over $\qq$.   Thus we have described not only the space of holomorphic forms $H^{17,0}$ but also a rational structure on $H^{17,0}(\Mb_{1,n}) \oplus H^{0,17}(\Mb_{1,n})$ and an associated $p$-adic Galois representation for each prime~$p$.

Likewise, there is a natural subquotient of  $H^{17}(\Mb_{2,14})$ that is isomorphic to $\overline \s_{0,10}$. This generates a subquotient of $H^{17}(\Mb_{2,*})$ in the category of functors from $\FA$ to motivic structures that is isomorphic to $C_{2^{7}} \otimes \overline \s_{0,10}$ and accounts for all holomorphic $17$-forms on $\Mb_{2,n}$, for all $n$. One word of caution is warranted: the $p$-adic Galois representations attached to $\s_{18}$ and $\overline \s_{0,10}$ are isomorphic, by the theory of Saito--Kurokawa lifts, cf. \cite[Section~2]{Petersenlocalsystems}. It is expected that they are isomorphic as motivic structures, but this is not known. In particular, it is not known whether the associated rational Hodge structures are isomorphic.
\end{rem}

\section{Inductive arguments with modular cooperads} \label{sec:AC}

Here we briefly recall the formalism of modular cooperads and Getzler--Kapranov graph complexes. We then present a lemma adapting the inductive arguments for computing cohomology groups of moduli spaces of curves from \cite{ArbarelloCornalba} into this formalism.

\subsection{Modular cooperads and the Feynmann transform}
A stable $\ss$-module is a collection of dg vector spaces $\P = \{\P(g,n)\}$ for each $(g,n)$ such that $g,n\geq 0$ and $2g+n\geq 3$, together with an $\ss_n$-action on each $\P(g,n)$. For a stable $\ss$-module $\P$ and stable graph $\Gamma$, we define the tensor product
\[
\bigotimes_{\Gamma} \P := \bigotimes_{v\in V(\Gamma)} \P(g_v,n_v).
\]
The group $\Aut(\Gamma)$ acts on $\bigotimes_{\Gamma} \P$. A modular cooperad is a stable $\ss$-module $\P$ together with morphisms
\[
\P(g,n)\rightarrow \bigotimes_{\Gamma} \P
\]
for every stable graph $\Gamma$ of genus $g$ with $n$ legs. 
Alternatively, a modular cooperad may be defined as a stable $\ss$-module $\P$ together with morphisms 
\begin{equation}\label{equ:mod coop def}
    \begin{aligned}
    \eta^* \colon \P(g,n) &\to \P(h,m+1) \otimes \P(g-h,n-m+1), 
    \\
    \xi^*\colon \P(g,n) &\to \P(g-1,n+2).
\end{aligned}
\end{equation}
These morphisms must satisfy a list of natural compatibility relations; see \cite[Section 5.3]{MSS}.

If the modular cooperad $\P^\bullet$ has an additional grading, we write
\[
\bigotimes_{\Gamma} \P^k := \bigoplus_{\sum k_v = k}\bigg( \bigotimes_{v\in V(\Gamma)} \P^{k_v}(g_v,n_v)\bigg).
\]

The Feynman transform of a modular cooperad $\P$ is a $\mathfrak K$-modular operad denoted $\F\P$. Note that the structure maps of a $\mathfrak K$-modular operad dual to \eqref{equ:mod coop def} have cohomological degree +1 instead of 0; see \cite{GetzlerKapranov}. The underlying $\ss$-module of $\F\P$ is
\[
\F\P(g,n) = \bigoplus_{[\Gamma]}\bigg[ \bigotimes_{\Gamma} \P\otimes \mathbb{F}[-1]^{\otimes |E(\Gamma)|} \bigg]_{\Aut (\Gamma)}.
\]
Here, the sum is over isomorphism classes of stable graphs $\Gamma$ of genus $g$ with $n$ legs, and $\mathbb{F}$ is the field over which each $\P(g,n)$ is a vector space. The differential is defined using the cooperadic structure maps $\eta^*$, $\xi^*$ of $\P$. If $\P^\bullet$ has an additional grading, then $\F\P^\bullet$ inherits an additional grading from $\P^\bullet$. 
In the case of most interest to us, $\P(g,n)=H^\bullet(\Mb_{g,n})$, this additional grading is the cohomological or weight grading. A decorated graph generator of $\F\P$ of total weight $k$ with $e$ edges has cohomological degree $k + e$, as each edge contributes $+1$ to the cohomological degree. %

\subsection{The Getzler--Kapranov complexes}
The cohomology groups $H^\bullet(\Mb_{g,n})$ give rise to a modular cooperad $H(\Mb)$. In this case, the structure maps \eqref{equ:mod coop def} are defined to be the pullback morphisms associated to the boundary gluing morphisms
\begin{align*}
    \eta \colon\Mb_{h,m+1}\times \Mb_{g-h,n-m+1} &\to \Mb_{g,n} 
    &
    \xi \colon\Mb_{g-1,n+2} \to \Mb_{g,n}. 
\end{align*}

The cooperad $H(\Mb)$ has the extra data of a grading by cohomological degree. This grading induces a grading on the Feynmann transform $\F H(\Mb)$. We define the weight $k$ Getzler--Kapranov graph complex to be 
\[
\mathsf{GK}_{g,n}^k:= \gr_k \F H(\Mb)(g,n).
\]
Because $\mathsf{GK}_{g,n}^k$ is canonically identified with the $k$th row of the weight spectral sequence for the compactification $\M_{g,n}\subset \Mb_{g,n}$, we have
\[
H^\bullet(\mathsf{GK}_{g,n}^k) = \gr_k H^{\bullet}_c(\M_{g,n}).
\]

We also consider the modular cooperad $H^{\bullet,0}(\Mb)$ of holomorphic forms on the moduli space of stable curves, which is graded by the degree of the holomorphic forms. The results of Section \ref{sec:gleq2} describe the $g\leq 2$ part of the cooperad $H^{\bullet,0}(\Mb)$. When $g\leq 2$, the underlying $\ss$-module is described by Propositions \ref{prop:genus1kforms}, \ref{prop:genus2evenforms}, and \ref{prop:genus2forms}. The structure maps are described by Corollaries \ref{cor:FA genus 1} and \ref{cor:FA genus 2}, and Lemmas \ref{lem:genus2togenus1} and \ref{lem:selfglue17}.

The grading on $H^{\bullet,0}(\Mb)$ induces a grading on the Feynmann transform $\F H^{\bullet,0}(\Mb)$.  The Hodge weight $(k,0)$ Getzler--Kapranov complex is
\[\mathsf{GK}_{g,n}^{k,0}:=\gr_k \F H^{\bullet,0}(\Mb)(g,n).\] The generators of $\mathsf{GK}_{g,n}^{k,0}$ are dual graphs of stable curves of genus $g$ with $n$ numbered legs, each of whose vertices is decorated by a copy of $H^{k_v,0}(\Mb_{g_v,n_v})$, such that $\sum_{v\in \Gamma} k_v = k$.
More explicitly, $\mathsf{GK}_{g,n}^{k,0}$ is the complex

\begin{equation}\label{GKforms}
H^{k,0}(\Mb_{g,n})\xrightarrow{} \bigoplus_{|E(\Gamma)| = 1} H^{k,0}(\Mb_{\Gamma})^{\Aut(\Gamma)} \rightarrow
\bigoplus_{|E(\Gamma)| = 2} (H^{k,0}\big(\Mb_{\Gamma}) \otimes \det E(\Gamma)\big)^{\Aut(\Gamma)}\rightarrow \dots
\end{equation}

Note that $\eqref{GKforms}$ is the subcomplex of $\mathsf{GK}^{k}_{g,n}\otimes \cc$ of holomorphic $k$-forms, and thus we have an identification
\[
H_c^{\bullet}(\M_{g,n})^{k,0}:=F^k \gr_{k} H^{\bullet}_c(\M_{g,n},\cc) \cong H^\bullet(\mathsf{GK}^{k,0}_{g,n}),
\]
where $F$ is the Hodge filtration on the pure weight $k$ Hodge structure $\gr_{k} H^{\bullet}_c(\M_{g,n})$. %

\subsection{An Arbarello--Cornalba induction for modular cooperads}
The following lemma adapts the inductive arguments of Arbarello and Cornalba \cite{ArbarelloCornalba} into the formalism of modular cooperads. It gives a mechanism for checking whether an inclusion of graded modular cooperads %
is an isomorphism in a fixed degree $k$, by induction on $g$ and $n$, using the first two cohomology groups of the Feynman transform.

We recall that if $\P^\bullet$ is a graded modular cooperad, and $\T^\bullet\subset \P^\bullet$ is a graded modular subcooperad, then $\F\T^\bullet\subset \F\P^\bullet$ is a graded $\mathfrak K$-modular suboperad. %

\begin{lem}%
\label{lem:AC}
Let $\T^\bullet\subset \P^\bullet$ be an inclusion of graded modular cooperads. Suppose
\begin{enumerate}
    \item the inclusion induces isomorphisms $\bigotimes_{
\Gamma} \T^{k} \cong \bigotimes_{\Gamma} \P^k$ for all stable graphs $\Gamma$ of genus $g$ with $n$ legs and one or two edges, and
\item the induced maps $H^{k+e}(\F \T^k(g,n)) \rightarrow H^{k+e}(\F\P^k(g,n))$ are isomorphisms for $e = 0, 1$.%
\end{enumerate}
Then $\T^k(g,n) = \P^k(g,n)$.
    
\end{lem}
\begin{proof}
    The inclusion $\F\T^k\subset \F\P^k$, together with assumption (1), induces an injection $$H^{k}(\F\T^k(g,n))\hookrightarrow H^{k}(\F\P^k(g,n))$$ and a surjection
$$H^{k+1}(\F\T^k(g,n))\twoheadrightarrow H^{k+1}(\F\P^k(g,n)).$$ If both are isomorphisms, then $\T^k(g,n) = \P^k(g,n)$, by the five lemma.
\end{proof}

In light of Lemma \ref{lem:AC} and its proof, we see in particular that if condition (1) is satisfied, $H^{k}(\F\P^k(g,n)) = 0$, and $H^{k+1}(\F\T^k(g,n)) = 0$, then $\T^k(g,n) = \P^k(g,n)$. 

Suppose now that $\P \subset H(\Mb)$ is a subcooperad. Then $H^{k}(\F\P^k(g,n)) = 0$ when $g$ and $n$ are sufficiently large compared to $k$; this key observation is due to Arbarello and Cornalba \cite{ArbarelloCornalba}.   
For applications to holomorphic $17$ and $19$ forms in Section \ref{sec:1719}, the main new ingredient is the vanishing of $H^{k+1}(\F\T^k)$ for an appropriate choice of $\T^\bullet$. 
We prove this vanishing as a consequence of more general vanishing results on the cohomology of graph complexes associated to simple $\FA$-modules that we establish in Section~\ref{sec:FA}.

\section{Holomorphic 17-forms and 19-forms} \label{sec:1719}
Here, we finish the proof of Theorem \ref{thm:upto17}, assuming Proposition \ref{prop:lowdegree}. We apply Lemma~\ref{lem:AC}
to a modular cooperad $\T^\bullet$ with an additional grading. The underlying graded stable $\ss$-module is
\[
\T^k(g,n) = \begin{cases}
\cc & \text{ if $k=0$}\\
    K_n^{17} & \text{ if } g=1 \text{ and } k=17\\
    \Ind_{\ss_{14}\times \ss_{n-14}}^{\ss_n}(V_{2^7}\boxtimes \mathbf{1}) & \text{ if } g=2 \text{ and } k=17
    \\
    0  & \text{ otherwise.}
\end{cases}
\]
Note that for $g\leq 2$, we have $\T^{17}(g,n)\cong H^{17,0}(\Mb_{g,n})$, by the results of Section \ref{sec:gleq2}. We use this fact to define the modular cooperad morphisms. For $g\leq 2$, we define the modular cooperad morphisms to be the same as those from $H^{17,0}(\Mb)$, and we set all of the other morphisms to be trivial. Thus, $\T^\bullet$ is a modular sub-cooperad of $H^{\bullet,0}(\Mb)$. %

\begin{proof}[Proof of Theorem \ref{thm:upto17} assuming Proposition \ref{prop:lowdegree}]
By construction,  $\T^{17}(g,n)\cong H^{17,0}(\Mb_{g,n})$ for $g\leq 2$. Moreover, for all $(g,n)$ such that $g \geq 3$ and $2g-2+n\leq 17$ or $(g,n)=(3,14)$, $\Mb_{g,n}$ is rationally connected \cite{BallicoCasnatiFontanari,Logan}, and hence has no holomorphic forms. It follows that $\T^{17}(g,n)\cong H^{17,0}(\Mb_{g,n})$ for these pairs $(g,n)$, as well. 

For $(g,n)$ such that $g\geq 3$, $2g-2+n>17$, and $(g,n)\neq (3,14)$, we argue using Lemma \ref{lem:AC}. For the criterion in Lemma \ref{lem:AC}(1), we note that $H^{17,0}(\Mb_{\Gamma})\cong \bigoplus_{v\in V(\Gamma)} H^{17,0}(\Mb_{g(v),n(v)})$ by the Hodge--K\"unneth decomposition, as $H^{p,0}(\Mb_{g(v),n(v)})=0$ for $1\leq p\leq 8$, \cite{BergstromFaberPayne, CanningLarsonPayne}. Therefore, using the base cases from the previous paragraph, we may assume by induction on $g$ and $n$ that the inclusion $\T^\bullet\subset H^{\bullet,0}(\Mb)$ induces isomorphisms $\bigotimes_{\Gamma} \T^{17} \cong \bigotimes_{\Gamma} H^{17,0}(\Mb)$ for all stable graphs $\Gamma$ of genus $g$ with $n$ legs and one or two edges. 

Because of our assumption $2g - 2 + n > 17$, we may apply \cite[Proposition 2.1]{BergstromFaberPayne} to see that 
\[H^{17}(\mathsf{GK}^{17,0}_{g,n}) = H_c^{17}(\M_{g,n})^{17,0}\subset H_c^{17}(\M_{g,n},\cc) = 0.\]
By Proposition \ref{prop:lowdegree} (cf. Remark \ref{rem:GC relation}), we have $H^{18}(\F\T^{17})=0$. %
Hence, the natural maps from Lemma \ref{lem:AC}(2) are trivially isomorphisms, and $\T^{17}(g,n)\cong H^{17,0}(\Mb_{g,n})$ for all $(g,n)$. %
\end{proof}

The proof of Theorem \ref{thm:19} is similar to that of Theorem \ref{thm:upto17}. %
\begin{proof}[Proof of Theorem \ref{thm:19}, assuming Proposition \ref{prop:lowdegree}] 
Let $\Q^\bullet$ be the modular operad with underlying graded stable $\ss$-module
\[
\Q^k(g,n) = \begin{cases}
\cc & \text{ if } k=0\\
    K_n^{19} & \text{ if } g=1 \text{ and } k=19\\
     \Ind_{\ss_{16}\times \ss_{n-16}}^{\ss_{n}} (V_{2^5,1^6}\boxtimes \mathbf{1})
     &  \text{ if } g=2 \text{ and } k=19 \\
    0  & \text{ otherwise.}
\end{cases}
\]
By definition, for $g\leq 2$, we have $\Q^{19}(g,n)\cong H^{19,0}(\Mb_{g,n})$. For $g\leq 2$, we define the modular cooperad morphisms to be the same as those from $H^{19,0}(\Mb)$, and we set all of the other morphisms to be trivial. Thus, $\Q^\bullet$ is a modular sub-cooperad of $H^{\bullet,0}(\Mb)$.

For all $(g,n)$ such that $g\geq 3$,  $2g-2+n\leq 19$, and $(g,n)\neq (3,15)$, $\Mb_{g,n}$ is rationally connected \cite{BallicoCasnatiFontanari,Logan}, and hence has no holomorphic forms. It follows that $\Q^{19}(g,n)\cong H^{19,0}(\Mb_{g,n})$ for these pairs $(g,n)$. The assertion that $H^{19,0}(\Mb_{3,15}) = 0$ is an assumption in the theorem statement.

For $(g,n)$ such that $g\geq 3$, $2g-2+n>19$ and $(g,n)\neq (3,15)$, we argue using Lemma \ref{lem:AC}. For the criterion in Lemma \ref{lem:AC}(1), we note that $H^{19,0}(\Mb_{\Gamma})\cong \bigoplus_{v\in V(\Gamma)} H^{19,0}(\Mb_{g(v),n(v)})$ by the Hodge--K\"unneth decomposition, as $H^{p,0}(\Mb_{g(v),n(v)})=0$ for $1\leq p\leq 9$, \cite{BergstromFaberPayne, CanningLarsonPayne}. Therefore, using the base cases from the previous paragraph and the assumptions of the theorem, we may assume by induction on $g$ and $n$ that the inclusion $\Q^\bullet\subset H^{\bullet,0}(\Mb)$ induces isomorphisms $\bigotimes_{\Gamma} \Q^{19} \cong \bigotimes_{\Gamma} H^{19,0}(\Mb)$ for all stable graphs $\Gamma$ of genus $g$ with $n$ legs and one or two edges. 

Continuing with our assumption $2g - 2 + n > 19$,
we have that
\[H^{19}(\mathsf{GK}^{19,0}_{g,n}) = H_c^{19}(\M_{g,n})^{19,0}\subset H_c^{19}(\M_{g,n},\cc) = 0,\]
by \cite[Proposition 2.1]{BergstromFaberPayne}. Assuming also that $g \geq 3$ and $(g, n) \neq (3, 15)$, Proposition \ref{prop:lowdegree} (cf. again Remark \ref{rem:GC relation}) implies $H^{20}(\F\Q^{19}(g,n))=0$.
Hence, the maps from Lemma \ref{lem:AC}(2) are trivially isomorphisms, and so $\Q^{19}(g,n)\cong H^{19,0}(\Mb_{g,n})$ for all $(g,n)$. %
\end{proof}

\section{Graph complexes associated to \texorpdfstring{{$\FA$}}{FA}-modules}\label{sec:FA}

\subsection{Definition of the graph complex}
To any dg $\FA$-module $M$ we associate a graph complex $G_M$, which is a symmetric sequence of loop-order-graded dg vector spaces. More precisely, to the $\FA$-module $M$, we first associate a graded modular cooperad $\P_M$ such that 
\[
\P_M^k(g,n)
=
\begin{cases}
    \qq & \text{if $k=0$} \\
    M(n) & \text{if $k=1$ and $g=0$} \\
    0 & \text{otherwise.}
\end{cases}
\]
We call the extra degree $k$ the weight. Note that assigning weight $1$ and genus 0 to $M$ is an arbitrary choice.
The cooperadic cocompositions \eqref{equ:mod coop def} are defined using the $\FA$-module structure on $M$.
More precisely, we have that $\xi^*$ and $\eta^*$ are the identity maps in weight 0, while in weight 1, $\xi^*=0$ and $\eta^*$ is given by the $\FA$-module structure on $M$. 

Then we define graph complexes
\[
\widehat G_M(g,n) = \gr_1\F \P_M(g,n)
\]
as the weight one part of the Feynman transform of $\P_M$.
More explicitly, elements of $\widehat G_M(g, n)$ can be understood as linear combinations of graphs of genus $g$ with $n$ numbered legs, with one special vertex $*$ of genus $0$ decorated by an element of $M(E_*)$, with $E_*$ the set of half-edges incident at $*$. 
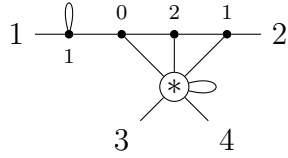
\[
\begin{tikzpicture}
    \node[ext] (c) at (0,0) {$*$};
    \node[int,label=90:{$\scriptstyle 2$}] (v1) at (0,.7) {};
    \node[int,label=90:{$\scriptstyle 0$}] (v2) at (-.7,.7) {};
    \node[int,label=90:{$\scriptstyle 1$}] (v3) at (.7,.7) {};
    \node[int,label=-90:{$\scriptstyle 1$}] (v4) at (-1.4,.7) {};
    \node (n1) at (-2.1,.7) {$1$};
    \node (n2) at (1.4,.7) {$2$};
    \node (n3) at (-.7,-.7) {$3$};
    \node (n4) at (.7,-.7) {$4$};
    \draw (c) edge (v1) edge (v2) edge (v3) edge (n3) edge (n4)
    (v2) edge (v4) (v3) edge (n2)
    (v1) edge (v2) edge (v3)
    (v4) edge (n1) edge[loop above] (v4);
    \draw (c) edge[loop right] (c);
\end{tikzpicture}
\]
There may be loop edges and each non-special (black) vertex has an associated genus, indicated in the drawing by the small number next to the vertex.
The cohomological grading on $\widehat G_M$ is by the number of edges, excluding the external legs, plus the cohomological degree of the decoration of the special vertex if the $\FA$-module $M$ carries a cohomological grading.
The $\ss_n$-action on $\widehat G_M(g,n)$ is by permuting the labels on the legs. %
The differential has the form $\delta= \delta_{split} + \delta_{loop}+ d_M$, with $d_M$ induced from the internal differential on $M$ (if present), while $\delta_{split}$ splits vertices and $\delta_{loop}$ acts as follows on the black vertices (only):
\[
\delta_{loop} :
\begin{tikzpicture}
    \node[int, label=0:{$\scriptstyle h$}] (v) at (0,0) {};
    \draw (v) edge +(-.5,-.5) edge +(0,-.5) edge +(.5,-.5)
    ;
\end{tikzpicture}
\mapsto
\begin{tikzpicture}
    \node[int, label=0:{$\scriptstyle h-1$}] (v) at (0,0) {};
    \draw (v) edge +(-.5,-.5) edge +(0,-.5) edge +(.5,-.5)
    edge[loop above] (v);
\end{tikzpicture}.
\]
Finally we define 
\[
G_M(g,n)\subset \widehat G_M(g,n)
\]
to be the subcomplex spanned by graphs in which each of the black vertices has genus 0, and no loop edge is attached to a black vertex.
\begin{lem}\label{lem:G vs hat G}
    The inclusion $G_M(g,n)\subset \widehat G_M(g,n)$ is a quasi-isomorphism for each $(g,n)$.
\end{lem}
\begin{proof}
    We use the argument of \cite[Proof of Theorem 3.3]{PayneWillwacher21}:
    Take a spectral sequence on the filtration by the number of vertices. Then the inclusion $G_M(g,n)\subset \widehat G_M(g,n)$ induces an isomorphism on the $E_1$-pages of the spectral sequences. Hence the result follows. 
\end{proof}

In the special case that $M=C_\lambda$ for $\lambda$ a Young diagram we abbreviate
\begin{align*}
\widehat G_\lambda &:= \widehat G_{C_\lambda} & G_\lambda &:= G_{C_\lambda},
\end{align*}
and in the special case $M=\tilde C_{1^k}$ we write 
\begin{align*}
\widehat G_{\tilde 1^k} &:= \widehat G_{\tilde C_{1^k}} & G_{\tilde 1^k} &:= G_{\tilde C_{1^k}}.
\end{align*}

\begin{rem}\label{rem:GM exact}
    The assignment $M\to G_M(g,n)$ is an exact functor from the category of dg $\FA$-modules to the category of dg $\ss_n$-modules. That is, exact sequences of $\FA$-modules $0\to M \to M'\to M''\to 0$ give rise to exact sequences of cochain complexes 
    \[
    0\to G_M(g,n) \to G_{M'}(g,n) \to G_{M''}(g,n) \to 0.
    \]
\end{rem}

\begin{rem}
    Note that the definition of the graph complex above only uses the $\FS$-module structure of $M$, not the full $\FA$-module structure, and hence ``$\FA$'' could be replaced with ``$\FS$'' throughout this section.
\end{rem}

\begin{rem}\label{rem:GC relation}
The graph complexes $\F\T^{17}$ and $\F\Q^{19}$ appearing in the proofs of Theorems \ref{thm:upto17} and \ref{thm:19} in Section \ref{sec:1719} above are related to our graph complexes $G_M$ and $\widehat G_M$ as 
\begin{align*}
    \F\T^{17}(g,n) &\cong  
    \widehat G_{\tilde 1^{17}}(g-1,n)[-17] \oplus 
    \widehat G_{2^{7}}(g-2,n)[-17]
    \\
    \F\Q^{19}(g,n) &\cong  
    \widehat G_{\tilde 1^{19}}(g-1,n)[-19] \oplus 
    \widehat G_{2^5 1^6}(g-2,n)[-19].
\end{align*}
Given Lemma \ref{lem:G vs hat G}, %
we hence have that 
\begin{align*}
    H^{17+e}(\F\T^{17})(g,n) &\cong  
    H^e(G_{\tilde 1^{17}})(g-1,n) \oplus 
    H^e(G_{2^{7}})(g-2,n)
    \\
    H^{19+e}(\F\Q^{19})(g,n) &\cong  
    H^e(G_{\tilde 1^{19}})(g-1,n) \oplus 
    H^e(G_{2^{5}1^6})(g-2,n).
\end{align*}

The vanishing results for the cohomology of $\F\T^{17}$ and $\F\Q^{19}$ used in Section \ref{sec:1719} are hence equivalent to vanishing results for the respective smaller complexes $G_M$. In this form, they are shown in Proposition \ref{prop:lowdegree} below.
\end{rem}

\subsubsection{Blown-up picture}\label{sec:blown up}
In practice, it is convenient to draw generators (graphs) of the graph complex $G_\lambda$ in a different way, as in \cite{PayneWillwacher21, PayneWillwacher23}.
Concretely, we may remove the special vertex and instead draw a disconnected graph
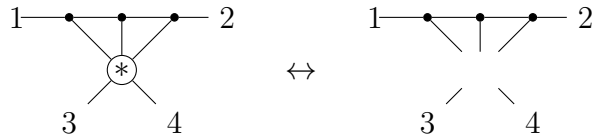
\[
\begin{tikzpicture}
    \node[ext] (c) at (0,0) {$*$};
    \node[int] (v1) at (0,.7) {};
    \node[int] (v2) at (-.7,.7) {};
    \node[int] (v3) at (.7,.7) {};
    \node (n1) at (-1.4,.7) {$1$};
    \node (n2) at (1.4,.7) {$2$};
    \node (n3) at (-.7,-.7) {$3$};
    \node (n4) at (.7,-.7) {$4$};
    \draw (c) edge (v1) edge (v2) edge (v3) edge (n3) edge (n4)
    (v2) edge (v4) (v3) edge (n2)
    (v1) edge (v2) edge (v3);
\end{tikzpicture}
\quad
\leftrightarrow
\quad
\begin{tikzpicture}
    \node (c) at (0,0) {$\phantom{**}$};
    \node[int] (v1) at (0,.7) {};
    \node[int] (v2) at (-.7,.7) {};
    \node[int] (v3) at (.7,.7) {};
    \node (n1) at (-1.4,.7) {$1$};
    \node (n2) at (1.4,.7) {$2$};
    \node (n3) at (-.7,-.7) {$3$};
    \node (n4) at (.7,-.7) {$4$};
    \draw (c) edge (v1) edge (v2) edge (v3) edge (n3) edge (n4)
    (v2) edge (v4) (v3) edge (n2)
    (v1) edge (v2) edge (v3);
\end{tikzpicture}
\]
Let $n=|\lambda|$ be the number of boxes of $\lambda$, and let $r\geq n$ be the valence of the special vertex in the graph we consider. We assume for simplicity that the half-edges at the special vertex are numbered $1,\dots,r$.
Then the decoration at the special vertex is an element of 
\[
C_\lambda(r) = \Ind_{\ss_{n}\times \ss_{n-r}}^{\ss_r} V_\lambda \boxtimes \mathbf{1} 
=
\bigoplus_{\substack{A\subset \{1,\dots,r\} \\ |A|=n}} V_\lambda.
\]
Hence the decoration at the special vertex can be considered combinatorially as a selection of a subset $A$ of the half-edges at the special vertex, of size $|A| = n$, together with an element of $V_\lambda$. We may enhance our drawing of the graph by marking the ``marked'' half-edges (those in $A$) at the special vertex with a symbol $\omega$, and the others with a symbol $\epsilon$, for example:
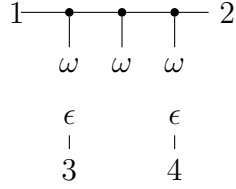
\[
\begin{tikzpicture}
    \node[int] (v1) at (0,.7) {};
    \node[int] (v2) at (-.7,.7) {};
    \node[int] (v3) at (.7,.7) {};
    \node (n1) at (-1.4,.7) {$1$};
    \node (n2) at (1.4,.7) {$2$};
    \node (n3) at (-.7,-1.4) {$3$};
    \node (n4) at (.7,-1.4) {$4$};
    \node (o1) at (-.7,0) {$\omega$};
    \node (o2) at (0,0) {$\omega$};
    \node (o3) at (.7,0) {$\omega$};
    \node (e1) at (-.7,-.7) {$\epsilon$};
    \node (e2) at (.7,-.7) {$\epsilon$};
    \draw (o2) edge (v1) (o1) edge (v2) (o3) edge (v3) 
    (e1) edge (n3)  (e2) edge (n4)
    (v2) edge (v4) (v3) edge (n2)
    (v1) edge (v2) edge (v3);
\end{tikzpicture}
\]
The element of $V_\lambda$ that is part of the decoration at the special vertex is omitted from the drawing.
We call this representation of generators the \emph{blown-up picture} and the connected components of the graph the \emph{blown-up components}. Note that in the blown-up picture there are three kinds of legs: numbered legs, $\omega$-legs, and $\epsilon$-legs. Because every generator is connected (before blowing up), each blown-up component has at least one $\omega$- or $\epsilon$-leg.

\subsection{Euler characteristic of \texorpdfstring{$G_\lambda$}{Glambda}}
We introduce the functions 
\begin{align*}
    E_\ell&:= \frac 1 \ell \sum_{d\mid \ell}\mu(\ell/d)\frac 1 {u^d},
    &
    \lambda_\ell &:= u^\ell (1-u^\ell)\ell,
\end{align*}
\begin{equation*} %
    B(z) := \sum_{r\geq 2}\frac{B_r}{r(r-1)} \frac 1 {z^{r-1}},
  \end{equation*}
where $\mu$ is the M\"obius function and $B_r$ is the $r$th Bernoulli number. Additionally, we define $U_\ell(X,u) = \exp(\log U_\ell(X,u))$ where
\begin{align*}
    \log U_\ell(X,u)
   &=   
    \log \frac {(-\lambda_\ell)^X \Gamma(-E_\ell+X) }{\Gamma(-E_\ell)}
    \\&=
        X\left(\log(\lambda_\ell E_\ell)-1 \right)+(-E_\ell+X-\textstyle{\frac 1 2} )\log(1-\textstyle{\frac X{E_\ell}}) + B(-E_\ell+X)- B(-E_\ell).
\end{align*}
Furthermore, for $f(w_1,\dots,w_k)$ a polynomial or power series in variables $w_1,\dots,w_k$, we denote by $T_{W}f$ the coefficient of the monomial $W$ in $f$.
Given representations $V$ and $W$ of $\mathbb{S}_m$ and $\mathbb{S}_n$, we define
\[V \hboxtimes W := \Ind_{\ss_m\times \ss_n}^{\ss_{m+n}} ( V \boxtimes W ).\]
Given partitions $\lambda$ and $\mu$, we also use the shorthand
\[\lambda \hboxtimes \mu = V_\lambda \hboxtimes V_\mu.\]

\begin{thm}\label{thm:euler general}
    The $\ss$-equivariant Euler characteristic of %
    $G_{1^{a_1}\hboxtimes \cdots \hboxtimes 1^{a_k}}$ satisfies:
\begin{multline*}
\sum_{g,n\geq 0} u^{g+n} \chi_{\ss_n}(G_{1^{a_1}\hboxtimes \cdots \hboxtimes 1^{a_k}}(g,n))
  \\=
  \  (-1)^{a_1+\cdots +a_k} T_{w_1^{a_1}\cdots w_k^{a_k}}\bigg(
  \prod_{\ell\geq 1} 
  \frac { 
    U_\ell(\frac 1 \ell \sum_{d\mid \ell} \mu(\ell/d) (-p_d  +1-(w_1^d+ \cdots + w_k^d) ), u )
  }
  { 
    U_\ell(\frac 1 \ell \sum_{d\mid \ell} \mu(\ell/d) (-p_d), u )
  }-1\bigg)\, .
\end{multline*} 
\end{thm}
\begin{proof}
    The proof is a straightforward extension of the proof of the special case $k=1$ provided in \cite[Section 7]{PayneWillwacher23}.
\end{proof}

\begin{cor}\label{cor:EC1}
    The $\ss$-equivariant Euler characteristic of the graph complex $G_{M}$ for $M=\tilde C_{1^{a}}$ satisfies 
    \begin{align*}
\sum_{g,n\geq 0} u^{g+n} \chi_{\ss_n}(G_{M}(g,n))
  =(-1)^{a-1}
 T_{\leq a-1}\bigg(
  \prod_{\ell\geq 1} 
  \frac { 
    U_\ell(\frac 1 \ell \sum_{d\mid \ell} \mu(\ell/d) (-p_d  +1-w^d ), u )
  }
  { 
    U_\ell(\frac 1 \ell \sum_{d\mid \ell} \mu(\ell/d) (-p_d), u )
  }-1\bigg)\, ,
\end{align*}
where we used the notation $T_{\leq r}:=\sum_{j=0}^r T_{w^j}$.
\end{cor}

\begin{proof}
    This follows from Theorem \ref{thm:euler general} for $k=1$, using the resolution \eqref{equ:tilde C fin res} for the $\FA$-module $\tilde C_{1^k}$, see also Remark \ref{rem:GM exact}. 
\end{proof}
\begin{cor}\label{cor:EC2}
    The $\ss$-equivariant Euler characteristic of the graph complex $G_{2^k1^\ell}$ satisfies 
    \begin{multline*}
\sum_{g,n\geq 0} u^{g+n} \chi_{\ss_n}(G_{2^k1^\ell}(g,n))
  \\=
  \  (T_{w_1^{k+\ell}w_2^k} -T_{w_1^{k+\ell+1}w_2^{k-1}})\bigg(
  \prod_{\ell\geq 1} 
  \frac { 
    U_\ell(\frac 1 \ell \sum_{d\mid \ell} \mu(\ell/d) (-p_d  +1-w_1^d-w_2^d ), u )
  }
  { 
    U_\ell(\frac 1 \ell \sum_{d\mid \ell} \mu(\ell/d) (-p_d), u )
  }-1\bigg)\, .
\end{multline*}
\end{cor}
\begin{proof}
    By Pieri's rule, the irreducible representation $V_{2^k1^\ell}$ is the cokernel of the inclusion 
    \[
    1^{k+\ell+1}\hboxtimes 1^{k-1} \to 1^{k+\ell}\hboxtimes 1^{k}.
    \]
    Hence, there is an exact sequence 
    \[
    0\to G_{1^{k+\ell+1}\hboxtimes 1^{k-1}}
    \to G_{1^{k+\ell}\hboxtimes 1^{k}}
    \to 
    G_{2^k1^\ell} \to 0,
    \]
    and the stated Euler characteristic formula follows from Theorem \ref{thm:euler general}.
\end{proof}

\subsection{Cohomology of  \texorpdfstring{$G_\lambda$}{Glambda}}
The cohomology of the graph complex $G_\lambda$ is complicated and not known in general. In this section, we state and prove some partial results.

\subsubsection{In high degrees}
\begin{lem}\label{lem:upper deg bound}
    Let $\lambda$ be a Young diagram with $N$ boxes. Then the complex $G_\lambda(g,n)$ is concentrated in cohomological degrees $\leq 3g+n-N$.
\end{lem}
\begin{proof}
    Consider some graph $\Gamma\in G_\lambda(g,n)$ with $v$ internal vertices, degree $k$ and whose special vertex has valence $N+p$. Note that $k$ is the total number of edges, not counting the numbered legs as edges.
    Then by the trivalence condition on the internal vertices we have that $2k + n \geq 3v + N + p$.
    On the other hand we have $v=k-g$, and hence $k\leq 3g+n-N-p\leq 3g+n-N$ as desired.
\end{proof}
\begin{cor}\label{cor:G tilde G top}
    The complex $G_{\tilde 1^N}(g,n)$ is concentrated in cohomological degrees $\leq 3g+n-N$. Furthermore, the top degree cohomology satisfies 
    \[
    H^{3g+n-N}(G_{\tilde 1^N})(g,n)
    =
    H^{3g+n-N}(G_{1^N})(g,n).
    \]
\end{cor}
\begin{proof}
    The first assertion is obvious from the preceding lemma, since $G_{\tilde 1^N}(g,n)$ is a quotient of $G_{1^N}(g,n)$.
    For the second assertion we use the resolution of the FA module $\tilde C_{1^N}$
    \begin{equation*}
\cdots \to C_{1^{N+2}}\to C_{1^{N+1}} \to C_{1^N} \to \tilde C_{1^N} \to 0, 
\end{equation*}
obtained by truncation of \eqref{equ:C long exact}.
From this we obtain a spectral sequence on the graph complexes converging to $G_{\tilde 1^N}(g,n)$ whose $E_1$-page is 
\begin{equation}
\label{equ:G 1N sseq}
\cdots \to H(G_{1^{N+2}})(g,n)\to H(G_{1^{N+1}})(g,n) \to H(G_{1^N})(g,n).
\end{equation}
(This spectral sequence already occurred in some form in \cite{PayneWillwacher23} in the special case $N=11$, $n=0$.)
By the preceding lemma we have $H^{k}(G_{1^{N}})(g,n)=0$ for $k > 3g+n-N-p$. Hence the only contribution to the top degree on the $E^1$-page comes from $H^{ 3g+n-N}(G_{1^{N+p}})(g,n)$. The remaining entries $H(G_{1^{N+p}})(g,n)$ on the $E^1$-page (for $p>0$) are concentrated in total cohomological degrees $\leq 3g+n-N-2$. Hence by degree reasons there is no further cancellation in the top degree on later pages of the spectral sequence, and the corollary follows.
\end{proof}

\subsubsection{In low degrees}
Here, we provide a complete description of the cohomology of $G_\lambda$ in degree $0$ and a nearly complete description in degree $1$.
\begin{prop}\label{prop:lowdegree}
Let $\lambda$ be a Young diagram with $N$ boxes.
\begin{enumerate}
    \item Suppose $\lambda$ has at least two columns. Then 
    \[
    H^0(G_\lambda)(g,n)=
    \begin{cases}
        V_\lambda & \text{if $g=0$ and $n=N$} \\
        0 &\text{otherwise}.
    \end{cases}
    \]
    Moreover, $H^1(G_\lambda)(g,n)=0$ for all pairs $(g,n)$ except possibly the following:
    \[
    (0,N+1), (0,N+2), (1,N-2), (1,N-1). %
    \]
    \item Furthermore, 
    \[
    H^0(G_{1^{N}})(g,n)=\begin{cases}
    V_{1^{N}} & \text{if $g=0$, $n=N$} \\
    V_{1^{N+1}} & \text{if $g=0$, $n=N+1$} \\
    0 & \text{otherwise}
    \end{cases}
    \]
    and
    \[
    H^1(G_{1^{N}})(g,n)
    =\begin{cases}
    V_{31^{N-2}} & \text{if $g=0$, $n=N+1$ and $N\geq 2$} \\
    V_{31^{N-1}} & \text{if $g=0$, $n=N+2$ and $N\geq 1$} \\
    0 & \text{otherwise.}
    \end{cases}
    \]
\end{enumerate}
\end{prop}

\begin{proof}
    To simplify the computation we use a trick of \cite{PayneWillwacher21}.
    In the blown-up picture, we call a blown-up component an $\epsilon$-component if it has no $\omega$-legs, and we call it an $\omega$-component otherwise. Let $G_\lambda^\star\subset G_\lambda$ denote the subcomplex spanned by graphs with no vertices in $\epsilon$-components, and at most one numbered leg in the union of the  $\epsilon$-components. Then the proof of \cite[Proposition 6.4]{PayneWillwacher21} shows that the inclusion $G_\lambda^\star\subset G_\lambda$ is a quasi-isomorphism. There are only a few graphs in $G_\lambda^\star$ of degrees $0$ and $1$, which we can easily list.
    
    To aid our computation let us also replace the representation $\lambda$ by the non-irreducible representation $\rho=1\hboxtimes \cdots \hboxtimes 1$, which contains all irreducible representations of $\ss_{N}$. Then the additional decoration in the representation that is normally not shown in the blown-up graph (see Section \ref{sec:blown up}) can be conveniently encoded combinatorially by numbering all the $\omega$-legs.
    Therefore, the graphs in $G_\rho$ have the $\omega$-legs labeled by, say, $\omega_1,\dots,\omega_{N}$. An additional decoration is then not necessary.
    Also note that $G_\rho(n)$ carries a representation of $\ss_n\times \ss_{N}$, with $\ss_n$ permuting the labels on the numbered legs as usual and $\ss_{N}$ permuting the labels on the $\omega$-legs. We may recover $G_\lambda$ as $G_\lambda=G_\rho \otimes_{\ss_{N}} V_\lambda$.
    
    We compute the degree $0$ and degree $1$ cohomology of $G_\lambda^\star(g,n)$ for each pair $g,n$ separately. For the pairs we omit, there are no graphs, so the cohomology is trivially zero.

\smallskip

    \underline{Case $(g,n)=(0,N)$:} There is only one contributing kind of graph
    \[
    \begin{tikzpicture}
    \node (v1) at (0,-.4) {$\omega_1$};
    \node (v2) at (0,.4) {$i_1$};
    \draw (v1) edge (v2);
    \end{tikzpicture}
    \cdots
    \begin{tikzpicture}
    \node (v1) at (0,-.4) {$\omega_{N}$};
    \node (v2) at (0,.4) {$i_{N}$};
    \draw (v1) edge (v2);
    \end{tikzpicture}
    \]
    and this contributes one copy of the representation $\rho$ in degree 0.

    \smallskip
    
    \underline{Case $(g,n)=(0,N+1)$:} Here the complex $G_\rho^\star(g,n)$ is spanned in degrees zero and one by two types of graphs 
    \begin{align*}
    A_{j;i_1,\dots,i_{N}} &:= \quad 
    \begin{tikzpicture}
    \node (v1) at (0,-.4) {$\omega_1$};
    \node (v2) at (0,.4) {$i_1$};
    \draw (v1) edge (v2);
    \end{tikzpicture}
    \cdots
    \begin{tikzpicture}
    \node (v1) at (0,-.4) {$\omega_{N}$};
    \node (v2) at (0,.4) {$i_{N}$};
    \draw (v1) edge (v2);
    \end{tikzpicture}
    \begin{tikzpicture}
    \node (v1) at (0,-.4) {$\epsilon$};
    \node (v2) at (0,.4) {$j$};
    \draw (v1) edge (v2);
    \end{tikzpicture}
    & \text{(in degree $0$)}
    \\
    B^k_{i,j;i_1,\dots,i_{N-1}} &:= \quad 
    \begin{tikzpicture}
    \node (v1) at (0,-.4) {$\omega_1$};
    \node (v2) at (0,.4) {$i_1$};
    \draw (v1) edge (v2);
    \end{tikzpicture}
    \cdots
    \begin{tikzpicture}
    \node (v1) at (0,-.4) {$\omega_{N}$};
    \node (v2) at (0,.4) {$i_{N-1}$};
    \draw (v1) edge (v2);
    \end{tikzpicture}
    \begin{tikzpicture}
    \node[int] (v1) at (0,.4) {};
    \node (v2) at (-.4,-.4) {$\omega_{k}$};
    \node (v3) at (0,-.4) {$i$};
    \node (v4) at (0.4,-.4) {$j$};
    \draw (v1) edge (v2) edge (v3) edge (v4);
    \end{tikzpicture} & \text{(in degree $1$)}
    \end{align*}
     with differential 
    \begin{align*}
    dA_{j;i_1,\dots,i_{N}} &=\sum_{\alpha=1}^{N} 
    B^\alpha_{j,i_\alpha;i_1,\dots,\hat \imath_\alpha,\dots ,i_{N}}
    &
    d B^k_{i,j;i_1,\dots,i_{N-1}}&= 0.
    \end{align*}
    To compute the degree $0$ cohomology consider the linear map $h_\beta$ such that 
    \[
    h_\beta(B^k_{i,j;i_1,\dots ,i_{N-1}})
    =
    \begin{cases}
        A_{j;i_1,\dots,i,\dots,i_{N-1}} & \text{if $k=\beta$} \\
        0 & \text{otherwise,}
    \end{cases}
    \]
    with the $i$ inserted at position $k=\beta$, for $i < j$.
    Then 
    \[
    h_\beta d(A_{j;i_1,\dots,i_{N}})
    =
    \begin{cases}
        A_{j;i_1,\dots,i_{N}} & \text{if $j>i_\beta$} \\
        A_{i_\beta;i_1,\dots, j,\dots,i_{N}} & \text{if $j<i_\beta$.}
    \end{cases}
    \]
    Now suppose that 
    \[
    x = \sum c_{j;i_1,\dots,i_{N}}A_{j;i_1,\dots,i_{N}}
    \]
    is in the kernel of $d$, i.e., $dx=0$. Then $h_\beta dx=0$ from which we learn that 
    \[
    c_{j;i_1,\dots,i_{N}}=-c_{i_\beta;i_1,\dots, j,\dots,i_{N}}.
    \]
    The only solution to this linear system is the fully antisymmetric tensor, i.e., possibly up to a multiplicative constant,
    \[
    x= \sum_{\sigma\in \ss_{N+1}}(-1)^\sigma 
    A_{\sigma(1);\sigma(2)\cdots \sigma(N+1)}.
    \]
    It follows that $H^0(G_{1^{N}})(0,N+1)=V_{1^{N+1}}$ and  $H^0(G_{\lambda})(0,N+1)=0$ for $\lambda$ with at least two columns. The cokernel of the differential $d$ is non-trivial and generally produces non-zero cohomology in degree 1 and $(g,n)=(0,N+1)$.
    Specifically, let us consider the case $\lambda=1^{N}$, corresponding to antisymmetrizing over the $\omega_1,\dots,\omega_{N}$ in our graphs.
    In this isotypical component the $A$-classes above span the $\ss_{N+1}$-representation $V_{1^{N+1}}\oplus V_{21^{N-1}}$, as can be seen from Pieri's rule. Similarly, the $B$-classes span the representation $ V_{21^{N-1}}\oplus V_{31^{N-2}}$. Since this differential is nonzero, we readily conclude that $H^1(G_{1^{N}})(0,N+1) \cong V_{31^{N-2}}$.

\smallskip

    \underline{Case $g=0$, $n\geq N+2$:}
    There are no degree 0 graphs, and two types of degree 1 graphs, with differential
    \begin{align}
    \label{equ:n02 first}
    \begin{tikzpicture}
    \node (v1) at (0,-.4) {$\omega_1$};
    \node (v2) at (0,.4) {$i_1$};
    \draw (v1) edge (v2);
    \end{tikzpicture}
    \cdots
    \begin{tikzpicture}
    \node (v1) at (0,-.4) {$\omega_{N}$};
    \node (v2) at (0,.4) {$i_{N}$};
    \draw (v1) edge (v2);
    \end{tikzpicture}
    \begin{tikzpicture}
    \node[int] (v1) at (0.2,.4) {};
    \node (v2) at (-.4,-.4) {$\omega_\alpha$};
    \node (v3) at (0,-.4) {$j_1$};
    \node (v4) at (0.4,-.4) {$\cdots$};
    \node (v5) at (0.8,-.4) {$j_r$};
    \draw (v1) edge (v2) edge (v3) edge (v4) edge (v5);
    \end{tikzpicture}
    &\mapsto 
    \sum
    \begin{tikzpicture}
    \node (v1) at (0,-.4) {$\omega_1$};
    \node (v2) at (0,.4) {$i_1$};
    \draw (v1) edge (v2);
    \end{tikzpicture}
    \cdots
    \begin{tikzpicture}
    \node (v1) at (0,-.4) {$\omega_{N}$};
    \node (v2) at (0,.4) {$i_{N}$};
    \draw (v1) edge (v2);
    \end{tikzpicture}
    \begin{tikzpicture}
    \node[int] (v1) at (0.2,.4) {};
    \node[int] (v1b) at (0.7,.4) {};
    \node (v2) at (-.4,-.4) {$\omega_\alpha$};
    \node (v3) at (0,-.4) {$j_{\bullet}$};
    \node (v4) at (0.9,-.4) {$\cdots$};
    \node (v5) at (1.3,-.4) {$j_{\bullet}$};
    \draw (v1) edge (v1b) edge (v2) edge (v3) 
    (v1b) edge (v4) edge (v5);
    \end{tikzpicture}
    \\
    \begin{tikzpicture}
    \node (v1) at (0,-.4) {$\omega_{1}$};
    \node (v2) at (0,.4) {$i_1$};
    \draw (v1) edge (v2);
    \end{tikzpicture}
    \cdots
    \begin{tikzpicture}
    \node (v1) at (0,-.4) {$\omega_{N}$};
    \node (v2) at (0,.4) {$i_{N}$};
    \draw (v1) edge (v2);
    \end{tikzpicture}
    \begin{tikzpicture}
    \node[int] (v1) at (0.2,.4) {};
    \node (v2) at (-.4,-.4) {$\omega_\alpha$};
    \node (v3) at (0,-.4) {$j_1$};
    \node (v4) at (0.4,-.4) {$\cdots$};
    \node (v5) at (0.8,-.4) {$j_r$};
    \draw (v1) edge (v2) edge (v3) edge (v4) edge (v5);
    \end{tikzpicture}
    \begin{tikzpicture}
    \node (v1) at (0,-.4) {$\epsilon$};
    \node (v2) at (0,.4) {$k$};
    \draw (v1) edge (v2);
    \end{tikzpicture}
    &\mapsto 
    \sum
    \begin{tikzpicture}
    \node (v1) at (0,-.4) {$\omega_1$};
    \node (v2) at (0,.4) {$i_1$};
    \draw (v1) edge (v2);
    \end{tikzpicture}
    \cdots
    \begin{tikzpicture}
    \node (v1) at (0,-.4) {$\omega_{N}$};
    \node (v2) at (0,.4) {$i_{N}$};
    \draw (v1) edge (v2);
    \end{tikzpicture}
    \begin{tikzpicture}
    \node[int] (v1) at (0.2,.4) {};
    \node[int] (v1b) at (0.7,.4) {};
    \node (v2) at (-.4,-.4) {$\omega_\alpha$};
    \node (v3) at (0,-.4) {$j_{\bullet}$};
    \node (v4) at (0.9,-.4) {$\cdots$};
    \node (v5) at (1.3,-.4) {$j_{\bullet}$};
    \draw (v1) edge (v1b) edge (v2) edge (v3) 
    (v1b) edge (v4) edge (v5);
    \end{tikzpicture}
    \begin{tikzpicture}
    \node (v1) at (0,-.4) {$\epsilon$};
    \node (v2) at (0,.4) {$k$};
    \draw (v1) edge (v2);
    \end{tikzpicture}
    +(\cdots).
    \label{equ:n02second}
    \end{align}
    where $(\cdots)$ are graphs without an $\epsilon$-leg.
    The (first) terms on the right are present for $r\geq 3$, and the sum is over partitions of the set $\{j_1,\dots,j_r\}$.
    These terms lead to the differential being injective since contracting the internal (horizontal) edge and dividing by the number of partitions is obviously a one-sided inverse.
    Hence the only degree 1 cohomology that can possibly remain for $g=0$ is in $n=N+2$, arising from the graph for $r=2$ on the left hand side of \eqref{equ:n02second}.
    Let us also compute this cohomology explicitly for the antisymmetric isotypical component of the representation of $\ss_{N}$ by permuting the $\omega$'s. In this case the remaining terms of the differential of the second generator \eqref{equ:n02second} become
    \begin{equation}\label{equ:n02 3}
    \begin{tikzpicture}
    \node (v1) at (0,-.4) {$\omega$};
    \node (v2) at (0,.4) {$i_1$};
    \draw (v1) edge (v2);
    \end{tikzpicture}
    \cdots
    \begin{tikzpicture}
    \node (v1) at (0,-.4) {$\omega$};
    \node (v2) at (0,.4) {$i_{N}$};
    \draw (v1) edge (v2);
    \end{tikzpicture}
    \begin{tikzpicture}
    \node[int] (v1) at (0,.4) {};
    \node (v2) at (-.4,-.4) {$\omega$};
    \node (v3) at (0,-.4) {$j_1$};
    \node (v4) at (0.4,-.4) {$j_2$};
    \draw (v1) edge (v2) edge (v3) edge (v4);
    \end{tikzpicture}
    \begin{tikzpicture}
    \node (v1) at (0,-.4) {$\epsilon$};
    \node (v2) at (0,.4) {$k$};
    \draw (v1) edge (v2);
    \end{tikzpicture}
    \mapsto 
    \begin{tikzpicture}
    \node (v1) at (0,-.4) {$\omega$};
    \node (v2) at (0,.4) {$i_1$};
    \draw (v1) edge (v2);
    \end{tikzpicture}
    \cdots
    \begin{tikzpicture}
    \node (v1) at (0,-.4) {$\omega$};
    \node (v2) at (0,.4) {$i_{N}$};
    \draw (v1) edge (v2);
    \end{tikzpicture}
    \begin{tikzpicture}
    \node[int] (v1) at (0.2,.4) {};
    \node[int] (v1b) at (0.7,.4) {};
    \node (v2) at (-.4,-.4) {$\omega$};
    \node (v3) at (0,-.4) {$k$};
    \node (v4) at (0.9,-.4) {$i$};
    \node (v5) at (1.3,-.4) {$j$};
    \draw (v1) edge (v1b) edge (v2) edge (v3) 
    (v1b) edge (v4) edge (v5);
    \end{tikzpicture}
    +
    \sum_{\beta}
    \pm
    \begin{tikzpicture}
    \node (v1) at (0,-.4) {$\omega$};
    \node (v2) at (0,.4) {$i_1$};
    \draw (v1) edge (v2);
    \end{tikzpicture}
    \cdots
    \begin{tikzpicture}
    \node (v1) at (0,-.4) {$\omega$};
    \node (v2) at (0,.4) {$i_{N}$};
    \draw (v1) edge (v2);
    \end{tikzpicture}
    \begin{tikzpicture}
    \node[int] (v1) at (0,.4) {};
    \node (v2) at (-.4,-.4) {$\omega$};
    \node (v3) at (0,-.4) {$i_\beta$};
    \node (v4) at (0.4,-.4) {$k$};
    \draw (v1) edge (v2) edge (v3) edge (v4);
    \end{tikzpicture}
    \begin{tikzpicture}
    \node[int] (v1) at (0,.4) {};
    \node (v2) at (-.4,-.4) {$\omega$};
    \node (v3) at (0,-.4) {$j_1$};
    \node (v4) at (0.4,-.4) {$j_2$};
    \draw (v1) edge (v2) edge (v3) edge (v4);
    \end{tikzpicture}
    .
    \end{equation}
    The composition with the projection onto the first term on the right-hand side is already a bijection. 
    However, note that the same graphs appear as the differential of the first generator \eqref{equ:n02 first}, so that a linear combination of both can be made to have no such graph in the differential. Concretely, such combinations span an $\ss_{N+2}$-module $V_{41^{N-2}}\oplus V_{31^{N-1}}$. The collection of graphs of the form appearing inside the sum in \eqref{equ:n02 3} spans a representation 
    \[
    V_{1^{N-2}}\boxtimes \ss_\boxtimes^2( V_2) 
    = V_{1^{N-2}}\boxtimes (V_4 \oplus V_{2^2})
    =
    V_{51^{N-3}}\oplus V_{41^{N-2}}\oplus V_{3^2 1^{N-4}}
    \oplus V_{3 2 1^{N-3}}\oplus V_{2^2 1^{N-2}}.
    \]
    Since the second piece of the differential in \eqref{equ:n02 3} is non-zero we can readily conclude that the $V_{41^{N-2}}$-irreps are sent isomorphically onto each other by the differential so that 
    \[
    H^1(V_{1^{N}})(0,N+2) \cong V_{31^{N-1}}.
    \]

\smallskip

    \underline{Case $(g,n)=(1,N-2)$:} Here there is one contributing type of graph of degree 1:
    \[
    \begin{tikzpicture}
    \node (v1) at (0,-.4) {$\omega_1$};
    \node (v2) at (0,.4) {$i_1$};
    \draw (v1) edge (v2);
    \end{tikzpicture}
    \cdots
    \begin{tikzpicture}
    \node (v1) at (0,-.4) {$\omega_{N}$};
    \node (v2) at (0,.4) {$i_{N-2}$};
    \draw (v1) edge (v2);
    \end{tikzpicture}
    \begin{tikzpicture}
    \node (v1) at (0,-.4) {$\omega_\alpha$};
    \node (v2) at (0,.4) {$\omega_{\beta}$};
    \draw (v1) edge (v2);
    \end{tikzpicture}.
    \]
    The differential vanishes and we generally obtain a contribution to the degree 1 cohomology. Note however that there is no $1^{N}$-isotypical component, due to the symmetry interchanging $\omega_{\alpha}$ and $\omega_{\beta}$. Hence $H^1(G_{1^{N}})(1,N-2)=0$.

    \smallskip
    
    \underline{Case $(g,n)=(1,N-1)$:}
    We have two types of degree 1 graphs, with differentials as shown:
    \begin{align*}
    \begin{tikzpicture}
    \node (v1) at (0,-.4) {$\omega_1$};
    \node (v2) at (0,.4) {$i_1$};
    \draw (v1) edge (v2);
    \end{tikzpicture}
    \cdots
    \begin{tikzpicture}
    \node (v1) at (0,-.4) {$\omega_{N}$};
    \node (v2) at (0,.4) {$i_{N-1}$};
    \draw (v1) edge (v2);
    \end{tikzpicture}
    \begin{tikzpicture}
    \node (v1) at (0,-.4) {$\omega_\alpha$};
    \node (v2) at (0,.4) {$\epsilon$};
    \draw (v1) edge (v2);
    \end{tikzpicture}
    &\mapsto
    \sum_\beta 
    \begin{tikzpicture}
    \node (v1) at (0,-.4) {$\omega_1$};
    \node (v2) at (0,.4) {$i_1$};
    \draw (v1) edge (v2);
    \end{tikzpicture}
    \cdots
    \begin{tikzpicture}
    \node (v1) at (0,-.4) {$\omega_{N}$};
    \node (v2) at (0,.4) {$i_{N-1}$};
    \draw (v1) edge (v2);
    \end{tikzpicture}
    \begin{tikzpicture}
    \node (v1) at (-.5,-.4) {$\omega_\alpha$};
    \node (v2) at (0,-.4) {$\omega_\beta$};
    \node (v3) at (0.5,-.4) {$i_\beta$};
    \node[int] (i) at (0,.4) {};
    \draw (i) edge (v1) edge (v2) edge (v3);
    \end{tikzpicture}
    \\
    \begin{tikzpicture}
    \node (v1) at (0,-.4) {$\omega_1$};
    \node (v2) at (0,.4) {$i_1$};
    \draw (v1) edge (v2);
    \end{tikzpicture}
    \cdots
    \begin{tikzpicture}
    \node (v1) at (0,-.4) {$\omega_{N}$};
    \node (v2) at (0,.4) {$i_{N-1}$};
    \draw (v1) edge (v2);
    \end{tikzpicture}
    \begin{tikzpicture}
    \node (v1) at (0,-.4) {$\omega_\alpha$};
    \node (v2) at (0,.4) {$\omega_{\beta}$};
    \draw (v1) edge (v2);
    \end{tikzpicture}
    \begin{tikzpicture}
    \node (v1) at (0,-.4) {$\epsilon$};
    \node (v2) at (0,.4) {$j$};
    \draw (v1) edge (v2);
    \end{tikzpicture}
    &\mapsto
    \sum_\gamma
    \begin{tikzpicture}
    \node (v1) at (0,-.4) {$\omega_1$};
    \node (v2) at (0,.4) {$i_1$};
    \draw (v1) edge (v2);
    \end{tikzpicture}
    \cdots
    \begin{tikzpicture}
    \node (v1) at (0,-.4) {$\omega_{N}$};
    \node (v2) at (0,.4) {$i_{N-1}$};
    \draw (v1) edge (v2);
    \end{tikzpicture}
    \begin{tikzpicture}
    \node (v1) at (-.5,-.4) {$\omega_\gamma$};
    \node (v2) at (0,-.4) {$i_\gamma$};
    \node (v3) at (0.5,-.4) {$j$};
    \node[int] (i) at (0,.4) {};
    \draw (i) edge (v1) edge (v2) edge (v3);
    \end{tikzpicture}
    \begin{tikzpicture}
    \node (v1) at (0,-.4) {$\omega_\alpha$};
    \node (v2) at (0,.4) {$\omega_{\beta}$};
    \draw (v1) edge (v2);
    \end{tikzpicture}.
    \end{align*}
    In general, these may produce non-trivial cohomology in $H^1(G_\lambda)(1,N-1)$. However, consider the special case $\lambda=1^{N}$, corresponding to antisymmetrizing over the $\omega_1,\dots,\omega_{N}$ in the pictures above. For this isotypical component the first map is an injection and the second generator vanishes by symmetry, hence $H^1(G_{1^{N}})(1,N-1)=0$.

    \smallskip

    \underline{Case $(g,n)=(1,N)$:}
    We have two types of degree 1 graphs, with differentials as shown:
    \begin{align*}
    \begin{tikzpicture}
    \node (v1) at (0,-.4) {$\omega_1$};
    \node (v2) at (0,.4) {$i_1$};
    \draw (v1) edge (v2);
    \end{tikzpicture}
    \cdots
    \begin{tikzpicture}
    \node (v1) at (0,-.4) {$\omega_{N}$};
    \node (v2) at (0,.4) {$i_{N}$};
    \draw (v1) edge (v2);
    \end{tikzpicture}
    \begin{tikzpicture}
    \node (v1) at (0,-.4) {$\epsilon$};
    \node (v2) at (0,.4) {$\epsilon$};
    \draw (v1) edge (v2);
    \end{tikzpicture}
    &\mapsto 
    \sum_\alpha
    \begin{tikzpicture}
    \node (v1) at (0,-.4) {$\omega_1$};
    \node (v2) at (0,.4) {$i_1$};
    \draw (v1) edge (v2);
    \end{tikzpicture}
    \cdots
    \begin{tikzpicture}
    \node (v1) at (0,-.4) {$\omega_{N}$};
    \node (v2) at (0,.4) {$i_{N}$};
    \draw (v1) edge (v2);
    \end{tikzpicture}
    \begin{tikzpicture}
    \node (v1) at (-.4,-.4) {$\omega_\alpha$};
    \node (v2) at (0,-.4) {$\epsilon$};
    \node (v3) at (0.4,-.4) {$i_\alpha$};
    \node[int] (i) at (0,.4) {};
    \draw (i) edge (v1) edge (v2) edge (v3);
    \end{tikzpicture}
    \\
    \begin{tikzpicture}
    \node (v1) at (0,-.4) {$\omega_1$};
    \node (v2) at (0,.4) {$i_1$};
    \draw (v1) edge (v2);
    \end{tikzpicture}
    \cdots
    \begin{tikzpicture}
    \node (v1) at (0,-.4) {$\omega_{N}$};
    \node (v2) at (0,.4) {$i_{N-1}$};
    \draw (v1) edge (v2);
    \end{tikzpicture}
    \begin{tikzpicture}
    \node (v1) at (0,-.4) {$\omega_\alpha$};
    \node (v2) at (0,.4) {$\epsilon$};
    \draw (v1) edge (v2);
    \end{tikzpicture}
    \begin{tikzpicture}
    \node (v1) at (0,-.4) {$\epsilon$};
    \node (v2) at (0,.4) {$j$};
    \draw (v1) edge (v2);
    \end{tikzpicture}
    &\mapsto
    \sum_\beta
    \begin{tikzpicture}
    \node (v1) at (0,-.4) {$\omega_1$};
    \node (v2) at (0,.4) {$i_1$};
    \draw (v1) edge (v2);
    \end{tikzpicture}
    \cdots
    \begin{tikzpicture}
    \node (v1) at (0,-.4) {$\omega_{N}$};
    \node (v2) at (0,.4) {$i_{N}$};
    \draw (v1) edge (v2);
    \end{tikzpicture}
    \begin{tikzpicture}
    \node (v1) at (0,-.4) {$\omega_\alpha$};
    \node (v2) at (0,.4) {$\epsilon$};
    \draw (v1) edge (v2);
    \end{tikzpicture}
    \begin{tikzpicture}
    \node (v1) at (-.4,-.4) {$\omega_\beta$};
    \node (v2) at (0,-.4) {$j$};
    \node (v3) at (0.4,-.4) {$i_\beta$};
    \node[int] (i) at (0,.4) {};
    \draw (i) edge (v1) edge (v2) edge (v3);
    \end{tikzpicture}
    +(\cdots)
    \end{align*}
    Here $(\cdots)$ denotes graphs not containing an $\epsilon-\omega$-edge. Clearly the differential on the first kind of graph is injective. 
    For the second kind of graph we may re-use the argument of Case $(0,N+1)$ above. This then shows that every closed element has the form 
    \[
    \sum_\alpha c_\alpha 
    \sum_{\sigma\in \ss_{N}}
    (-1)^\sigma
    \begin{tikzpicture}
    \node (v1) at (0,-.4) {$\omega_1$};
    \node (v2) at (0,.4) {$i_{\sigma(1)}$};
    \draw (v1) edge (v2);
    \end{tikzpicture}
    \cdots
    \begin{tikzpicture}
    \node (v1) at (0,-.4) {$\omega_{N}$};
    \node (v2) at (0,.4) {$i_{\sigma(N-1)}$};
    \draw (v1) edge (v2);
    \end{tikzpicture}
    \begin{tikzpicture}
    \node (v1) at (0,-.4) {$\omega_\alpha$};
    \node (v2) at (0,.4) {$\epsilon$};
    \draw (v1) edge (v2);
    \end{tikzpicture}
    \begin{tikzpicture}
    \node (v1) at (0,-.4) {$\epsilon$};
    \node (v2) at (0,.4) {$i_{\sigma(N)}$};
    \draw (v1) edge (v2);
    \end{tikzpicture}.
    \]
    Applying the differential produces (from the $(\cdots)$-terms before)
    \[
    \sum_\alpha c_\alpha 
    \sum_\beta
    \sum_{\sigma\in \ss_{N}}
    (-1)^\sigma
    \begin{tikzpicture}
    \node (v1) at (0,-.4) {$\omega_1$};
    \node (v2) at (0,.4) {$i_{\sigma(1)}$};
    \draw (v1) edge (v2);
    \end{tikzpicture}
    \cdots
    \begin{tikzpicture}
    \node (v1) at (0,-.4) {$\omega_{N}$};
    \node (v2) at (0,.4) {$i_{\sigma(N-1)}$};
    \draw (v1) edge (v2);
    \end{tikzpicture}
    \begin{tikzpicture}
    \node (v1) at (-.4,-.4) {$\omega_\alpha$};
    \node (v2) at (0.4,-.4) {$\omega_\beta$};
    \node (v3) at (0,1) {$i_{\sigma(n_\beta-1)}$};
    \node[int] (i) at (0,.2) {};
    \draw (i) edge (v1) edge (v2) edge (v3);
    \end{tikzpicture}
    \begin{tikzpicture}
    \node (v1) at (0,-.4) {$\epsilon$};
    \node (v2) at (0,.4) {$i_{\sigma(N)}$};
    \draw (v1) edge (v2);
    \end{tikzpicture}.
    \]
    The coefficient of every graph must be zero, and this leads to $(-1)^\alpha c_\alpha= -(-1)^\beta c_\beta$. This linear system has no nonzero solutions, so that we have $H^1(G_\lambda)(1,N)=0$.

    \smallskip
    
    \underline{Case $(g,n)=(1,N+1)$:}
    We have one type of degree 1 graph with the following differential:
    \begin{align*}
    \begin{tikzpicture}
    \node (v1) at (0,-.4) {$\omega_1$};
    \node (v2) at (0,.4) {$i_1$};
    \draw (v1) edge (v2);
    \end{tikzpicture}
    \cdots
    \begin{tikzpicture}
    \node (v1) at (0,-.4) {$\omega_{N}$};
    \node (v2) at (0,.4) {$i_{N+1}$};
    \draw (v1) edge (v2);
    \end{tikzpicture}
    \begin{tikzpicture}
    \node (v1) at (0,-.4) {$\epsilon$};
    \node (v2) at (0,.4) {$\epsilon$};
    \draw (v1) edge (v2);
    \end{tikzpicture}
    \begin{tikzpicture}
    \node (v1) at (0,-.4) {$\epsilon$};
    \node (v2) at (0,.4) {$j$};
    \draw (v1) edge (v2);
    \end{tikzpicture}
        &\mapsto
        2
        \sum_\alpha
        \begin{tikzpicture}
    \node (v1) at (0,-.4) {$\omega_1$};
    \node (v2) at (0,.4) {$i_1$};
    \draw (v1) edge (v2);
    \end{tikzpicture}
    \cdots
    \begin{tikzpicture}
    \node (v1) at (0,-.4) {$\omega_{N}$};
    \node (v2) at (0,.4) {$i_{N+1}$};
    \draw (v1) edge (v2);
    \end{tikzpicture}
    \begin{tikzpicture}
    \node (v1) at (-.4,-.4) {$\omega_\alpha$};
    \node (v2) at (0,-.4) {$\epsilon$};
    \node (v3) at (0.4,-.4) {$i_\alpha$};
    \node[int] (i) at (0,.4) {};
    \draw (i) edge (v1) edge (v2) edge (v3);
    \end{tikzpicture}
    \begin{tikzpicture}
    \node (v1) at (0,-.4) {$\epsilon$};
    \node (v2) at (0,.4) {$j$};
    \draw (v1) edge (v2);
    \end{tikzpicture}
    +(\cdots).
    \end{align*}
    Here $(\cdots)$ are linear combinations of other graphs. 
    The first term already makes the differential injective, and there is no degree 1 cohomology.

    There are no graphs of degrees 0 and 1 for $g\geq 2$ or for $g=1$ and $n> N+1$. Hence collecting the various non-trivial pieces of cohomology we have found the proposition follows.
\end{proof}

\begin{cor}
    We have that
    \begin{align*}
        H^0(G_{\tilde 1^{N}})(g,n)&=
        \begin{cases}
            V_{1^{N}} & \text{for $g=0$, $n={N}$} \\
            0 & \text{otherwise}
        \end{cases}
    \end{align*}
    and
        \begin{align*}
        H^1(G_{\tilde 1^{N}})(g,n)&=
        \begin{cases}
            V_{31^{N-2}} & \text{for $g=0$, $n={N+1}$, and $N\geq 2$} \\
            0 & \text{otherwise.}
        \end{cases}
    \end{align*}
\end{cor}
\begin{proof}
    We resolve $\tilde C_{1^{N}}$ by \eqref{equ:tilde C fin res} and take the corresponding spectral sequence for the graph complex whose $E^1$-page is the complex 
    \[
    H^\bullet(G_{1^{{N}-1}}) \to H^\bullet(G_{1^{{N}-2}}) \to \cdots
    \]
    and apply the previous proposition.
    Note that the differential here sends $H^k(G_{1^{{N}-1}})$ to $H^k(G_{1^{{N}-2}})$ and acts on the graph by making one $\omega$-decoration into an $\epsilon$. Looking at the explicit representatives found in the preceding proof, this differential maps the generator of $H^0(G_{1^{{N}-1}})(0,N-1)$ nontrivially to the generator of $H^0(G_{1^{{N}-2}})(0,N-1)$. The stated result for $H^0(G_{\tilde 1^{N}})$ follows. 
    Since the generator of $H^0(G_{1^{{N}-2}})(0,N-2)$ is sent nontrivially to the generator of $H^0(G_{1^{{N}-3}})(0,N-2)$
    there is no contribution from $H^\bullet(G_{1^{{N}-2}})$ to $H^1(G_{\tilde 1^{N}})$. 
    Hence all degree one cohomology comes from $H^1(G_{1^{{N}-1}})$. The generator of $H^1(G_{1^{{N}-1}})(0,N+1)$ is sent nontrivially into $H^1(G_{1^{{N}-2}})(0,N+1)$, as can again be seen from the explicit generators found in the preceding proof, and the result for $H^1(G_{\tilde 1^{N}})$ follows. 
\end{proof}

\subsubsection{In $n=0$}
In the absence of numbered external legs, the cohomology of $G_M$ can be fully expressed in terms of the cohomology of the Feynman transform of the commutative operad, or equivalently the weight zero part $W_0H_c^\bullet(\M_{g,n})$ of the compactly supported cohomology of $\M_{g,n}$.
We will follow the latter approach. 

To state the result, let us assemble the $W_0H_c^k(\M_{g,n})$ into a genus-graded symmetric sequence $W_0H_c^k(\M)$ such that 
\[
W_0H_c^k(\M)(g,n) =
\begin{cases}
    W_0H_c^k(\M_{g-n+1,n})\otimes \qq[-1]^{\otimes n} & \text{if $2(g-n+1)+n\geq 3$} \\
    \qq[-1] & \text{if $g=1$, $n=0,2$} \\
    0 & \text{otherwise}
\end{cases}.
\]
The $\ss_n$-action on $W_0H_c^k(\M_{g-n+1,n})\otimes \qq[-1]^{\otimes n}$ is the diagonal one, from the natural $\ss_n$-action on $W_0H_c^k(\M_{g-n+1,n})$ and the one on $\qq[-1]^{\otimes n}$ by permuting the factors with Koszul signs. In other words, the effect of the factor $\qq[-1]^{\otimes n}$ is the same as multiplication by a sign representation, and degree shift by $+n$.
In the cases $g=1$, $n=2$, the one-dimensional graded vector space $\qq[-1]$ is considered in the trivial $\ss_2$-representation. 
Also, mind the genus shifts. These shifts are necessary because  $W_0H_c^k(\M_{g,n})$ corresponds to a blown-up component of internal loop order $g$ with $n$ legs, but the overall contribution to the non-blown-up graph (with the legs fused to the special vertex) is to genus $g+n-1$.

We then define the genus graded symmetric sequence 
\[
M =  \mathrm{Exp}(W_0H_c^\bullet(\M))
= 
\bigoplus_k \big(W_0H_c^\bullet(\M)\hboxtimes \cdots \hboxtimes W_0H_c^\bullet(\M)\big)_{\ss_k}
\]
as the plethystic exponential of $W_0H_c^\bullet(\M)$, extending the genus grading additively. Here we use the product of symmetric sequences
\[
(A \hboxtimes B)(n) = \bigoplus_{r+s = n} \Ind_{\ss_r\times \ss_s}^{\ss_n}A(r) \boxtimes B(s),
\]
see also \cite{GetzlerKapranov} for more details on the plethystic exponential.

Furthermore, extend $M$ as follows:
\[
M' = (1 \oplus V_1[-1])\hboxtimes M,
\]
where we consider the $\ss_1$-representation $V_1$ as symmetric sequence concentrated in arity one, and having genus 1.

\begin{thm}\label{thm:n0 cohom}
Let $\lambda$ be a Young diagram with $n$ boxes.
     Then the graph complex $G_{\lambda}(g, 0)$ has cohomology
     \[
     H^\bullet(G_{\lambda}(g,0)) \cong 
     M'(g,n)
     \otimes_{\ss_n} V_\lambda.
 \]
\end{thm}
\begin{proof}
    We continue to work with the graph complex $G_\lambda^\star \subset G_\lambda$ quasi-isomorphic to $G_\lambda$ introduced in the proof of Proposition \ref{prop:lowdegree} above.
    We consider in particular the arity zero part $X_{\lambda,g}:=G_\lambda^\star(g,0)$. Recall that elements of this complex are graphs (in the blown-up picture) without numbered legs and with no vertices in any $\epsilon$-component. The differential on $G_\lambda^\star$ has three terms, $\delta=\delta_{split}+\delta_{join}^\epsilon+\delta_{join}^\omega$ with $\delta_{split}$ splitting internal vertices and $\delta_{join}^\epsilon$ (resp. $\delta_{join}^\omega$) fusing an arbitrary subset of $\epsilon$-legs (resp. together with one $\omega$-leg) into a new internal vertex with one $\epsilon$-leg (resp. $\omega$-leg).
    \begin{align*}
        \delta_{join}^\epsilon\colon 
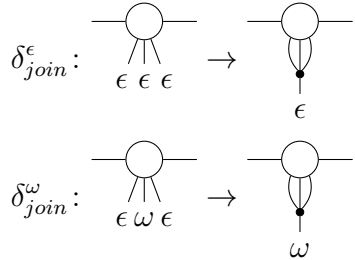
        \begin{tikzpicture}
            \node[ext] (v) at (0,.5) {$\quad$};
            \node (e1) at (-.3,-.3) {$\epsilon$};
            \node (e2) at (0,-.3) {$\epsilon$};
            \node (e3) at (0.3,-.3) {$\epsilon$};
            \draw (v) edge (e1) edge (e2) edge (e3) edge +(0.7, 0) edge +(-0.7, 0);
        \end{tikzpicture}
        & \to
        \begin{tikzpicture}
            \node[ext] (v) at (0,.5) {$\quad$};
            \node[int] (w) at (0,-.2) {};
            \node (e1) at (0,-.7) {$\epsilon$};
            \draw (v) edge[bend left] (w) edge[bend right] (w) edge (w) edge +(0.7, 0) edge +(-0.7, 0)
            (w) edge (e1);
        \end{tikzpicture}
        \\
        \delta_{join}^\omega\colon 
        \begin{tikzpicture}
            \node[ext] (v) at (0,.5) {$\quad$};
            \node (e1) at (-.3,-.3) {$\epsilon$};
            \node (e2) at (0,-.3) {$\omega$};
            \node (e3) at (0.3,-.3) {$\epsilon$};
            \draw (v) edge (e1) edge (e2) edge (e3) edge +(0.7, 0) edge +(-0.7, 0);
        \end{tikzpicture}
                & \to
        \begin{tikzpicture}
            \node[ext] (v) at (0,.5) {$\quad$};
            \node[int] (w) at (0,-.2) {};
            \node (e1) at (0,-.7) {$\omega$};
            \draw (v) edge[bend left] (w) edge[bend right] (w) edge (w) edge +(0.7, 0) edge +(-0.7, 0)
            (w) edge (e1);
        \end{tikzpicture}
    \end{align*}
    Now as in \cite[Section 6]{PayneWillwacher21} there is an isomorphism of dg vector spaces 
    \[
    \Xi \colon(X_{\lambda,g}, \delta_{split}+\delta_{join}^\epsilon)\to (X_{\lambda,g}, \delta_{split}+\delta_{join}^\epsilon+\delta_{join}^\omega).
    \]
    In other words, we may simply drop the term $\delta_{join}^\omega$ from the differential.
    We continue following analogous arguments in \cite[Section 7]{PayneWillwacher21} and define the subcomplex 
    \begin{equation}\label{equ:Hg def}
    (H_{\lambda,g},\delta_{split}) \subset (X_{\lambda,g}, \delta_{split}+\delta_{join}^\epsilon)
    \end{equation}
    spanned by graphs that have either no $\epsilon$-legs, exactly one $\epsilon$-leg, two $\epsilon$-legs forming an $\epsilon$-$\epsilon$-edge, or three $\epsilon$-legs, two of which form an $\epsilon$-$\epsilon$-edge.
    \[
    \underbrace
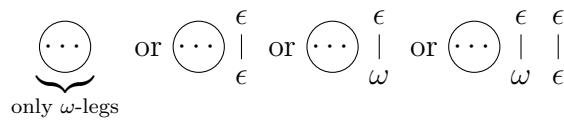
{\begin{tikzpicture}
        \node[ext] (v) at (0,0) {$\cdots$};
    \end{tikzpicture}}_{\text{only $\omega$-legs}}
    \text{  or  }
    \begin{tikzpicture}
        \node[ext] (v) at (0,0) {$\cdots$};
    \end{tikzpicture}
    \begin{tikzpicture}
        \node (e1) at (0,-.4) {$\epsilon$};
        \node (e2) at (0,0.4) {$\epsilon$};
        \draw (e1) edge (e2);
    \end{tikzpicture}
        \text{  or  }
    \begin{tikzpicture}
        \node[ext] (v) at (0,0) {$\cdots$};
    \end{tikzpicture}
    \begin{tikzpicture}
        \node (e1) at (0,-.4) {$\omega$};
        \node (e2) at (0,0.4) {$\epsilon$};
        \draw (e1) edge (e2);
    \end{tikzpicture}
            \text{  or  }
    \begin{tikzpicture}
        \node[ext] (v) at (0,0) {$\cdots$};
    \end{tikzpicture}
    \begin{tikzpicture}
        \node (e1) at (0,-.4) {$\omega$};
        \node (e2) at (0,.4) {$\epsilon$};
        \draw (e1) edge (e2);
    \end{tikzpicture}
        \begin{tikzpicture}
        \node (e1) at (0,-.4) {$\epsilon$};
        \node (e2) at (0,0.4) {$\epsilon$};
        \draw (e1) edge (e2);
    \end{tikzpicture}
    \]
    By the same argument as in \cite[proof of Proposition 7.2]{PayneWillwacher21} we conclude that \eqref{equ:Hg def} is a quasi-isomorphism, so that $H^\bullet(H_{\lambda,g})\cong H^\bullet(G_\lambda(g,0))$.
    The differential on $H_{\lambda,g}$ only acts by splitting internal vertices, and there is no interaction with the external legs.
    Let temporarily $G(g,n)$ be the commutative graph complex computing 
    \[
    W_0 H_c^\bullet(\M_{g,n}) \cong H^\bullet(G(g,n))
    \]
    for $g,n$ such that $2g+n\geq 3$. We define a variant
    \begin{align*}
        G'(g,n) = 
        \begin{cases}
            G(g-n+1,n)\otimes \qq[-1]^{\otimes n} & \text{if $2(g-n+1)+n\geq 3$} \\
            \qq[-1] & \text{if $g=1$, $n=0,2$} \\
            0 & \text{otherwise.}
        \end{cases}
    \end{align*}
    Define a genus graded symmetric sequence $\widehat M'$ such that (cf. the definition of $M'$ above)
    \begin{align*}
        \widehat M &:= \mathrm{Exp}(G')\\
        \widehat M' &:= (1 \oplus V_1[-1])\hboxtimes \widehat M,
    \end{align*}
    so that $H^\bullet(\widehat M')=M'$.
    Then we have the isomorphism of complexes 
\begin{equation}\label{equ:H lambda g id}
    H_{\lambda,g} \cong \widehat M'(g,n)\otimes_{\ss_{n}} V_\lambda.
    \end{equation}
    In more detail, to see that \eqref{equ:H lambda g id} holds, let us discuss the combinatorial meaning of the construction of $\widehat M'$.
    \begin{enumerate}
        \item The construction of $G'$ from $G$ is such that we (i) replace the genus grading by the genus obtained after gluing all hairs to one vertex, (ii) we add a generator in $g=1$, $n=2$ that corresponds to the $(\omega-\omega)$-edge and (iii) we add another generator in $g=0$, $n=1$ that represents an $(\epsilon-\epsilon)$-edge. Hence elements of $G'(g,n)$ can be seen as linear combinations of connected graphs with $n$ hairs and (alternatively counted) genus $g$, allowing the two special one-edge graphs we added. 
        \item Taking the plethytistic exponential produces from the dg vector spaces of connected graphs $G'$ the symmetric sequence $\widehat M$ whose elements can be seen as linear combinations of arbitrary, possible disconneted graphs.
        \item Finally, $\widehat M'$ is obtained from $\widehat M$ by allowing in addition at most one $(\omega-\epsilon)$-leg in the graph, as is represented by the factor $V_1$, in degree $1$ and genus 1.
        \item Elements of $\widehat M'(g,n)$ are linear combinations of graphs with $n$ hairs. We pass to graphs with no hairs but one special vertex with decoration in $V_\lambda$ (to which we connect the hairs) by taking the tensor product $\widehat M'(g,n)\otimes_{\ss_{n}} V_\lambda$, and arrive at \eqref{equ:H lambda g id}.
    \end{enumerate}
    From \eqref{equ:H lambda g id} we conclude the desired result 
    \[
    H^\bullet(G_\lambda(g,0)) = H^\bullet(H_{\lambda,g}) \cong 
    H^\bullet(\widehat M'(g,n))\otimes_{\ss_{n}} V_\lambda
    =
    M'(g,n)\otimes_{\ss_{n}} V_\lambda. \qedhere
    \]
\end{proof}

\begin{rem}
    Note that Theorem \ref{thm:n0 cohom} does not cover the case of $G_{\tilde 1^N}(g,0)$ a priori. However, by Corollary \ref{cor:G tilde G top} we at least know that the top cohomology of $G_{\tilde 1^N}(g,0)$ is the same as that of $G_{1^N}(g,0)$ and can hence be evaluated by Theorem \ref{thm:n0 cohom}. Furthermore, as in the proof of Corollary \ref{cor:G tilde G top}, we have a spectral sequence converging to $H(G_{\tilde 1^N})(g,0)$ whose $E^1$-page \eqref{equ:G 1N sseq} can be evaluated using Theorem \ref{thm:n0 cohom}. The further cancellations on higher pages of this spectral sequence are complicated and not well understood, see the discussion in \cite{PayneWillwacher23}.
\end{rem}

\subsection{Cohomology bound and cohomology in low excess}

Let $\lambda$ be a Young diagram with $N$ boxes.
Then we define the number $r_\lambda$ as the largest integer $\leq N/2$ such that 
\[
V_\lambda \otimes_{\ss_{r_\lambda}\wr \ss_2} (V_2^{\otimes r_\lambda}\otimes \mathrm{sgn}_{r_\lambda})
\neq 0.
\]
Equivalently, this is the largest $r_\lambda$ such that the decomposition of the plethysm 
\begin{equation}\label{equ:Vplet}
V_{1^{r_\lambda}} \circ V_2
\end{equation}
into irreducibles contains a $V_{\lambda'}$ such that $\lambda'\subset \lambda$.
\begin{example}\leavevmode
    \begin{itemize}
        \item If $\lambda=1^k$ then $r_\lambda =0$ as any $\lambda'$ appearing in the decomposition of \eqref{equ:Vplet} has at least two columns.
        \item If $\lambda = p1^q$ is a hook shape with $p\geq 2$ then $r_\lambda=[(p+q)/2]$.
        \item If $\lambda$ has exactly 2 columns (i.e., $\lambda_1=2$) then $r_\lambda=1$. To see this, note that obviously $r_\lambda \geq 1$. Furthermore, for $r\geq 2$,
        $$
        V_{1^{r}} \circ V_2\subset (V_{1^{2}} \circ V_2)\hboxtimes (V_{1^{r-2}} \circ V_2).
        $$
        However, since $V_{1^{2}} \circ V_2=V_{31}$ already has three columns, any irreducible representation $V_{\lambda'}$ occurring in $V_{1^{r}} \circ V_2$ must have $\lambda_1'\geq 3$ columns as well. Hence we conclude that $r_\lambda<2$ and thus $r_\lambda=1$.
        In particular we have $r_{2^k}=r_{2^51^6}=1$.
    \end{itemize}
\end{example}

\begin{prop}\label{prop:lower bound}
Let $\lambda$ be a Young diagram with $N$ boxes. Then 
\[
H^\bullet(G_\lambda)(g,n) = 0
\]
as long as 
\[
3g+2n + \underbrace{\min \{r_\lambda,g\}}_{:=r_g} < 2N
\]
Furthermore, if $3g+2n + r_g = 2N$ then $H^\bullet(G_\lambda)(g,n)$ is either zero or concentrated in a single cohomological degree, represented by graphs of the form 
\[
\begin{tikzpicture}
\node (v1) at (0,-.4) {$\omega$};
\node (v2) at (0,.4) {$1$};
\draw (v1) edge (v2);
\end{tikzpicture}
\cdots 
\begin{tikzpicture}
\node (v1) at (0,-.4) {$\omega$};
\node (v2) at (0,.4) {$n$};
\draw (v1) edge (v2);
\end{tikzpicture}
\underbrace{
\begin{tikzpicture}
\node (v1) at (0,-.4) {$\omega$};
\node (v2) at (0,.4) {$\omega$};
\draw (v1) edge (v2);
\end{tikzpicture}
\cdots 
\begin{tikzpicture}
\node (v1) at (0,-.4) {$\omega$};
\node (v2) at (0,.4) {$\omega$};
\draw (v1) edge (v2);
\end{tikzpicture}
}_{r_g}
\underbrace{
\begin{tikzpicture}
\node[int] (v1) at (0,.4) {};
\node (v2) at (-.4,-.4) {$\omega$};
\node (v3) at (0,-.4) {$\omega$};
\node (v4) at (0.4,-.4) {$\omega$};
\draw (v1) edge (v2) edge (v3) edge (v4);
\end{tikzpicture}
\cdots 
\begin{tikzpicture}
\node[int] (v1) at (0,.4) {};
\node (v2) at (-.4,-.4) {$\omega$};
\node (v3) at (0,-.4) {$\omega$};
\node (v4) at (0.4,-.4) {$\omega$};
\draw (v1) edge (v2) edge (v3) edge (v4);
\end{tikzpicture}
}_{\max \{0,(g-r_g)/2\}}
\]
\end{prop}
\begin{proof}
We define the excess of a blown-up component $C_i$ of loop order $h_i$ with $n_i$ numbered legs, $n_\omega$ $\omega$-legs and $n_\epsilon$ $\epsilon$-legs to be 
\begin{equation}\label{equ:exc defi}
E(C_i) = 3(h_i-1) +2n_i + 3n_\epsilon + n_\omega.
\end{equation}
Then for any graph $\Gamma=C_1\cup \dots \cup C_k$ in $G_\lambda(g,n)$ we have
\[
\sum_{i=1}^k E(C_i) =
3\sum_{i=1}^k(h_i - 1 + n_\epsilon + n_\omega) + 2 \sum_{i=1}^k n_i - 2\sum_{i=1}^k n_\omega=
3g + 2n -2N.
\]
For example, the excess is $-1$ for the blown-up component $\begin{tikzpicture}
\node (v1) at (-.4,0) {$\omega$};
\node (v2) at (.4,0) {$\omega$};
\draw (v1) edge (v2);
\end{tikzpicture}$ and zero for the blown up components 
\[
\begin{tikzpicture}
\node (v1) at (0,-.4) {$\omega$};
\node (v2) at (0,.4) {$i$};
\draw (v1) edge (v2);
\end{tikzpicture}
\text{ and } 
\begin{tikzpicture}
\node[int] (v1) at (0,.4) {};
\node (v2) at (-.4,-.4) {$\omega$};
\node (v3) at (0,-.4) {$\omega$};
\node (v4) at (0.4,-.4) {$\omega$};
\draw (v1) edge (v2) edge (v3) edge (v4);
\end{tikzpicture}.
\]
We claim that the excess of any other blown-up component is $\geq 1$. To see this, note that a blown-up component $C$ of excess $E(C)\leq 0$ necessarily has loop order $h=0$, since it must have at least one $\epsilon$- or $\omega$-leg.
So $C$ is a tree. If $C$ is a single edge and not of the forms above, it is one of 
\begin{align*}
&
\begin{tikzpicture}
\node (v1) at (0,-.4) {$\epsilon$};
\node (v2) at (0,.4) {$i$};
\draw (v1) edge (v2);
\end{tikzpicture},
\quad
\begin{tikzpicture}
\node (v1) at (0,-.4) {$\omega$};
\node (v2) at (0,.4) {$\epsilon$};
\draw (v1) edge (v2);
\end{tikzpicture},
\quad
\begin{tikzpicture}
\node (v1) at (0,-.4) {$\epsilon$};
\node (v2) at (0,.4) {$\epsilon$};
\draw (v1) edge (v2);
\end{tikzpicture},
\end{align*}
each of which have excess $\geq 1$. Hence $C$ must have at least one vertex, and hence at least 3 legs. But then it is clear from the definition \eqref{equ:exc defi} that $C$ can only be of excess zero if $C$ has exactly three legs which are all $\omega$-legs.

The number $r_\lambda$ is defined to be the largest number of components $\begin{tikzpicture}
\node (v1) at (-.4,0) {$\omega$};
\node (v2) at (.4,0) {$\omega$};
\draw (v1) edge (v2);
\end{tikzpicture}$ that can be present on the grounds of symmetry. This is for a generic pair $(g,n)$. For specific $g,n$ there might well be fewer, in particular at most $g$.
This readily yields the bound and cohomology representative of the proposition.
\end{proof}

\begin{rem}
Let us combine the above analysis with Theorem \ref{thm:n0 cohom}.
The weight zero cohomology $W_0H_c^\bullet(\M_{g,n})$ is known completely for $g=0,1$ and has furthermore been computed for all $g,n$ such that $3(g-1)+n\leq 12$ \cite{BrunWillwacher, BCGH}.
This means we know the blown-up components up to excess 12.
By the analysis in the preceding proof, the maximum excess of a blown-up component appearing in $G_\lambda(g,n)$ is 
$3g+2n+r_\lambda-2|\lambda|$. Hence we can compute $H^\bullet(G_\lambda)(g,n)$ from existing knowledge of $W_0H_c^\bullet(\M_{g,n})$ as long as $3g+2n+r_\lambda-2|\lambda| \leq 12$.  
\end{rem}

\subsection{Special case \texorpdfstring{$n=0$}{n0} and \texorpdfstring{$G_{\tilde 1^k}$}{Gtilde1k}}
Note that the complex $G_{\tilde 1^k}$ is a priori not covered by Theorem \ref{thm:n0 cohom}. However, by explicit computation similar to \cite{PayneWillwacher23} we find the following result:
\begin{prop}\label{prop:n0 tilde 1k}
    We have that $H^\bullet(G_{\tilde 1^k}(g,0))=0$ as long as $3g<2k$. Furthermore, 
    \[
    H^j(G_{\tilde 1^k}(\lceil 2 k/3 \rceil,0))
        =
        \begin{cases}
            \qq & \text{if $j=k$} \\
            0 &\text{otherwise}
        \end{cases}
    \]
    
\end{prop}
\begin{proof}
In each case we collect the contributing graphs. As in  \cite[Section 4]{PayneWillwacher23}, we only need to consider blown-up components of low excess.
We distinguish three cases.

    Case $k\equiv 0 \mod 3$: This is excess 0, and the only contributing graph is 
    \[
    \underbrace
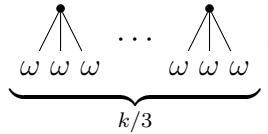
{
    \begin{tikzpicture}
\node[int] (v1) at (0,.4) {};
\node (v2) at (-.4,-.4) {$\omega$};
\node (v3) at (0,-.4) {$\omega$};
\node (v4) at (0.4,-.4) {$\omega$};
\draw (v1) edge (v2) edge (v3) edge (v4);
\end{tikzpicture}
\cdots 
\begin{tikzpicture}
\node[int] (v1) at (0,.4) {};
\node (v2) at (-.4,-.4) {$\omega$};
\node (v3) at (0,-.4) {$\omega$};
\node (v4) at (0.4,-.4) {$\omega$};
\draw (v1) edge (v2) edge (v3) edge (v4);
\end{tikzpicture}
}_{k/3}.
    \]
    This graph lives in degree $k$.

    Case $k\equiv 1 \mod 3$: This is excess 1, so the contributing graphs have exactly one blown-up component in excess 1.
    There are hence three contributing graphs:
    \begin{align*}
        A&:=
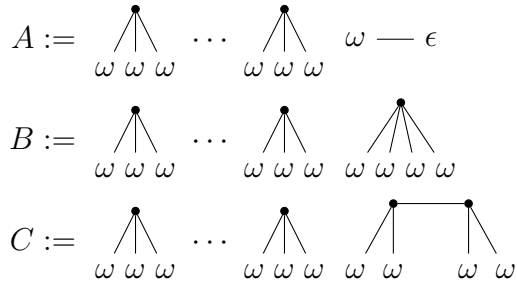
\begin{tikzpicture}
\node[int] (v1) at (0,.4) {};
\node (v2) at (-.4,-.4) {$\omega$};
\node (v3) at (0,-.4) {$\omega$};
\node (v4) at (0.4,-.4) {$\omega$};
\draw (v1) edge (v2) edge (v3) edge (v4);
\end{tikzpicture}
\cdots 
\begin{tikzpicture}
\node[int] (v1) at (0,.4) {};
\node (v2) at (-.4,-.4) {$\omega$};
\node (v3) at (0,-.4) {$\omega$};
\node (v4) at (0.4,-.4) {$\omega$};
\draw (v1) edge (v2) edge (v3) edge (v4);
\end{tikzpicture}
  \begin{tikzpicture}
    \node (v) at (0,0) {$\omega$};
    \node (w) at (1,0) {$\epsilon$};
    \draw (v) edge (w);
  \end{tikzpicture}
  \\
B&:= \begin{tikzpicture}
\node[int] (v1) at (0,.4) {};
\node (v2) at (-.4,-.4) {$\omega$};
\node (v3) at (0,-.4) {$\omega$};
\node (v4) at (0.4,-.4) {$\omega$};
\draw (v1) edge (v2) edge (v3) edge (v4);
\end{tikzpicture}
\cdots 
\begin{tikzpicture}
\node[int] (v1) at (0,.4) {};
\node (v2) at (-.4,-.4) {$\omega$};
\node (v3) at (0,-.4) {$\omega$};
\node (v4) at (0.4,-.4) {$\omega$};
\draw (v1) edge (v2) edge (v3) edge (v4);
\end{tikzpicture}
\begin{tikzpicture}
  \node[int] (i) at (0,.5) {};
  \node (v1) at (-.6,-.4) {$\omega$};
  \node (v2) at (-.2,-.4) {$\omega$};
  \node (v3) at (.2,-.4) {$\omega$};
  \node (v4) at (.6,-.4) {$\omega$};
\draw (i) edge (v1) edge (v2) edge (v3) edge (v4);
\end{tikzpicture}
  \\
C&:= \begin{tikzpicture}
\node[int] (v1) at (0,.4) {};
\node (v2) at (-.4,-.4) {$\omega$};
\node (v3) at (0,-.4) {$\omega$};
\node (v4) at (0.4,-.4) {$\omega$};
\draw (v1) edge (v2) edge (v3) edge (v4);
\end{tikzpicture}
\cdots 
\begin{tikzpicture}
\node[int] (v1) at (0,.4) {};
\node (v2) at (-.4,-.4) {$\omega$};
\node (v3) at (0,-.4) {$\omega$};
\node (v4) at (0.4,-.4) {$\omega$};
\draw (v1) edge (v2) edge (v3) edge (v4);
\end{tikzpicture}
\begin{tikzpicture}
  \node[int] (i) at (0,.5) {};
  \node[int] (j) at (1,.5) {};
  \node (v1) at (-.5,-.4) {$\omega$};
  \node (v2) at (0,-.4) {$\omega$};
  \node (v3) at (1,-.4) {$\omega$};
  \node (v4) at (1.5,-.4) {$\omega$};
\draw (i) edge (v1) edge (v2) edge (j) (j) edge (v3) edge (v4);
\end{tikzpicture}
\end{align*}
Graphs $A$ and $B$ live in cohomological degree $k$ while graph $C$ lives in degree $k+1$. Graph $C$ is hit by the differential, as is easy to see, so that the cohomology is again 1-dimensional concentrated in degree $k$.

    Case $k\equiv 2 \mod 3$: This is excess 2, so we can have exactly one blown-up component in excess 2, or two blown-up components in excess 1. 
    There are a priori four possibilities for a graph with a blown up component of excess 2. Three of these possibilities have five $\omega$-legs and are not drawn, because one of them is zero by symmetry, and the other two cancel under the differential. The fourth is drawn as $B$ below. This is however zero by the relations in the $\ss_k$-module $V_{1^k}$. Namely, the relation is obtained graphically by starting with a graph with $k+1$ many $\omega$-legs (such as graph $A$ drawn below) and then summing over all ways of making one $\omega$-leg into an $\epsilon$-leg. This produces from $A$ a non-zero multiple of $B$, hence $B=0$.
    
    Next, there are six possibilities for a graph with two blown up components of excess 1. The possible excess one components are seen pictured above as the last component of the graphs in the previous case. All are seen to cancel each other under the differential except graph $C$ below. 

        \begin{align*}
            A&:=\begin{tikzpicture}
\node[int] (v1) at (0,.4) {};
\node (v2) at (-.4,-.4) {$\omega$};
\node (v3) at (0,-.4) {$\omega$};
\node (v4) at (0.4,-.4) {$\omega$};
\draw (v1) edge (v2) edge (v3) edge (v4);
\end{tikzpicture}
\cdots 
\begin{tikzpicture}
\node[int] (v1) at (0,.4) {};
\node (v2) at (-.4,-.4) {$\omega$};
\node (v3) at (0,-.4) {$\omega$};
\node (v4) at (0.4,-.4) {$\omega$};
\draw (v1) edge (v2) edge (v3) edge (v4);
\end{tikzpicture}
\\
B&:=
\begin{tikzpicture}
\node[int] (v1) at (0,.4) {};
\node (v2) at (-.4,-.4) {$\omega$};
\node (v3) at (0,-.4) {$\omega$};
\node (v4) at (0.4,-.4) {$\omega$};
\draw (v1) edge (v2) edge (v3) edge (v4);
\end{tikzpicture}
\cdots 
\begin{tikzpicture}
\node[int] (v1) at (0,.4) {};
\node (v2) at (-.4,-.4) {$\omega$};
\node (v3) at (0,-.4) {$\omega$};
\node (v4) at (0.4,-.4) {$\omega$};
\draw (v1) edge (v2) edge (v3) edge (v4);
\end{tikzpicture}
\begin{tikzpicture}
\node[int] (v1) at (0,.4) {};
\node (v2) at (-.4,-.4) {$\omega$};
\node (v3) at (0,-.4) {$\omega$};
\node (v4) at (0.4,-.4) {$\epsilon$};
\draw (v1) edge (v2) edge (v3) edge (v4);
\end{tikzpicture}
\\
C&:=
\begin{tikzpicture}
\node[int] (v1) at (0,.4) {};
\node (v2) at (-.4,-.4) {$\omega$};
\node (v3) at (0,-.4) {$\omega$};
\node (v4) at (0.4,-.4) {$\omega$};
\draw (v1) edge (v2) edge (v3) edge (v4);
\end{tikzpicture}
\cdots 
\begin{tikzpicture}
\node[int] (v1) at (0,.4) {};
\node (v2) at (-.4,-.4) {$\omega$};
\node (v3) at (0,-.4) {$\omega$};
\node (v4) at (0.4,-.4) {$\omega$};
\draw (v1) edge (v2) edge (v3) edge (v4);
\end{tikzpicture}
  \begin{tikzpicture}
    \node (v) at (0,-0.3) {$\omega$};
    \node (w) at (1,-0.3) {$\epsilon$};
        \node (v1) at (0,0.3) {$\omega$};
    \node (w1) at (1,0.3) {$\epsilon$};
    \draw (v) edge (w) (v1) edge (w1);
  \end{tikzpicture}.
\end{align*}
The only contribution to the cohomology is hence from graph $C$, which is of degree $k$.
\end{proof}

\section{Hodge weight (17,0) and (19,0) %
cohomology of \texorpdfstring{$\M_{g,n}$}{Mgn}} \label{sec:app}

We can use the graph complexes of the previous section to study the Hodge weight $(17,0)$ and $(19,0)$ parts of the compactly supported cohomology of $\M_{g,n}$, as follows. 

\subsection{Hodge weight (17,0)} Recall the modular cooperad $\T$ of Section \ref{sec:1719} as well as the fact that the weight 17 part of its Feynman transform $\F \T^{17}$ computes $\gr_{17,0}H_c^{\bullet}(\M_{g,n})$.
By Remark \ref{rem:GC relation} the cohomology of $\F \T^{17}$ can be expressed through graph complexes $G_\lambda$, as
\[
\gr_{17,0}H_c^{k}(\M_{g,n})
\cong 
H^{k-17}(G_{\tilde 1^{17}}(g-1,n))
\oplus H^{k-17}(G_{2^7}(g-2,n)).
\]

Corollaries \ref{cor:EC1} and \ref{cor:EC2} then immediately yield:
\begin{thm} \label{thm:ec17}
    The $\ss_n$-equivariant Euler characteristic of the Hodge weight $(17,0)$ compactly supported cohomology of $\M_{g,n}$ satisfies 
    \begin{multline*}
\sum_{g,n\geq 0} u^{g+n} \chi_{\ss_n}(\mathrm{gr}_{17,0}H_c^\bullet(\M_{g,n}) )
  \\=
\   -u T_{\leq 16}\bigg(
  \prod_{\ell\geq 1} 
  \frac { 
    U_\ell(\frac 1 \ell \sum_{d\mid \ell} \mu(\ell/d) (-p_d  +1-w^d ), u )
  }
  { 
    U_\ell(\frac 1 \ell \sum_{d\mid \ell} \mu(\ell/d) (-p_d), u )
  }-1\bigg)
  \\
  \  -u^2(T_{w_1^7w_2^7} -T_{w_1^{8}w_2^{6}})\bigg(
  \prod_{\ell\geq 1} 
  \frac { 
    U_\ell(\frac 1 \ell \sum_{d\mid \ell} \mu(\ell/d) (-p_d  +1-w_1^d-w_2^d ), u )
  }
  { 
    U_\ell(\frac 1 \ell \sum_{d\mid \ell} \mu(\ell/d) (-p_d), u )
  }-1\bigg) \, .
\end{multline*}
\end{thm}
\begin{figure}
    \centering
    \resizebox{\textwidth}{!}{
\begin{tabular}{|g|M|M|M|M|M|M|M|} \hline \rowcolor{Gray} g,n 
& 0 
& 1 
& 2 
& 3 
& 4 
& 5 
& 6 
\\ 
\hline
 8 
& $ 0 $ 
& $ 0 $ 
& $ 0 $ 
& $ 0 $ 
& $ 0 $ 
& $ 0 $ 
& $ 0 $ 
 \\  
\hline
 9 
& $ 0 $ 
& $ 0 $ 
& $ 0 $ 
& $ 0 $ 
& $ 0 $ 
& $ -s_{1,1,1,1,1} $ 
& $ -s_{1,1,1,1,1,1} + 2 s_{3,1,1,1} $ 
 \\  
\hline
 10 
& $ 0 $ 
& $ 0 $ 
& $ 0 $ 
& $ 0 $ 
& $ -s_{2,1,1} $ 
& $ -s_{2,1,1,1} + s_{3,1,1} + s_{3,2} + 2 s_{4,1} $ 
& $ -4 s_{1,1,1,1,1,1} - 2 s_{2,1,1,1,1} - 3 s_{2,2,1,1} - 3 s_{2,2,2} + s_{3,1,1,1} - 2 s_{3,2,1} - 2 s_{3,3} + 2 s_{4,1,1} - 3 s_{4,2} - 2 s_{5,1} - 3 s_{6} $ 
 \\  
\hline
 11 
& $ 0 $ 
& $ 0 $ 
& $ s_{1,1} $ 
& $ s_{1,1,1} - 2 s_{3} $ 
& $ 2 s_{2,2} - 2 s_{3,1} $ 
& $ -s_{1,1,1,1,1} - 7 s_{2,1,1,1} - 6 s_{3,1,1} + 2 s_{3,2} + s_{4,1} + 4 s_{5} $ 
& $ -3 s_{1,1,1,1,1,1} - 15 s_{2,1,1,1,1} + 3 s_{2,2,1,1} - 5 s_{2,2,2} + 2 s_{3,1,1,1} + 11 s_{3,2,1} + 10 s_{3,3} + 21 s_{4,1,1} + 3 s_{4,2} + 6 s_{5,1} $ 
 \\  
\hline
 12 
& $ 0 $ 
& $ 0 $ 
& $ s_{2} $ 
& $ 4 s_{1,1,1} $ 
& $ 10 s_{1,1,1,1} + 2 s_{2,1,1} - 2 s_{2,2} - 10 s_{3,1} - 3 s_{4} $ 
& $ 17 s_{1,1,1,1,1} - s_{2,1,1,1} + 8 s_{2,2,1} - 30 s_{3,1,1} + 7 s_{3,2} + 14 s_{5} $ 
& $ 20 s_{1,1,1,1,1,1} - 20 s_{2,1,1,1,1} + 4 s_{2,2,1,1} - 8 s_{2,2,2} - 85 s_{3,1,1,1} + 37 s_{3,3} - 9 s_{4,1,1} + 25 s_{4,2} + 58 s_{5,1} + 16 s_{6} $ 
\\  
\hline
 13 
& $ s_{} $ 
& $ s_{1} $ 
& $ 4 s_{2} $ 
& $ s_{1,1,1} + 11 s_{2,1} + 4 s_{3} $ 
& $ 9 s_{1,1,1,1} + 26 s_{2,1,1} - 6 s_{2,2} - 9 s_{3,1} - 27 s_{4} $ 
& $ 41 s_{1,1,1,1,1} + 52 s_{2,1,1,1} + 26 s_{2,2,1} - 35 s_{3,1,1} - 24 s_{3,2} - 68 s_{4,1} - 21 s_{5} $ 
& $ 122 s_{1,1,1,1,1,1} + 67 s_{2,1,1,1,1} + 17 s_{2,2,1,1} + 116 s_{2,2,2} - 258 s_{3,1,1,1} - 70 s_{3,2,1} + 66 s_{3,3} - 240 s_{4,1,1} + 124 s_{4,2} + 89 s_{5,1} + 119 s_{6} $ 
 \\  
\hline
 14 
& $ -2 s_{} $ 
& $ -6 s_{1} $ 
& $ -9 s_{1,1} - 2 s_{2} $ 
& $ -15 s_{1,1,1} + 8 s_{2,1} + 31 s_{3} $ 
& $ -6 s_{1,1,1,1} + 61 s_{2,1,1} - 19 s_{2,2} + 67 s_{3,1} + 7 s_{4} $ 
& $ 22 s_{1,1,1,1,1} + 222 s_{2,1,1,1} - 19 s_{2,2,1} + 150 s_{3,1,1} - 215 s_{3,2} - 208 s_{4,1} - 171 s_{5} $ 
& $ 158 s_{1,1,1,1,1,1} + 615 s_{2,1,1,1,1} + 164 s_{2,2,1,1} + 410 s_{2,2,2} + 121 s_{3,1,1,1} - 78 s_{3,2,1} - 266 s_{3,3} - 665 s_{4,1,1} + 33 s_{4,2} - 239 s_{5,1} + 112 s_{6} $ 
 \\  
\hline
 15 
& $ 2 s_{} $ 
& $ s_{1} $ 
& $ -12 s_{1,1} - 19 s_{2} $ 
& $ -72 s_{1,1,1} - 13 s_{2,1} + 35 s_{3} $ 
& $ -183 s_{1,1,1,1} - 63 s_{2,1,1} + 69 s_{2,2} + 293 s_{3,1} + 164 s_{4} $ 
& $ -335 s_{1,1,1,1,1} + 178 s_{2,1,1,1} - 108 s_{2,2,1} + 967 s_{3,1,1} - 135 s_{3,2} + 75 s_{4,1} - 379 s_{5} $ 
& $ -400 s_{1,1,1,1,1,1} + 854 s_{2,1,1,1,1} + 178 s_{2,2,1,1} + 2 s_{2,2,2} + 2648 s_{3,1,1,1} - 355 s_{3,2,1} - 1539 s_{3,3} + 55 s_{4,1,1} - 1904 s_{4,2} - 2361 s_{5,1} - 997 s_{6} $ 
 \\  
\hline
 16 
& $ 0 $ 
& $ 4 s_{1} $ 
& $ -2 s_{1,1} - 66 s_{2} $ 
& $ -32 s_{1,1,1} - 210 s_{2,1} - 123 s_{3} $ 
& $ -299 s_{1,1,1,1} - 651 s_{2,1,1} + 181 s_{2,2} + 190 s_{3,1} + 582 s_{4} $ 
& $ -1007 s_{1,1,1,1,1} - 1211 s_{2,1,1,1} - 298 s_{2,2,1} + 1233 s_{3,1,1} + 1420 s_{3,2} + 2348 s_{4,1} + 878 s_{5} $ 
& $ -2743 s_{1,1,1,1,1,1} - 2033 s_{2,1,1,1,1} - 363 s_{2,2,1,1} - 3884 s_{2,2,2} + 7113 s_{3,1,1,1} + 851 s_{3,2,1} - 2133 s_{3,3} + 6530 s_{4,1,1} - 4981 s_{4,2} - 3563 s_{5,1} - 3703 s_{6} $ 
 \\  
\hline
 17 
& $ 8 s_{} $ 
& $ 56 s_{1} $ 
& $ 119 s_{1,1} + 45 s_{2} $ 
& $ 251 s_{1,1,1} - 286 s_{2,1} - 601 s_{3} $ 
& $ 142 s_{1,1,1,1} - 1260 s_{2,1,1} + 119 s_{2,2} - 1529 s_{3,1} - 207 s_{4} $ 
& $ -661 s_{1,1,1,1,1} - 5866 s_{2,1,1,1} - 322 s_{2,2,1} - 4471 s_{3,1,1} + 4972 s_{3,2} + 5278 s_{4,1} + 4593 s_{5} $ 
& $ -4034 s_{1,1,1,1,1,1} - 12900 s_{2,1,1,1,1} - 2330 s_{2,2,1,1} - 7708 s_{2,2,2} - 246 s_{3,1,1,1} + 9701 s_{3,2,1} + 9578 s_{3,3} + 20880 s_{4,1,1} + 7594 s_{4,2} + 11768 s_{5,1} - 525 s_{6} $ 
 \\  
\hline
 18 
& $ -22 s_{} $ 
& $ 8 s_{1} $ 
& $ 242 s_{1,1} + 346 s_{2} $ 
& $ 1098 s_{1,1,1} + 386 s_{2,1} - 567 s_{3} $ 
& $ 2976 s_{1,1,1,1} + 1119 s_{2,1,1} - 1948 s_{2,2} - 6074 s_{3,1} - 4189 s_{4} $ 
& $ 5518 s_{1,1,1,1,1} - 4924 s_{2,1,1,1} + 564 s_{2,2,1} - 19949 s_{3,1,1} + 929 s_{3,2} - 3591 s_{4,1} + 5978 s_{5} $ 
& $ 6480 s_{1,1,1,1,1,1} - 22699 s_{2,1,1,1,1} - 12797 s_{2,2,1,1} + 3223 s_{2,2,2} - 62759 s_{3,1,1,1} + 8152 s_{3,2,1} + 36941 s_{3,3} - 2232 s_{4,1,1} + 56407 s_{4,2} + 60353 s_{5,1} + 28989 s_{6} $ \\
 \hline
\end{tabular}
    }
    \caption{Contribution of the first term of Theorem \ref{thm:ec17} to the equivariant Euler characteristics $\chi_{\ss_n}(\mathrm{gr}_{17,0}H^\bullet_c(\M_{g,n} ))$ for various $(g,n)$.}
    \label{fig:ec table1}
\end{figure}

\begin{figure}
    \centering
    \resizebox{\textwidth}{!}{
\begin{tabular}{|g|M|M|M|M|M|M|M|} \hline \rowcolor{Gray} g,n 
& 0 
& 1 
& 2 
& 3 
& 4 
& 5 
& 6 
\\ 
\hline
 8 
& $ 0 $ 
& $ 0 $ 
& $ 0 $ 
& $ 0 $ 
& $ 0 $ 
& $ 2 s_{1,1,1,1,1} + 4 s_{2,1,1,1} + s_{2,2,1} + s_{3,1,1} + s_{3,2} $ 
& $ -3 s_{2,1,1,1,1} - 9 s_{2,2,1,1} - 8 s_{2,2,2} - 16 s_{3,1,1,1} - 20 s_{3,2,1} - 4 s_{3,3} - 14 s_{4,1,1} - 7 s_{4,2} - 3 s_{5,1} $ 
 \\  
\hline
 9 
& $ 0 $ 
& $ 0 $ 
& $ 0 $ 
& $ -s_{1,1,1} $ 
& $ -2 s_{1,1,1,1} + 6 s_{2,1,1} + 2 s_{2,2} + 6 s_{3,1} $ 
& $ -5 s_{1,1,1,1,1} - 13 s_{2,1,1,1} - 15 s_{2,2,1} - 16 s_{3,1,1} - 21 s_{3,2} - 20 s_{4,1} - 9 s_{5} $ 
& $ 21 s_{1,1,1,1,1,1} + 80 s_{2,1,1,1,1} + 110 s_{2,2,1,1} + 79 s_{2,2,2} + 117 s_{3,1,1,1} + 169 s_{3,2,1} + 45 s_{3,3} + 78 s_{4,1,1} + 95 s_{4,2} + 35 s_{5,1} + 16 s_{6} $ 
 \\  
\hline
 10 
& $ 0 $ 
& $ 0 $ 
& $ -s_{1,1} - 2 s_{2} $ 
& $ -2 s_{1,1,1} - s_{2,1} + 4 s_{3} $ 
& $ -16 s_{1,1,1,1} - 25 s_{2,1,1} - 23 s_{2,2} - 6 s_{3,1} + 4 s_{4} $ 
& $ 3 s_{1,1,1,1,1} + 97 s_{2,1,1,1} + 92 s_{2,2,1} + 197 s_{3,1,1} + 98 s_{3,2} + 105 s_{4,1} + 5 s_{5} $ 
& $ -59 s_{1,1,1,1,1,1} - 158 s_{2,1,1,1,1} - 492 s_{2,2,1,1} - 241 s_{2,2,2} - 379 s_{3,1,1,1} - 890 s_{3,2,1} - 404 s_{3,3} - 589 s_{4,1,1} - 563 s_{4,2} - 401 s_{5,1} - 94 s_{6} $ 
\\  
\hline
 11 
& $ s_{} $ 
& $ 3 s_{1} $ 
& $ 5 s_{1,1} $ 
& $ -18 s_{1,1,1} - 25 s_{2,1} - 16 s_{3} $ 
& $ -34 s_{1,1,1,1} - 5 s_{2,1,1} + 63 s_{2,2} + 87 s_{3,1} + 59 s_{4} $ 
& $ -190 s_{1,1,1,1,1} - 316 s_{2,1,1,1} - 312 s_{2,2,1} - 134 s_{3,1,1} - 182 s_{3,2} - 50 s_{4,1} - 44 s_{5} $ 
& $ -155 s_{1,1,1,1,1,1} + 325 s_{2,1,1,1,1} + 834 s_{2,2,1,1} + 355 s_{2,2,2} + 1794 s_{3,1,1,1} + 1920 s_{3,2,1} + 369 s_{3,3} + 1727 s_{4,1,1} + 745 s_{4,2} + 254 s_{5,1} - 166 s_{6} $ 
 \\  
\hline
 12 
& $ -s_{} $ 
& $ 2 s_{1} $ 
& $ 15 s_{1,1} - 10 s_{2} $ 
& $ 49 s_{1,1,1} - 52 s_{2,1} - 75 s_{3} $ 
& $ 39 s_{1,1,1,1} - 214 s_{2,1,1} + 27 s_{2,2} - 189 s_{3,1} + 64 s_{4} $ 
& $ -185 s_{1,1,1,1,1} - 849 s_{2,1,1,1} - 404 s_{2,2,1} - 673 s_{3,1,1} + 274 s_{3,2} + 517 s_{4,1} + 513 s_{5} $ 
& $ -502 s_{1,1,1,1,1,1} - 430 s_{2,1,1,1,1} + 1532 s_{2,2,1,1} - 417 s_{2,2,2} + 2696 s_{3,1,1,1} + 4261 s_{3,2,1} + 1746 s_{3,3} + 4836 s_{4,1,1} + 2043 s_{4,2} + 1989 s_{5,1} - 210 s_{6} $ 
 \\  
\hline
 13 
& $ 0 $ 
& $ 21 s_{1} $ 
& $ 67 s_{1,1} + 77 s_{2} $ 
& $ 281 s_{1,1,1} + 216 s_{2,1} - 118 s_{3} $ 
& $ 260 s_{1,1,1,1} - 535 s_{2,1,1} - 440 s_{2,2} - 1666 s_{3,1} - 881 s_{4} $ 
& $ 1256 s_{1,1,1,1,1} + 1189 s_{2,1,1,1} + 3723 s_{2,2,1} + 643 s_{3,1,1} + 4378 s_{3,2} + 2706 s_{4,1} + 1991 s_{5} $ 
& $ -1246 s_{1,1,1,1,1,1} - 12768 s_{2,1,1,1,1} - 17207 s_{2,2,1,1} - 10299 s_{2,2,2} - 25562 s_{3,1,1,1} - 28781 s_{3,2,1} - 3794 s_{3,3} - 17334 s_{4,1,1} - 11853 s_{4,2} - 2448 s_{5,1} + 858 s_{6} $ 
 \\  
\hline
 14 
& $ -18 s_{} $ 
& $ -68 s_{1} $ 
& $ s_{1,1} + 230 s_{2} $ 
& $ 262 s_{1,1,1} + 413 s_{2,1} + 320 s_{3} $ 
& $ 2093 s_{1,1,1,1} + 3472 s_{2,1,1} - 285 s_{2,2} - 263 s_{3,1} - 1716 s_{4} $ 
& $ 3186 s_{1,1,1,1,1} - 439 s_{2,1,1,1} - 3825 s_{2,2,1} - 12764 s_{3,1,1} - 9856 s_{3,2} - 12117 s_{4,1} - 2934 s_{5} $ 
& $ 9627 s_{1,1,1,1,1,1} + 12953 s_{2,1,1,1,1} + 18312 s_{2,2,1,1} + 20329 s_{2,2,2} - 8129 s_{3,1,1,1} + 31460 s_{3,2,1} + 20906 s_{3,3} + 748 s_{4,1,1} + 39137 s_{4,2} + 28585 s_{5,1} + 16320 s_{6} $ 
 \\  
\hline
 15 
& $ 3 s_{} $ 
& $ -154 s_{1} $ 
& $ -399 s_{1,1} + 115 s_{2} $ 
& $ -1170 s_{1,1,1} + 1585 s_{2,1} + 2427 s_{3} $ 
& $ 1113 s_{1,1,1,1} + 9825 s_{2,1,1} + 2144 s_{2,2} + 8997 s_{3,1} + 277 s_{4} $ 
& $ 400 s_{1,1,1,1,1} + 3360 s_{2,1,1,1} - 20535 s_{2,2,1} - 18142 s_{3,1,1} - 45101 s_{3,2} - 45588 s_{4,1} - 24448 s_{5} $ 
& $ 32291 s_{1,1,1,1,1,1} + 113929 s_{2,1,1,1,1} + 150606 s_{2,2,1,1} + 118739 s_{2,2,2} + 137966 s_{3,1,1,1} + 234417 s_{3,2,1} + 63063 s_{3,3} + 85621 s_{4,1,1} + 143486 s_{4,2} + 57037 s_{5,1} + 27655 s_{6} $ 
 \\  
\hline
 16 
& $ 26 s_{} $ 
& $ -315 s_{1} $ 
& $ -1842 s_{1,1} - 2196 s_{2} $ 
& $ -3467 s_{1,1,1} + 1192 s_{2,1} + 5174 s_{3} $ 
& $ -14640 s_{1,1,1,1} - 16387 s_{2,1,1} - 5349 s_{2,2} + 10378 s_{3,1} + 11945 s_{4} $ 
& $ 5055 s_{1,1,1,1,1} + 108419 s_{2,1,1,1} + 89528 s_{2,2,1} + 190609 s_{3,1,1} + 74642 s_{3,2} + 75807 s_{4,1} - 11394 s_{5} $ 
& $ -54427 s_{1,1,1,1,1,1} - 151826 s_{2,1,1,1,1} - 427124 s_{2,2,1,1} - 254844 s_{2,2,2} - 347621 s_{3,1,1,1} - 926374 s_{3,2,1} - 428737 s_{3,3} - 607974 s_{4,1,1} - 710097 s_{4,2} - 512889 s_{5,1} - 155073 s_{6} $ 
 \\  
\hline
 17 
& $ 252 s_{} $ 
& $ 865 s_{1} $ 
& $ -154 s_{1,1} - 4000 s_{2} $ 
& $ -6619 s_{1,1,1} - 14154 s_{2,1} - 9302 s_{3} $ 
& $ -33963 s_{1,1,1,1} - 52219 s_{2,1,1} + 14136 s_{2,2} + 22689 s_{3,1} + 41601 s_{4} $ 
& $ -52790 s_{1,1,1,1,1} + 23275 s_{2,1,1,1} + 109714 s_{2,2,1} + 281787 s_{3,1,1} + 247195 s_{3,2} + 291020 s_{4,1} + 84810 s_{5} $ 
& $ -231884 s_{1,1,1,1,1,1} - 503062 s_{2,1,1,1,1} - 857241 s_{2,2,1,1} - 740662 s_{2,2,2} - 398802 s_{3,1,1,1} - 1677250 s_{3,2,1} - 791468 s_{3,3} - 676014 s_{4,1,1} - 1498807 s_{4,2} - 996016 s_{5,1} - 446949 s_{6} $ 
 \\  
\hline
 18 
& $ 35 s_{} $ 
& $ 2229 s_{1} $ 
& $ 7826 s_{1,1} + 1696 s_{2} $ 
& $ -1832 s_{1,1,1} - 62681 s_{2,1} - 58260 s_{3} $ 
& $ 50121 s_{1,1,1,1} + 76108 s_{2,1,1} + 118328 s_{2,2} + 97675 s_{3,1} + 72055 s_{4} $ 
& $ -329090 s_{1,1,1,1,1} - 1371608 s_{2,1,1,1} - 1285571 s_{2,2,1} - 1663736 s_{3,1,1} - 853232 s_{3,2} - 515358 s_{4,1} + 105597 s_{5} $ 
& $ 445297 s_{1,1,1,1,1,1} + 3353804 s_{2,1,1,1,1} + 7326523 s_{2,2,1,1} + 3618568 s_{2,2,2} + 8859122 s_{3,1,1,1} + 14402303 s_{3,2,1} + 4800221 s_{3,3} + 10298183 s_{4,1,1} + 8331411 s_{4,2} + 5152956 s_{5,1} + 805797 s_{6} $ 
 \\  
 \hline
\end{tabular}
    }
    \caption{Contribution of the second term of Theorem \ref{thm:ec17} to the equivariant Euler characteristics $\chi_{\ss_n}(\mathrm{gr}_{17,0}H_c^\bullet(\M_{g,n} ))$ for various $(g,n)$.}
    \label{fig:ec table2}
\end{figure}

Furthermore, Proposition \ref{prop:lower bound} implies:
\begin{cor}
We have that $\gr_{17,0}H_c^{\bullet}(\M_{g,n})=0$ as long as 
$3g+2n+\min\{1,g-2\}<34$. If $3g+2n<37$, then we have 
\[
\gr_{17,0}H_c^{k}(\M_{g,n})\cong H^{k-17}(G_{2^7}(g-2,n)).
\]
\end{cor}

Finally, from Theorem \ref{thm:n0 cohom}, Proposition \ref{prop:n0 tilde 1k}, and a numerical computation using the available data on $W_0H_c^\bullet(\M_{g,n})$ from the literature we obtain:
\begin{cor} \label{lowg17}
    The Hodge weight $(17,0)$ cohomology of $\M_g$ for $11 \leq g \leq 13$ is given by 
    \[
    \begin{aligned}
        \gr_{17,0}H_c^{k}(\M_{11}) &=
        \begin{cases}
            \cc &\text{if $k=30$} \\
            0 & \text{otherwise}
        \end{cases}
        \\
        \gr_{17,0}H_c^{k}(\M_{12}) &=
        \begin{cases}
            \cc &\text{if $k=31$} \\
            0 & \text{otherwise}
        \end{cases}
        \end{aligned} \quad \quad \quad \quad 
        \begin{aligned}
        \gr_{17,0}H_c^{k}(\M_{13}) &=
        \begin{cases}
            \cc &\text{if $k=34$} \\
            \cc^7 &\text{if $k=35$} \\
            \cc^7 &\text{if $k=36$} \\
            0 & \text{otherwise.}
        \end{cases}
    \end{aligned}
    \]
\end{cor}
\noindent For $g=13$ the contribution in degree $34$ comes from $G_{\tilde 1^{17}}$. All other displayed contributions come from $G_{2^7}$. For example, the generator in genus $g=11$ corresponds to the graph 
\[
\begin{tikzpicture}
\node[int] (v1) at (0,.4) {};
\node (v2) at (-.4,-.4) {$\omega$};
\node (v3) at (0,-.4) {$\omega$};
\node (v4) at (0.4,-.4) {$\omega$};
\draw (v1) edge (v2) edge (v3) edge (v4);
\end{tikzpicture}
\begin{tikzpicture}
\node[int] (v1) at (0,.4) {};
\node (v2) at (-.4,-.4) {$\omega$};
\node (v3) at (0,-.4) {$\omega$};
\node (v4) at (0.4,-.4) {$\omega$};
\draw (v1) edge (v2) edge (v3) edge (v4);
\end{tikzpicture}
\begin{tikzpicture}
\node[int] (v1) at (0,.4) {};
\node (v2) at (-.4,-.4) {$\omega$};
\node (v3) at (0,-.4) {$\omega$};
\node (v4) at (0.4,-.4) {$\omega$};
\draw (v1) edge (v2) edge (v3) edge (v4);
\end{tikzpicture}
\begin{tikzpicture}
\node[int] (v1) at (0,.4) {};
\node (v2) at (-.4,-.4) {$\omega$};
\node (v3) at (0,-.4) {$\omega$};
\node (v4) at (0.4,-.4) {$\omega$};
\draw (v1) edge (v2) edge (v3) edge (v4);
\end{tikzpicture}
  \begin{tikzpicture}
    \node (v) at (0,0) {$\omega$};
    \node (w) at (1,0) {$\omega$};
    \draw (v) edge (w);
  \end{tikzpicture},
\]
with the decoration in $V_{2^7}$ suppressed from the notation.

\subsection{Hodge weight (19,0) cohomology of $\M_{g,n}$}
Conditional on the conjectured vanishing $H^{19,0}(\Mb_{3,15})=0$ we have 
\[
\gr_{19,0}H_c^{k}(\M_{g,n})
\cong 
H^{k-19}(G_{\tilde 1^{19}}(g-1,n))
\oplus H^{k-19}(G_{2^51^6}(g-2,n)).
\]
We hence obtain conditional results on $\gr_{19,0}H_c^{k}(\M_{g,n})$ by our general results on the graph cohomology, analogous to those in the weight (17,0) case above.

Corollaries \ref{cor:EC1} and \ref{cor:EC2} immediately yield:
\begin{thm} \label{thm:ec19}
    If $H^{19,0}(\Mb_{3,15})=0$ then the $\ss_n$-equivariant Euler characteristic of the Hodge weight $(19,0)$ compactly supported cohomology of $\M_{g,n}$ satisfies 
    \begin{multline*}
\sum_{g,n\geq 0} u^{g+n} \chi_{\ss_n}(\mathrm{gr}_{19,0}H_c^\bullet(\M_{g,n}) )
  \\=
\   -u T_{\leq 18}\bigg(
  \prod_{\ell\geq 1} 
  \frac { 
    U_\ell(\frac 1 \ell \sum_{d\mid \ell} \mu(\ell/d) (-p_d  +1-w^d ), u )
  }
  { 
    U_\ell(\frac 1 \ell \sum_{d\mid \ell} \mu(\ell/d) (-p_d), u )
  }-1\bigg)
  \\
  \  -u^2(T_{w_1^{11}w_2^5} -T_{w_1^{12}w_2^{4}})\bigg(
  \prod_{\ell\geq 1} 
  \frac { 
    U_\ell(\frac 1 \ell \sum_{d\mid \ell} \mu(\ell/d) (-p_d  +1-w_1^d-w_2^d ), u )
  }
  { 
    U_\ell(\frac 1 \ell \sum_{d\mid \ell} \mu(\ell/d) (-p_d), u )
  }-1\bigg) \, .
\end{multline*}
\end{thm}

From Proposition \ref{prop:lower bound} we obtain:
\begin{cor}
If $H^{19,0}(\Mb_{3,15})=0$ then we have that $\gr_{19,0}H_c^{\bullet}(\M_{g,n})=0$ as long as 
$3g+2n+\min\{1,g-2\}<38$. If $3g+2n<41$, then we have 
\[
\gr_{19,0}H_c^{k}(\M_{g,n})\cong H^{k-19}(G_{2^51^6}(g-2,n)).
\]
\end{cor}
Finally, from Theorem \ref{thm:n0 cohom}, Proposition \ref{prop:n0 tilde 1k},  and a numerical computation using the available data on $W_0H_c^\bullet(\M_{g,n})$ from the literature we obtain:
\begin{cor} \label{lowg19}
    If $H^{19,0}(\Mb_{3,15})=0$ then the Hodge weight $(19,0)$ cohomology of $\M_g$ for $g=13,14$ is given by 
    \[
    \begin{aligned}
        \gr_{19,0}H_c^{k}(\M_{13}) &=
        \begin{cases}
            \cc^3 &\text{if $k=35$} \\
            \cc &\text{if $k=36$} \\
            0 & \text{otherwise}
        \end{cases}
        \quad \quad \quad \quad 
        \gr_{19,0}H_c^{k}(\M_{14}) &=
        \begin{cases}
            \cc^3 &\text{if $k=36$} \\
            \cc &\text{if $k=37$} \\
            \cc^9 &\text{if $k=38$} \\
            \cc^{19} &\text{if $k=39$} \\
            0 & \text{otherwise.}
        \end{cases}
    \end{aligned}
    \]
\end{cor}

\noindent For $g=14$ one class in degree $38$ originates from $G_{\tilde 1^{19}}$, with representative
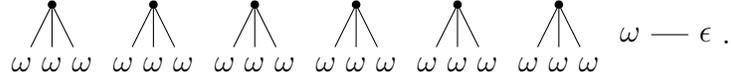
\[
\begin{tikzpicture}
\node[int] (v1) at (0,.4) {};
\node (v2) at (-.4,-.4) {$\omega$};
\node (v3) at (0,-.4) {$\omega$};
\node (v4) at (0.4,-.4) {$\omega$};
\draw (v1) edge (v2) edge (v3) edge (v4);
\end{tikzpicture}
\begin{tikzpicture}
\node[int] (v1) at (0,.4) {};
\node (v2) at (-.4,-.4) {$\omega$};
\node (v3) at (0,-.4) {$\omega$};
\node (v4) at (0.4,-.4) {$\omega$};
\draw (v1) edge (v2) edge (v3) edge (v4);
\end{tikzpicture}
\begin{tikzpicture}
\node[int] (v1) at (0,.4) {};
\node (v2) at (-.4,-.4) {$\omega$};
\node (v3) at (0,-.4) {$\omega$};
\node (v4) at (0.4,-.4) {$\omega$};
\draw (v1) edge (v2) edge (v3) edge (v4);
\end{tikzpicture}
\begin{tikzpicture}
\node[int] (v1) at (0,.4) {};
\node (v2) at (-.4,-.4) {$\omega$};
\node (v3) at (0,-.4) {$\omega$};
\node (v4) at (0.4,-.4) {$\omega$};
\draw (v1) edge (v2) edge (v3) edge (v4);
\end{tikzpicture}
\begin{tikzpicture}
\node[int] (v1) at (0,.4) {};
\node (v2) at (-.4,-.4) {$\omega$};
\node (v3) at (0,-.4) {$\omega$};
\node (v4) at (0.4,-.4) {$\omega$};
\draw (v1) edge (v2) edge (v3) edge (v4);
\end{tikzpicture}
\begin{tikzpicture}
\node[int] (v1) at (0,.4) {};
\node (v2) at (-.4,-.4) {$\omega$};
\node (v3) at (0,-.4) {$\omega$};
\node (v4) at (0.4,-.4) {$\omega$};
\draw (v1) edge (v2) edge (v3) edge (v4);
\end{tikzpicture}
  \begin{tikzpicture}
    \node (v) at (0,0) {$\omega$};
    \node (w) at (1,0) {$\epsilon$};
    \draw (v) edge (w);
  \end{tikzpicture}.
  \]
All other displayed contributions come from $G_{2^51^6}$.
The classes in genus 13 are represented in the complex $H_{2^51^6,13}$ of the proof of Theorem \ref{thm:n0 cohom} by the following graphs:
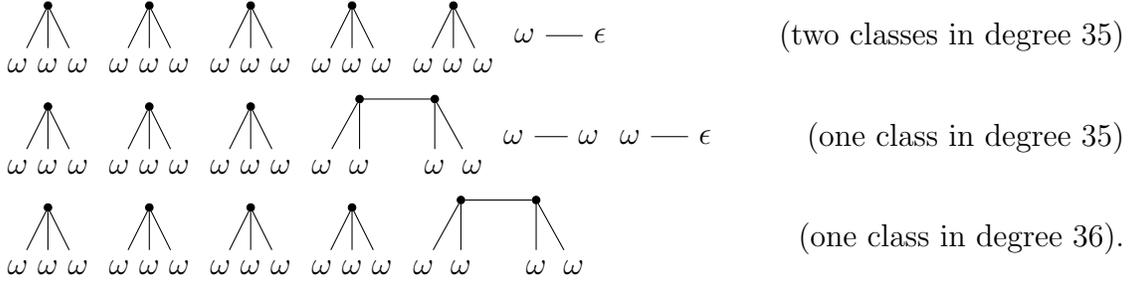
\begin{align*}
&\begin{tikzpicture}
\node[int] (v1) at (0,.4) {};
\node (v2) at (-.4,-.4) {$\omega$};
\node (v3) at (0,-.4) {$\omega$};
\node (v4) at (0.4,-.4) {$\omega$};
\draw (v1) edge (v2) edge (v3) edge (v4);
\end{tikzpicture}
\begin{tikzpicture}
\node[int] (v1) at (0,.4) {};
\node (v2) at (-.4,-.4) {$\omega$};
\node (v3) at (0,-.4) {$\omega$};
\node (v4) at (0.4,-.4) {$\omega$};
\draw (v1) edge (v2) edge (v3) edge (v4);
\end{tikzpicture}
\begin{tikzpicture}
\node[int] (v1) at (0,.4) {};
\node (v2) at (-.4,-.4) {$\omega$};
\node (v3) at (0,-.4) {$\omega$};
\node (v4) at (0.4,-.4) {$\omega$};
\draw (v1) edge (v2) edge (v3) edge (v4);
\end{tikzpicture}
\begin{tikzpicture}
\node[int] (v1) at (0,.4) {};
\node (v2) at (-.4,-.4) {$\omega$};
\node (v3) at (0,-.4) {$\omega$};
\node (v4) at (0.4,-.4) {$\omega$};
\draw (v1) edge (v2) edge (v3) edge (v4);
\end{tikzpicture}
\begin{tikzpicture}
\node[int] (v1) at (0,.4) {};
\node (v2) at (-.4,-.4) {$\omega$};
\node (v3) at (0,-.4) {$\omega$};
\node (v4) at (0.4,-.4) {$\omega$};
\draw (v1) edge (v2) edge (v3) edge (v4);
\end{tikzpicture}
  \begin{tikzpicture}
    \node (v) at (0,0) {$\omega$};
    \node (w) at (1,0) {$\epsilon$};
    \draw (v) edge (w);
  \end{tikzpicture}
  & \text{(two classes in degree 35)} \\
  &\begin{tikzpicture}
\node[int] (v1) at (0,.4) {};
\node (v2) at (-.4,-.4) {$\omega$};
\node (v3) at (0,-.4) {$\omega$};
\node (v4) at (0.4,-.4) {$\omega$};
\draw (v1) edge (v2) edge (v3) edge (v4);
\end{tikzpicture}
\begin{tikzpicture}
\node[int] (v1) at (0,.4) {};
\node (v2) at (-.4,-.4) {$\omega$};
\node (v3) at (0,-.4) {$\omega$};
\node (v4) at (0.4,-.4) {$\omega$};
\draw (v1) edge (v2) edge (v3) edge (v4);
\end{tikzpicture}
\begin{tikzpicture}
\node[int] (v1) at (0,.4) {};
\node (v2) at (-.4,-.4) {$\omega$};
\node (v3) at (0,-.4) {$\omega$};
\node (v4) at (0.4,-.4) {$\omega$};
\draw (v1) edge (v2) edge (v3) edge (v4);
\end{tikzpicture}
\begin{tikzpicture}
  \node[int] (i) at (0,.5) {};
  \node[int] (j) at (1,.5) {};
  \node (v1) at (-.5,-.4) {$\omega$};
  \node (v2) at (0,-.4) {$\omega$};
  \node (v3) at (1,-.4) {$\omega$};
  \node (v4) at (1.5,-.4) {$\omega$};
\draw (i) edge (v1) edge (v2) edge (j) (j) edge (v3) edge (v4);
\end{tikzpicture}
  \begin{tikzpicture}
    \node (v) at (0,0) {$\omega$};
    \node (w) at (1,0) {$\omega$};
    \draw (v) edge (w);
  \end{tikzpicture}
    \begin{tikzpicture}
    \node (v) at (0,0) {$\omega$};
    \node (w) at (1,0) {$\epsilon$};
    \draw (v) edge (w);
  \end{tikzpicture}
& \text{(one class in degree 35)} \\
&\begin{tikzpicture}
\node[int] (v1) at (0,.4) {};
\node (v2) at (-.4,-.4) {$\omega$};
\node (v3) at (0,-.4) {$\omega$};
\node (v4) at (0.4,-.4) {$\omega$};
\draw (v1) edge (v2) edge (v3) edge (v4);
\end{tikzpicture}
\begin{tikzpicture}
\node[int] (v1) at (0,.4) {};
\node (v2) at (-.4,-.4) {$\omega$};
\node (v3) at (0,-.4) {$\omega$};
\node (v4) at (0.4,-.4) {$\omega$};
\draw (v1) edge (v2) edge (v3) edge (v4);
\end{tikzpicture}
\begin{tikzpicture}
\node[int] (v1) at (0,.4) {};
\node (v2) at (-.4,-.4) {$\omega$};
\node (v3) at (0,-.4) {$\omega$};
\node (v4) at (0.4,-.4) {$\omega$};
\draw (v1) edge (v2) edge (v3) edge (v4);
\end{tikzpicture}
\begin{tikzpicture}
\node[int] (v1) at (0,.4) {};
\node (v2) at (-.4,-.4) {$\omega$};
\node (v3) at (0,-.4) {$\omega$};
\node (v4) at (0.4,-.4) {$\omega$};
\draw (v1) edge (v2) edge (v3) edge (v4);
\end{tikzpicture}
\begin{tikzpicture}
  \node[int] (i) at (0,.5) {};
  \node[int] (j) at (1,.5) {};
  \node (v1) at (-.5,-.4) {$\omega$};
  \node (v2) at (0,-.4) {$\omega$};
  \node (v3) at (1,-.4) {$\omega$};
  \node (v4) at (1.5,-.4) {$\omega$};
\draw (i) edge (v1) edge (v2) edge (j) (j) edge (v3) edge (v4);
\end{tikzpicture}
& \text{(one class in degree 36).} \\
\end{align*}

Our methods also give an unconditional proof of the following weaker statement.
\begin{cor}
    There is an injective map 
    \[
    H^{3g+n-16}(G_{2^51^6}) \to \gr_{19,0} H_c^{3g-3+n}(\M_{g,n}).
    \]
    In particular, we have that $\mathrm{dim}\ \gr_{19,0} H_c^{36}(\M_{13})\geq 1$ and $\mathrm{dim}\  \gr_{19,0} H_c^{39}(\M_{14})\geq 19$.
\end{cor}
\begin{proof}
    Recall that $\gr_{19,0} H_c^{k}(\M_{g,n})=H^k(\GK^{19,0}_{g,n})$ is computed by the Getzler-Kapranov complex, i.e., the Feynman transform of $H(\Mb)$. Our graph complex $G_{2^51^6}$ is quasi-isomorphic to the part of $\GK^{19,0}_{g,n}$ spanned by graphs for which the special vertex has genus 2. To produce the injection claimed in the Corollary, it is sufficient to argue that the classes from $H^{3g+n-16}(G_{2^51^6})$ cannot be in the image of the differential acting on a graph in $\GK^{19,0}_{g,n}$ with special vertex of genus 3. However, as the proof of Lemma \ref{lem:upper deg bound} shows, all graphs in the top degree (i.e. degree $3g+n-16$-)part of $G_{2^51^6}$ are such that the special vertex has valence 16.
    Hence in order for these to be in the image of the differential of a graph (or linear combination thereof) $\Gamma$ with special vertex of genus 3, the special vertex in $\Gamma$ must have valence 14. It hence has decoration in $H^{19,0}(\Mb_{3,14})$, which vanishes, so that the result follows.
\end{proof}

\section{Odd cohomology when \texorpdfstring{$g\geq 16$}{g16}} \label{sec:odd}

In this section, we prove Theorem \ref{thm:odd}. We first reduce to the case $g=16$ and $n=0$.

\begin{lem}\label{lem:increaseg}
If $H^k(\Mb_g)\neq 0$ then $H^{k+2(h-g)}(\Mb_{h,n})\neq 0$ for all $h\geq g\geq 2$.

\end{lem}
\begin{proof}
   
   Because pullback by the morphism forgetting markings is injective, it suffices to prove the case $n=0$. The proof is by induction on $g$. Let $\pi\colon \Mb_{g,1}\rightarrow \Mb_{g}$ be the forgetful map, and define $A$ to be the stable graph of genus $g+1$ with one vertex of genus $g$ and one vertex of genus $1$, connected by one edge. Let $\zeta\colon \Mb_{A} =\Mb_{g,1}\times \Mb_{1,1}\rightarrow \Mb_{g+1}$ be the gluing map. Let $\gamma\in H^{k}(\Mb_{g})$ be a nonzero class. Then $\zeta_*(\pi^*\gamma\otimes 1)\in H^{k+2}(\Mb_{g+1})$, and we claim that this class is not zero. To prove this, we check that the pullback along $\zeta$ is nonzero, for which we use the excess intersection formula and generic $(A, B)$ structures as in \cite[Appendix A]{GraberPandharipande}. %
   
    Consider the fiber product diagram
    \[
    \begin{tikzcd}
{\coprod_{(\Gamma,f,g)\in \mathfrak{G}_{A,A}}\Mb_{\Gamma}} \arrow[d,"\xi_{g}"] \arrow[r,"\xi_{f}"] & \Mb_{A} \arrow[d, "\zeta"] \\
\Mb_{A} \arrow[r, "\zeta_A"]                                                       & \Mb_{g+1}.                
\end{tikzcd}
    \]
    The index set $\mathfrak{G}_{A,A}$ is the set of isomorphism classes of generic $(A,A)$ structures $(\Gamma,f,g)$. A generic $(A,A)$ structure is a stable graph $\Gamma$ together with two morphisms $f,g\colon \Gamma\rightarrow A$ such that every half edge of $\Gamma$ corresponds to a half edge of $A$ under the maps $f$ and $g$. An isomorphism of generic $(A,A)$ structures $(\Gamma,f,g)$ and $(\Gamma',f',g')$ is an isomorphism of stable graphs $\tau \colon \Gamma\rightarrow \Gamma'$ such that $f=f'\circ \tau$ and $g=g'\circ \tau$. The morphisms $\xi_g$ and $\xi_f$ are partial gluing maps determined by the morphisms $f,g \colon \Gamma\rightarrow A$ of stable graphs. 
    
    By excess intersection theory, as in \cite[Equation (11)]{GraberPandharipande}, we have
    \begin{equation}\label{excess}
    \zeta^*\zeta_*(\pi^*\gamma\otimes 1) = \sum_{(\Gamma,f,g)\in \mathfrak{G}_{A,A}} \xi_{f*}\bigg(\xi_g^*(\pi^*\gamma\otimes 1)\cdot \prod_{(h,h')\in \im(f)\cap \im(g)} (-\psi_h - \psi_{h'}) \bigg ).
    \end{equation}
    Here, the sum runs over edges of $\Gamma$ formed from the half edge pair $(h,h')$ that come from edges of $A$ under both $f$ and $g$.

    There is a unique generic $(A,A)$ structure on $A$ itself, given by two copies of the identity morphism. There is only one other graph with a generic $(A,A)$ structure, the graph $\Gamma$ with a central vertex of genus $g-1$ attached two two outer vertices, each of genus $1$. The $(A,A)$ structure is given by the pair of edge contraction morphisms $f,g \colon \Gamma\rightarrow A$. This is the unique $(A,A)$ structure up to $\Gamma$ isomorphisms on $\Gamma$. Therefore,
    \[
    \zeta^*\zeta_*(\pi^*\gamma\otimes 1) = (-\psi\cdot \pi^*\gamma \otimes 1 - \pi^*\gamma\otimes \psi) +\xi_{f*} \xi_g^*(\pi^*\gamma\otimes 1).
    \]

We have $\Mb_{\Gamma} = \Mb_{1,1}\times \Mb_{g-1,2}\times \Mb_{1,1}$, the gluing map $\xi_g = (\alpha \times \id)$, the gluing map $\xi_f = (\id \times \beta)$, where $\alpha\colon\Mb_{1,1}\times \Mb_{g-1,2}\rightarrow \Mb_{g,1}$ and $\beta \colon \Mb_{g-1,2}\times \Mb_{1,1}\rightarrow \Mb_{g,1}$ are the gluing maps along different markings. Therefore,
\[
\xi_{f*} \xi_g^*(\pi^*\gamma\otimes 1) = (\id \times \beta)_* (\alpha \times \id)^*(\pi^*\gamma\otimes 1) = (\id\times \beta)_*(\alpha^*\pi^*\gamma\otimes 1).
\]
We have the commutative diagram
\[
    \begin{tikzcd}
 \Mb_{1,1}\times \Mb_{g-1,2}\arrow[d,"\id \times \pi'"] \arrow[r,"\alpha"] & \Mb_{g,1} \arrow[d, "\pi"] \\
\Mb_{1,1}\times \Mb_{g-1,1} \arrow[r, "\alpha'"]                                                       & \Mb_{g},                
\end{tikzcd}
\]
and so 
\[
(\id\times \beta)_*(\alpha^*\pi^*\gamma\otimes 1) = (\id\times\beta)_*((\id\times \pi')^*(\alpha'^*\gamma)\otimes 1).
\]
We write $\alpha'^*\gamma = \sum a_i\otimes b_i$. We have
\[
(\id\times\beta)_*((\id\times \pi')^*\alpha'^*\gamma\otimes 1) = (\id\times\beta)_*\left((\id\times \pi')^*\left(\sum a_i\otimes b_i\right)\otimes 1\right) = \sum a_i \otimes \beta_*(\pi'^*b_i\otimes 1).
\]
Combining all the terms and being careful to put all the tensors in correct order, we have
\[
 \zeta^*\zeta_*(\pi^*\gamma\otimes 1) = -1\otimes \psi\cdot \pi^*\gamma - \psi\otimes \pi^*\gamma + \sum a_i \otimes \beta_*(\pi'^*b_i\otimes 1).
\]
We show that this is nonzero by pushing forward by $(\id \times \pi)$. We have
\[
(\id\times \pi)_* \zeta^*\zeta_*(\pi^*\gamma\otimes 1) = -1\otimes (2g-2)\gamma +  \sum a_i \otimes \pi_*\beta_*(\pi'^*b_i\otimes 1).
\]
To compute the last term, we use the diagram
\[
    \begin{tikzcd}
 \Mb_{g-1,2} \times \Mb_{1,1} \arrow[d," \pi'\times \id"] \arrow[r,"\beta"] & \Mb_{g,1} \arrow[d, "\pi"] \\
 \Mb_{g-1,1}\times \Mb_{1,1} \arrow[r, "\beta'"]                                                       & \Mb_{g}.                
\end{tikzcd}
\]
We thus have 
\begin{align*}
    -1\otimes (2g-2)\gamma +  \sum a_i \otimes \pi_*\beta_*(\pi'^*b_i\otimes 1) &= -1\otimes (2g-2)\gamma + \sum a_i \otimes \beta'_*(\pi'\times \id)_*(\pi'^*b_i\otimes 1) \\&= -1\otimes (2g-2)\gamma \\&\neq 0. \qedhere
\end{align*}

\end{proof}
Given $k \leq 7$, let $A_{2+2k,14-2k}$ be the stable graph of genus $2+2k$ with $14-2k$ legs that has a central genus $2$ vertex with $14$ half-edges together with $k$ genus $1$ vertices with $2$ half edges each so that each genus $1$ vertex attached twice to the central genus $2$ vertex. Let \[\xi_{A_{2+2k,14-2k}} \colon \Mb_{A_{2+2k,14-2k}}=\Mb_{2,14}\times (\Mb_{1,2})^{k}\rightarrow \Mb_{2+2k,14-2k}\] be the gluing map. Set $A= A_{16,0}$. Then $\Aut(A)\cong \ss_2\wr\ss_7$. Let $\eta\in H^{17,0}(\Mb_{2,14})^{\Aut(A)}$ be a generator for the one dimensional subspace of invariant forms. 
\begin{prop}\label{prop:g16}
    For $k \leq 7$, we have $\xi_{A_{2+2k,14-2k}*}(\eta\otimes 1\otimes\dots\otimes 1)\neq 0$. In particular, $H^{45}(\Mb_{16})$ is nonzero.
\end{prop}
\begin{proof}
    It suffices to show that $\xi_{A*}(\eta\otimes 1\otimes\dots\otimes 1)\in H^{45}(\Mb_{16})$ is nonzero. To do so, we show $(\mathrm{pr}_{1*}\xi_{A}^*\xi_{A*}(\eta\otimes 1\otimes\dots\otimes 1))\neq 0$. We use the excess intersection formula, as in the proof of Lemma \ref{lem:increaseg}. Consider the fiber product diagram
    \[
    \begin{tikzcd}
{\coprod_{(\Gamma,f,g)\in \mathfrak{G}_{A,A}}\Mb_{\Gamma}} \arrow[d,"\xi_{g}"] \arrow[r,"\xi_{f}"] & \Mb_{A} \arrow[d, "\xi_A"] \\
\Mb_{A} \arrow[r, "\xi_A"]                                                       & \Mb_{16}.                
\end{tikzcd}
    \]
    
    By excess intersection theory, 
    \begin{equation}\label{eq:excess2}
    \xi_{A}^*\xi_{A*}(\eta\otimes 1\otimes\dots\otimes 1) = \sum_{(\Gamma,f,g)\in \mathfrak{G}_{A,A}} \xi_{f*}\bigg (\xi_g^*(\eta\otimes 1\otimes\dots\otimes 1)\cdot \prod_{(h,h')\in \im(f)\cap \im(g)} (-\psi_h - \psi_{h'}) \bigg ).
    \end{equation}
    As in the proof of Lemma \ref{lem:increaseg}, the index set $\mathfrak{G}_{A,A}$ is the set of isomorphism classes of generic $(A,A)$ structures $(\Gamma,f,g)$.
    The product runs over edges of $\Gamma$ formed from the half edge pair $(h,h')$ that come from edges of $A$ under both $f$ and $g$.

    We first consider the generic $(A,A)$ structures $(A,f,g)$. Up to isomorphism, we may assume $f$ is the identity and $g$ is any automorphism of $A$. Therefore, there are $|\Aut(A)|$ %
    generic $(A,A)$ structures on $A$. We have $\xi_g^*(\eta\otimes 1\otimes\dots\otimes 1)=\eta\otimes 1\otimes\dots\otimes 1$ for all $g$ because $\eta$ was chosen to be $\Aut(A)$ invariant. Therefore, the contribution to equation \eqref{eq:excess2} from the generic $(A,A)$ structures  of the form $(A,f,g)$ is
    \[
    |\Aut(A)| (\eta\otimes 1\otimes\dots\otimes 1) \prod_{(h,h')\in E(A)} (-\psi_h-\psi_{h'}).
    \]
    After expanding the product, only the term of the form
    \[
    |\Aut(A)|(\eta\otimes \psi_1\psi_2\otimes \dots \otimes \psi_1\psi_2)
    \]
    with two $\psi$ classes on each of the genus $1$ vertices can survive the pushforward to $\Mb_{2,14}$. After pushing forward to $\Mb_{2,14}$, we obtain a nonzero multiple of $\eta$. 

    Now we consider the contributions from generic $(A,A)$ structures whose underlying stable graph $\Gamma$ is not isomorphic to $A$. First, suppose that the preimage in $\Gamma$ of the genus $2$ vertex of $A$ under $g$ is of cardinality at least $2$. Then $\xi_g^*(\eta\otimes 1\otimes\dots\otimes 1)=0$ because the form $\eta$ pulls back trivially to all boundary divisors of $\Mb_{2,14}$. In particular, all such generic $(A,A)$ structures $(\Gamma,f,g)$ contribute trivially to \eqref{eq:excess2}.

    Finally, we show that there are no other types of generic $(A,A)$ structures $(\Gamma, f,g)$. Suppose that $\Gamma$ has a unique vertex of genus $2$, which is the preimage under $g$ or $f$ of the genus $2$ vertex of $A$, but that $\Gamma$ is not isomorphic to $A$. Then there must be some genus $1$ vertex $w$ of $A$ that is replaced in $\Gamma$ by a stable graph of genus $1$ with two legs and at least one edge. One checks that there are no \emph{generic} $(A,A)$ structures in this situation, as the edge contraction $f:\Gamma\rightarrow A$ will have to contract some of the same edges as $g:\Gamma\rightarrow A$.

\end{proof}
\begin{proof}[Proof of Theorem \ref{thm:odd}]
    The first statement follows immediately from Lemma \ref{lem:increaseg} and Proposition \ref{prop:g16}. The second statement follows from the first because $\Mb_g$ is smooth and proper, so the cohomology is of pure weight.
\end{proof}

\bibliographystyle{amsplain}
\bibliography{refs}
\end{document}